\documentclass[a4paper,reqno]{amsart}
\usepackage{graphicx}
\usepackage{mathptmx}      
\usepackage{enumitem}
\usepackage{amssymb}
\usepackage{mathtools}

\newtheorem{definition}[equation]{Definition}
\newtheorem{lemma}[equation]{Lemma}
\newtheorem{proposition}[equation]{Proposition}
\newtheorem{theorem}[equation]{Theorem}
\newtheorem{corollary}[equation]{Corollary}
\newtheorem{remark}[equation]{Remark}

\numberwithin{equation}{section}

\newcommand{\N}{\mathbb{N}}
\newcommand{\R}{\mathbb{R}}

\DeclareMathOperator{\randreg}{{\mathcal R}_\epsilon\hspace{-0.05cm}} 
\DeclareMathOperator{\randregn}{{\mathcal R}_{\epsilon_{\textit n}}\hspace{-0.05cm}} 
\DeclareMathOperator{\Teta}{{\mathcal T}_{\eta}} 
\DeclareMathOperator{\Tdelta}{{\mathcal T}_{\delta}} 
\DeclareMathOperator{\trn}{{\operatorname{tr}^n}} 
\DeclareMathOperator{\trneta}{{\operatorname{tr}^n_{\eta}}} 
\DeclareMathOperator{\trnetan}{{\operatorname{tr}^n_{\eta_n}}} 
\DeclareMathOperator{\trnetanull}{{\operatorname{tr}^n_{\eta_0}}} 
\DeclareMathOperator{\trnregdeltanull}{{\operatorname{tr}^n_{\randreg\delta(0,\cdot)}}} 
\DeclareMathOperator{\treta}{{\operatorname{tr}_{\eta}}} 
\DeclareMathOperator{\trdelta}{{\operatorname{tr}_{\delta}}} 
\DeclareMathOperator{\trdeltan}{{\operatorname{tr}_{\delta_n}}} 
\DeclareMathOperator{\trtildeetan}{{\operatorname{tr}_{{\tilde\eta}_n}}} 
\DeclareMathOperator{\trtildeeta}{{\operatorname{tr}_{\tilde\eta}}} 
\DeclareMathOperator{\trregdelta}{{\operatorname{tr}_{\randreg\delta}}} 
\DeclareMathOperator{\trregdeltan}{{\operatorname{tr}_{\randreg\delta_n}}} 
\DeclareMathOperator{\trregnetan}{{\operatorname{tr}_{\randregn\eta_n}}} 
\DeclareMathOperator{\trregeta}{{\operatorname{tr}_{\randreg\eta}}} 
\DeclareMathOperator{\tretan}{{\operatorname{tr}_{\eta_n}}} 
\DeclareMathOperator{\feta}{{\mathcal{F}_{\eta}}} 
\DeclareMathOperator{\ftildeeta}{{\mathcal{F}_{\tilde\eta}}} 
\DeclareMathOperator{\ftildeetan}{{\mathcal{F}_{{\tilde\eta}_n}}} 
\DeclareMathOperator{\fetanull}{{\mathcal{F}_{\eta_0}}} 
\DeclareMathOperator{\fetan}{{\mathcal{F}_{\eta_n}}} 
\DeclareMathOperator{\frandregdeltan}{{\mathcal{F}_{\randreg\delta_n}}} 
\DeclareMathOperator{\frandregnetan}{{\mathcal{F}_{\randregn\eta_n}}} 
\DeclareMathOperator{\aussenrandlambda}{\overline{{\mathcal R}}_\lambda\hspace{-0.05cm}} 
\DeclareMathOperator{\Tlambda}{{\mathcal T}_{\aussenrandlambda\delta}} 
\DeclareMathOperator{\Tlambdat}{{\mathcal T}_{\aussenrandlambda\delta(t)}} 
\DeclareMathOperator{\invTlambda}{{\mathcal T}_{\aussenrandlambda\delta(t)}^{-1}} 

\DeclareMathOperator{\intetat}{\int_{\Omega_{\eta(t)}}\hspace{-0.3cm}} 
\DeclareMathOperator{\intetas}{\int_{\Omega_{\eta(s)}}\hspace{-0.3cm}} 
\DeclareMathOperator{\intreps}{\int_{\Omega_{\randreg\delta(t)}}\hspace{-0.7cm}} 
\DeclareMathOperator{\intrepnull}{\int_{\Omega_{\randreg\delta(0)}}\hspace{-0.7cm}} 
\DeclareMathOperator{\intrepsecht}{\int_{\Omega_{\randreg\delta(s)}}\hspace{-0.7cm}} 
\DeclareMathOperator{\intrepsn}{\int_{\Omega_{\randreg\delta_n(t)}}\hspace{-0.8cm}} 
\DeclareMathOperator{\intregta}{\int_{\Omega_{\randreg\eta(t)}}\hspace{-0.7cm}} 
\DeclareMathOperator{\intregntan}{\int_{\Omega_{\randregn\eta_{\textit n}(t)}}\hspace{-1.0cm}} 
\DeclareMathOperator{\intdeltan}{\int_{\Omega_{\delta_n(t)}}\hspace{-0.5cm}} 
\DeclareMathOperator{\intdelta}{\int_{\Omega_{\delta(t)}}\hspace{-0.4cm}} 
\DeclareMathOperator{\intdeltans}{\int_{\Omega_{\delta_n(s)}}\hspace{-0.5cm}} 
\DeclareMathOperator{\intdeltas}{\int_{\Omega_{\delta(s)}}\hspace{-0.4cm}} 
\DeclareMathOperator{\fdeltan}{{\mathcal{F}_{\delta_n}}} 
\DeclareMathOperator{\fdelta}{{\mathcal{F}_{\delta}}} 
\begin{document}

\title[Interaction of a Newtonian Fluid with a Koiter Shell]{Global weak solutions for an Newtonian Fluid interacting with a 
  Koiter Type Shell under natural boundary conditions%
}


\author{Hannes Eberlein} \address{Hannes Eberlein 
  Mathematisches Institut, %
  Ernst-Zermelo-Str.~1, %
  79104 Freiburg, %
  Germany}
\email{hannes.eberlein@mathematik.uni-freiburg.de}
\author{Michael R\r{u}\v{z}i\v{c}ka} \address{Hannes Eberlein 
  Mathematisches Institut, %
  Ernst-Zermelo-Str.~1, %
  79104 Freiburg, %
  Germany}
\email{rose@mathematik.uni-freiburg.de}




\maketitle

\begin{abstract}
  We consider an viscous, incompressible Newtonian fluid flowing
  through a thin elastic structure. The motion of the structure is
  described by the equations of a linearised Koiter shell, whose
  motion is restricted to transverse displacements. The fluid and the
  structure are coupled by the continuity of
  velocities 
  and an equilibrium of surface forces on the interface between fluid
  and structure. On a fixed in- and outflow region we prescribe
  natural boundary conditions. We show that weak solutions exist as
  long as the shell does not 
  self-intersect.
  {\it  2000MSC}: {76D05, 74F10}
\end{abstract}

\section{Introduction}
\label{intro}
Fluid-structure interaction (FSI) problems are common in nature, most
prominently seen in the blood flow through vessels. Blood 
generally exhibits a non-Newtonian behaviour (see \cite{cokelet87}),
while 
the vascular wall consists of many layers all having different
biomechanical properties (see \cite{Quarteroni2000}). If medium or big
blood vessels are considered one can use a simplified model. Firstly, we treat blood as 
a Newtonian
fluid. 
Since only a small part of the
circulatory system is considered we 
introduce artificial in- and outflow
boundary conditions. Secondly, we will model the vascular wall as a linearly
elastic Koiter shell whose motion is restricted to transverse
displacements. 

Due to the apparent regularity incompatibilities between the parabolic 
fluid phase and the hyperbolic dispersive solid phase, the interaction 
between an elastic structure and a fluid is exceedingly difficult, see 
e.\,g.\ \cite{MR2027753,MR2208319,MR2166981,MR2438783,MR3061763} and the 
references therein. For the existence proof of weak solutions, two different strategies 
have been proven successful. The first one is working directly on the 
fluid domain and therefore preserving the structure of the equations. 
Since by the non-cylindrical space-time domain the usual Bochner-space 
theory is not applicable, the key element in this approach is to find an 
appropriate compactness argument. This was accomplished in \cite{phdlengeler}, 
\cite{MR3147436} and \cite{MR3233099}, by generalising methods from 
\cite{MR2166981}, \cite{MR2438783}. There, the existence of a 
global-in-time weak solution to the interaction of an incompressible 
generalised Newtonian fluid completely surrounded with an linearised 
transversal Koiter shell was shown. The second approach consists of 
transforming the fluid equation by an ALE mapping to a reference domain 
and using a semi-discrete, operator splitting Lie scheme. This 
method was used in \cite{MR3017292}, \cite{MR3226949} and
\cite{MR3450736} to show the 
existence of a global-in-time weak solution to the interaction between 
a Newtonian fluid and a (semi-) linear transversal Koiter shell, where 
the flow is driven by a dynamical pressure condition and no other 
external forces apply. Furthermore, in \cite{MR3482692} the Navier slip 
boundary condition was used for the coupling between the fluid 
and an elastic structure in a two-dimensional setting. 

The present paper is based on the first authors Ph.D. thesis
\cite{phdeberlein} and extends the result of \cite{MR3147436} to the
case of an in- and outflow region. We will use the same method as in
\cite{MR3147436} and thus the structure and the arguments of the
present paper are similar to \cite{MR3147436}. Since the in- and
outflow region admits a flow through the domain, various new
difficulties have to be solved.  Moreover, since the reference domain
has no $C^4$-boundary, but only a Lipschitz one, special care has to
be taken to transfer the compactness argument to our situation. In
comparison to \cite{MR3226949} and \cite{MR3450736}, we consider a
different in- and outflow condition and a more general
geometry. Furthermore, we allow external forces acting on the fluid
and the shell.

This paper is organised as follows: In the next Subsections, we
introduce our setting and formally derive an a priori estimate for the
resulting system. In Section \ref{sec:moving_domains} we develop the
mathematical framework for our non-Lipschitz in- and outflow
domains. In particular, in Subsection \ref{subsec:general_trace} we
define a generalised trace operator and show some density results for
functions vanishing on a part of the boundary, and in Subsection
\ref{section:div_fortsetzung} we construct a suitable extension
operator of test functions for the shell equation. In Section
\ref{section:main_result}, we can finally formulate the main result of
this paper, the existence of a global-in-time weak solution. Crucial
for the proof of this result is that one can generalise 
the compactness result from \cite{MR3147436} to our setting, which is
possible due to our preparatory work
(see 
Subsection \ref{subsection:compactness}). The proof of the main
Theorem is then carried out by looking first at a decoupled,
regularised and linearised problem, then a fixed point argument to
restore the coupling, and finally a limiting process to eliminate the
regularisation.
\subsection{Koiter's energy and statement of the problem}
By $\Omega\subset \R^3$ we denote a reference domain with $\partial 
\Omega = \Gamma \cup M$, where $\Gamma$ is assumed to be the fixed 
in- and outflow region and $\partial \Gamma = \partial M \neq \emptyset$. 
Moreover, let $M$ represent the middle surface of the thin elastic shell 
of thickness $2\,\varepsilon_s>0$ in its rest state. The deformation of 
the shell is then given by the displacement $\boldsymbol\eta$ relatively 
to $M$. By $\boldsymbol\nu$ we denote the unit normal on $\partial\Omega$, 
by $g$ and $h$ the first and second fundamental form of $M$ induced by the 
ambient Euclidean space, and by $dA$ the surface measure of $M$ or 
$\partial \Omega$, respectively. As in \cite{MR3147436,MR3226949} we restrict 
the deformations to transverse displacements, i.\,e.\ $\boldsymbol\eta = 
\eta\,\boldsymbol\nu$. Following \cite{MR3147436}, we assume further that 
the elastic shell, clamped on $\partial M$, consists of a homogeneous, 
isotropic material whose linear elastic behaviour is characterised by the 
Lam\'e constants $\lambda$ and $\mu$, and the elastic energy is given by 
\emph{Koiter's energy for linearly elastic shells and transverse displacements} 
$K(\eta):= K(\eta,\eta)$ with 
\begin{align*}
  K(\eta,\zeta) := 
    \frac{1}{2}\int_M \varepsilon_s\,\langle C,\sigma(\eta)
      \otimes \sigma(\zeta)\rangle + \frac{\varepsilon_s^3}{3}
        \langle C, \xi(\eta)\otimes \xi(\zeta)\rangle\;dA.
\end{align*}
Here $C$ denotes the elasticity tensor of the shell,
\begin{align*}
  C_{\alpha\beta\gamma\delta}
    = \frac{4\mu\lambda}{\lambda + 2\mu}\,g_{\alpha\beta}\,g_{\gamma\delta}
      + 2\mu\,\big( g_{\alpha\gamma}\,g_{\beta\delta}
        + g_{\alpha\delta}\,g_{\beta\gamma}\big)
\end{align*}
and 
\begin{align*}
  \sigma(\eta)= - h\,\eta,
  \qquad
  \xi(\eta) = \nabla^2\eta - k\,\eta
\end{align*}
are the linearised strain tensors, where $k_{\alpha\beta}:= 
h^\sigma_\alpha h_{\alpha\beta}$, $\nabla$ denotes the Levi-Civita 
connection of $M$, and $\Delta$ is the corresponding Laplacian. We refer 
the interested reader to \cite{MR1757535,MR2312300} for additional 
details to the Koiter model. As shown in \cite[Theorem 4.4-2]{MR2312300}, 
$K(\eta, \zeta)$ is a symmetric bilinear form (and therefore $K(\eta)$ a 
quadratic form) which is coercive on $H_0^2(M)$. Moreover, as has been 
shown by partial integration in \cite[Chapter 3]{phdlengeler}, 
the $L^2$-gradient of this Koiter energy satisfies
\begin{align}\label{koiter_l2grad}
  2\,K(\eta,\zeta) = \int_M grad_{L^2}\,K(\eta)\,\zeta\;dA,  
    \qquad \eta,\zeta \in H^2_0(M).
\end{align}

Throughout the paper, we use standard notations for 
(vector valued) function spaces. 
We denote the fluid domain, which depends on a time $t\in I$ and is a 
priori not known, by $\Omega_{\eta(t)}\subset \R^3$. Then the motion 
of a homogeneous, incompressible, Newtonian fluid on the space-time 
domain $\Omega_\eta^I:=\bigcup_{t\in I}\{t\} \times \Omega_{\eta(t)}$ 
is governed by the system 
\begin{alignat}{2}
  \label{eqn:stokes}
    \rho\,\partial_t \mathbf{u} + \rho\,(\mathbf{u}\cdot \nabla)\mathbf{u} 
      &= \operatorname{div}\big( 2\sigma\,\mathbf{D}\mathbf{u}
        - \pi\,Id\big) + \mathbf{f} &\quad\quad&\text{in } \Omega_\eta^I,\\
  \label{eqn:inkompressible}
    \operatorname{div} \mathbf{u} &= 0 &\quad\quad&\text{in } \Omega_\eta^I,
\end{alignat}
where $\mathbf{u}$ is the velocity field, $\mathbf{D}\mathbf{u}$ the 
symmetric part of the gradient of $\mathbf{u}$, $\pi$ the pressure field, 
$\rho$ the density, $\sigma$ the dynamic viscosity, $\mathbf{f}$ an external 
body force and $Id$ the $3\times 3$ unit matrix. We will assume no-slip 
boundary conditions on the deformed boundary and natural 
boundary conditions on $\Gamma$, i.\,e.\
\begin{alignat}{2}
  \label{eqn:stokes_rand_moving}
    \mathbf{u}(\cdot,\cdot + \eta\,\boldsymbol{\nu}) 
      &=\partial_t\eta\,\boldsymbol{\nu} &\quad\quad&\text{on } I\times M,\\
  \label{eqn:stokes_rand_stress}
    \big( 2\sigma\,\mathbf{D}\mathbf{u}- \pi\,Id\big)\,\boldsymbol\nu 
      &= \frac{\rho}{2}(\mathbf{u}\cdot\boldsymbol\nu)\,\mathbf{u} 
        &\quad\quad&\text{on } I\times \Gamma.
\end{alignat}
Using the map $\Phi_{\eta(t)} : M \rightarrow \partial 
\Omega_{\eta(t)}\setminus \Gamma$, $\Phi_{\eta(t)}(q):= 
\eta(t,q)\,\boldsymbol\nu(q)$ to parametrize the deformed boundary, the 
force exerted by the fluid on this boundary is given by 
\begin{align*}
 \mathbf{F}&:= \big(-2\sigma\,\mathbf{D}\mathbf{u}\;\boldsymbol\nu_{\eta(\cdot)} 
 + \pi\,\boldsymbol{\nu}_{\eta(\cdot)}\big)\circ\Phi_{\eta(\cdot)}|\operatorname{det} d\Phi_{\eta(\cdot)}|,
\end{align*}
where $\boldsymbol\nu_{\eta(t)}$ is the outer normal of 
$\Omega_{\eta(t)}$. The external forces acting on the shell along 
the outer normal are thereby composed of $\mathbf{F}\cdot \boldsymbol\nu$ 
and some given external force $g$. Using Hamilton's principle, the 
displacement $\eta$ of the shell is a stationary point of the integrated 
difference between the kinetic and potential energy of the shell and 
therefore satisfies the corresponding Euler-Lagrange equation
\begin{alignat}{2}
  \label{eqn:koiter}
    2\,\varepsilon_s\,\rho_s\,\partial_t^2 \eta + grad_{L^2}K(\eta) 
      &= g + \mathbf{F}\cdot\boldsymbol{\nu} 
        &\quad\quad&\text{in } I\times M.
\end{alignat}
Since the shell is clamped, we have
\begin{align}
  \label{eqn:koiter:nullrand}
    \eta = 0, \qquad \nabla\eta &= \boldsymbol 0 
      \qquad \text{on } I\times \partial M.
\end{align}
Finally we prescribe the initial conditions
\begin{align}
  \label{eqn:gln:anfangswert}
    \mathbf{u}(\cdot,0) =\mathbf{u}_0 \text{ in } \Omega_{\eta_0}
  \quad \text{ and } \quad 
    \eta(\cdot,0) = \eta_0,\;
    \partial_t\eta(\cdot,0) = \eta_1 
    \text{ in } M.
\end{align}
In the following, we will analyse the system 
\eqref{eqn:stokes}--\eqref{eqn:gln:anfangswert}.
\subsection{Formal a priori estimates}
\label{sec:formal_apriori}
%
We take (sufficiently smooth) solutions $\mathbf{u}$ and $\eta$ of the 
fluid-  and shell equation.  Multiplying the fluid equation 
\eqref{eqn:stokes} with $\mathbf{u}(t)$ and integrating over 
$\Omega_{\eta(t)}$ leads to
\begin{align}\label{eqn:form_apriori_1}
  &\rho\intetat (\partial_t \mathbf{u})(t)\cdot\mathbf{u}(t)\;dx 
    + \rho\intetat\big(\mathbf{u}(t)\cdot \nabla\big)\mathbf{u}(t)\cdot 
      \mathbf{u}(t)\;dx\\
    &\quad = \intetat \operatorname{div} (2\sigma\,\mathbf{D}\mathbf{u}(t))\cdot \mathbf{u}(t)\;dx 
      - \intetat \nabla \pi(t)\cdot \mathbf{u}(t)\;dx 
      + \intetat \mathbf{f}(t)\cdot \mathbf{u}(t)\;dx.\notag
\end{align}
Using Reynolds transport theorem and our boundary conditions, i.\,e.\ 
the domain velocity equals the fluid velocity on the moving boundary 
$\partial\Omega_{\eta(t)}\setminus \Gamma$ and vanishes 
on $\Gamma$, the first term of the equation can be written as
\begin{align}\label{eqn:form_apriori_reynolds}
  \rho\intetat &(\partial_t \mathbf{u})(t)\cdot\mathbf{u}(t)\;dx\\\notag
    &= \frac{\rho}{2}\frac{d}{dt}\intetat |\mathbf{u}(t)|^2\;dx 
      - \frac{\rho}{2}\int_{\partial \Omega_{\eta(t)}\setminus{\Gamma}}
        |\mathbf{u}(t)|^2\,\mathbf{u}(t)\cdot \boldsymbol\nu_{\eta(t)}\;dA_{\eta(t)},
\end{align}
where $\boldsymbol\nu_{\eta(t)}$ 
denotes the outer normal and  $dA_{\eta(t)}$ 
the surface measure on $\partial\Omega_{\eta(t)}$. By 
partial integration and taking the divergence constraint into account, 
the convective term vanishes except a boundary term
\begin{align}\label{eqn:form_apriori_wirbel}
  \rho\intetat \big(\mathbf{u}(t) \cdot \nabla\big)\mathbf{u}(t)
      \cdot \mathbf{u}(t)\;dx 
    = \frac{\rho}{2} \int_{\partial\Omega_{\eta(t)}} |\mathbf{u}(t)|^2\, 
      \mathbf{u}(t)\cdot\boldsymbol\nu_{\eta(t)}\;dA_{\eta(t)}.
\end{align}
Likewise we get by partial integration, the divergence constraint, 
the symmetry of $\mathbf{D}\mathbf{u}$ and the the boundary condition 
\eqref{eqn:stokes_rand_stress} as well as $\boldsymbol\nu_{\eta(t)} 
= \boldsymbol{\nu}$ on the fixed boundary $\Gamma$
\begin{align}\label{eqn:form_apriori_3}
  \intetat &\operatorname{div} (2\sigma\,\mathbf{D}\mathbf{u}(t))\cdot \mathbf{u}(t)\;dx
    -\intetat \nabla \pi(t)\cdot \mathbf{u}(t)\;dx\\\notag
  &= - 2\sigma\,\intetat \mathbf{D}\mathbf{u}(t):\mathbf{D}\mathbf{u}(t)\;dx
    + \int_{\partial \Omega_{\eta(t)}\setminus \Gamma}
        \big(2\sigma\,\mathbf{D}\mathbf{u}(t)\;
          \boldsymbol\nu_{\eta(t)}\big)\cdot\mathbf{u}(t)\;dA_{\eta(t)}\\
  &\qquad - \int_{\partial\Omega_{\eta(t)}\setminus \Gamma} \pi(t)\,
        \mathbf{u}(t)\cdot\boldsymbol\nu_{\eta(t)}\;dA_{\eta(t)}
    + \frac{\rho}{2} \int_{\Gamma} |\mathbf{u}(t)|^2\, 
      \mathbf{u}(t)\cdot\boldsymbol\nu\;dA.\notag
\end{align}
Hence, using \eqref{eqn:form_apriori_reynolds}, \eqref{eqn:form_apriori_wirbel} 
and \eqref{eqn:form_apriori_3}, the equation \eqref{eqn:form_apriori_1} 
can be written as
\begin{align}\label{eqn:form_apriori_6}
  &\frac{\rho}{2}\frac{d}{dt}\intetat |\mathbf{u}(t)|^2\;dx 
    + 2\sigma\,\intetat \mathbf{D}\mathbf{u}(t):\mathbf{D}\mathbf{u}(t)\;dx\\
  &\quad = \intetat \mathbf{f}(t)\cdot\mathbf{u}(t)\;dx 
    + \int_{\partial \Omega_{\eta(t)}\setminus \Gamma}
        \big(2\sigma\,\mathbf{D}\mathbf{u}(t)\;
          \boldsymbol\nu_{\eta(t)} - \pi(t)\,\boldsymbol\nu_{\eta(t)} \big)
            \cdot\mathbf{u}(t)\;dA_{\eta(t)}.\notag
\end{align}
Multiplying equation \eqref{eqn:koiter} with 
$\partial_t \eta(t)$, integrating over $M$ and using the bilinearity of 
the Koiter-Energy as well as 
$(grad_{L^2}\,K(\eta(t)),\partial_t\eta(t))_{L^2(M)} = 
2\,K(\eta(t),\partial_t\eta(t))$, we get 
\begin{align}\label{form_apriori_7}\begin{aligned}
  \varepsilon_s\,\rho_s \frac{d}{dt} \int_M &\big|\partial_t \eta (t)\big|^2\;dA  
      + \frac{d}{dt} K\big(\partial_t \eta(t)\big)\\
  &= \int_M g(t)\;\partial_t \eta(t)\;dA 
    + \int_M \mathbf{F}(t)\cdot \boldsymbol\nu\; \partial_t \eta(t)\;dA.
\end{aligned}\end{align}
Using the definition of $\mathbf{F}$, the boundary condition 
\eqref{eqn:stokes_rand_moving} and a change of variables, adding 
\eqref{eqn:form_apriori_6} and \eqref{form_apriori_7} leads to 
the energy equality
\begin{align}\notag
  \frac{d}{dt}\Big(\frac{\rho}{2}\intetat|\mathbf{u}(t)|^2\;dx 
    &+ 2\sigma \int_0^t \intetas |\mathbf{D}\mathbf{u}(s)|^2\;dx\,ds
    + \varepsilon_s\,\rho_s \int_M |\partial_t\eta(t)|^2\;dA 
    + K\big(\eta(t)\big)\Big)\\
  &= \intetat\! \mathbf{f}(t)\cdot\mathbf{u}(t)\;dx 
    + \int_M\! g(t)\;\partial_t \eta(t)\;dA.
\label{form_apriori_9}
\end{align}
We denote the expression in the parentheses on the left-hand side of 
\eqref{form_apriori_9} by $E(t)$, and set $E_0:= E(0)$, which depends by 
our initial conditions \eqref{eqn:gln:anfangswert} only on $\eta_0$, 
$\eta_1$ and $\mathbf{u}_0$. The coercivity of the Koiter Energy 
$K$ and Gr\"onwall's inequality (see \cite[Appendix]{MR0348562}) imply 
the estimate
\begin{align}\label{form_apriori_11}
  \operatorname*{esssup}_{t\in (0,T)}\sqrt{E(t)}
    \leq \sqrt{E_0}
      + \int_0^{T}\!\! \frac{1}{\sqrt{2\,\rho}}\,\lVert\mathbf{f}(s,\cdot)
        \rVert_{L^2(\Omega_{\eta(s)})} 
      + \frac{1}{2\,\sqrt{\varepsilon_s\rho_s}}\lVert g(s,\cdot)\rVert_{L^2(M)}\;ds.
\end{align}
Again by the coercivity of the Koiter Energy the quantity 
\begin{align*}
  \lVert \mathbf{u}\rVert_{L^\infty(I;L^2(\Omega_{\eta(t)}))}^2 
    + \lVert \mathbf{D}\mathbf{u}\rVert_{L^2(I;L^2(\Omega_{\eta(t)}))}^2 
    + \lVert \partial_t\eta\rVert_{L^\infty(I;L^2(M))}^2 
    + \lVert \eta\rVert_{L^\infty(I;H^2(M))}^2 
\end{align*}
is bounded by a constant depending only on the data. This (spatial)
regularity of the displacement $\eta(t,\cdot)\in H^2(M)$ is not enough
to ensure Lipschitz continuity, but only H\"older continuity
$C^{0,\lambda}(M)$ for any $\lambda<1$. Hence, the boundary of the
moving domain $\Omega_{\eta(t)}$ is not necessarily Lipschitz and the
classical partial integration theorem, trace operators, etc. cannot be
used. In the next Section we will develop the necessary framework for
the moving domain, using a special reference domain. For the sake of
better readability, we set the constants $\epsilon_s$, $\rho_s$,
$\rho$ and $\sigma$ equal to 1 throughout the paper, but emphasize
that all the computations hold with the original constants.
\section[Moving domains]{Moving domains}\label{sec:moving_domains}
We assume that $\Omega\subset \R^3$ is a bounded domain with $\partial 
\Omega\in C^{0,1}$ and $\partial \Omega = \Gamma \cup M$, where $M,\, 
\Gamma \neq \emptyset$ are compact, oriented, embedded two-dimensional 
$C^4$-manifolds with smooth non-empty boundaries and $\partial M = 
\partial \Gamma$. Furthermore, we assume $M$ to be connected and $\Gamma$ 
to be the finite union of connected components $\Gamma_i$, where 
$\Gamma_i$ is subset of a hyperplane perpendicular to $M$, i.\,e.\ the 
continuous extension of the outer normal $\boldsymbol\nu$ on $int\,M$ to 
$\partial M\cap \partial \Gamma_i$ is perpendicular to the outer normal 
of the hyperplane. For $\alpha >0$ the open $\alpha$-tube $S_\alpha$ 
and the half-closed tube $\underline{S_\alpha}$ around $int\,M$ are 
given by
\begin{align*}
  S_\alpha &:= \{ x \in \R^3 \,|\, x = q + 
    s\,\boldsymbol{\nu}(q),\, q\in int\,M,\, -\alpha < s < \alpha\,\},\\
  \underline{S_\alpha} &:= \{ x \in \R^3 \,|\, x = q + 
    s\,\boldsymbol{\nu}(q),\, q\in M,\, -\alpha < s < \alpha\,\}.
\end{align*}
We assume\footnote{Since $int\,M$ is an embedded $C^4$-manifold, the 
tubular neighbourhood theorem already ensures the existence of an 
$C^3$-diffeomorphism from $(-\kappa,\kappa)\times int\,M$ to $S_\kappa$.} 
that there exists a $\kappa>0$ such that the map 
$\Lambda:M\times (-\kappa,\kappa) \rightarrow \underline{S_\kappa}$, 
$\Lambda(q,s):= q + s\,\boldsymbol\nu(q)$ is a $C^3$-diffeomorphism, and 
write $\Lambda^{-1}(x) = (\mathbf{q}(x),s(x))$ for the inverse mapping.
For $0<\alpha\leq \kappa$ we divide the boundary of $S_\alpha$ into 
the parts\label{definition_randstueke} 
\begin{align*}
  \partial S_\alpha
    & = \{ x\in \R^3\bigm| x = q + \alpha\,\boldsymbol{\nu}(q),\, 
      q\in int\,M\}\\
    &\qquad \cup \{ x\in \R^3\bigm| x = q - \alpha\,\boldsymbol{\nu}(q),
      \, q\in int\,M\}\\
    &\qquad\quad \cup \{ x\in \R^3 \bigm| x = q + s\,\boldsymbol{\nu}(q),
      \, q\in \partial M,\, -\alpha \leq s \leq \alpha\big\}\\
    &=: M^\alpha_+ \cup M^\alpha_- \cup \Gamma_s^\alpha,
\end{align*}
where the outer normal is again defined through $int\,M$. 
Furthermore, we can assume that the sets $M^\alpha_+$, $M^\alpha_-$ and 
$\Gamma_s^\alpha$ are disjoint and the representations through 
$x = q + \beta\,\boldsymbol{\nu}(q)$ are unique. Hence, the 
orthogonality assumption implies $\partial S_\alpha \in C^{0,1}$. We 
will require that the domain's deformation is taking place inside 
$\Omega\cup S_\kappa$. To ensure that the deformed moving boundary does 
not interfere with the fixed in- and outflow boundary $\Gamma$, we 
assume that 
\begin{align*}
  \{ x\in \R^3 \bigm| x = q + s\,\boldsymbol{\nu}(q),\, q\in 
    \Gamma\cup \Gamma_s^\kappa,\,
    s\in [0,\varepsilon)\}\cap (S_\kappa \cup \Omega)= \emptyset
\end{align*}
for some $\varepsilon>0$, where $\boldsymbol\nu$ is the 
extension of the outer normal on $int\, \Gamma$. Finally, we assume that 
for all $0<\alpha <\kappa$ it holds that $int\, \Gamma_s^{\alpha,c}\neq 
\emptyset$, where
\begin{align*}
\Gamma_s^{\alpha,c}:=\Gamma \setminus \{x\in \R^3 \,|\, x = 
q + s\,\boldsymbol{\nu}(q),\, q\in \partial M,\, -\alpha < s \leq 
\alpha\}.
\end{align*}
We call a domain $\Omega$ fulfilling these requirements an \emph{admissible 
in- and outflow domain}. Its simplest example is the straight tubular 
cylinder. For the rest of paper, we will always assume that $\Omega$ is 
an admissible in- and outflow domain.
\begin{figure}\centering 
  \includegraphics{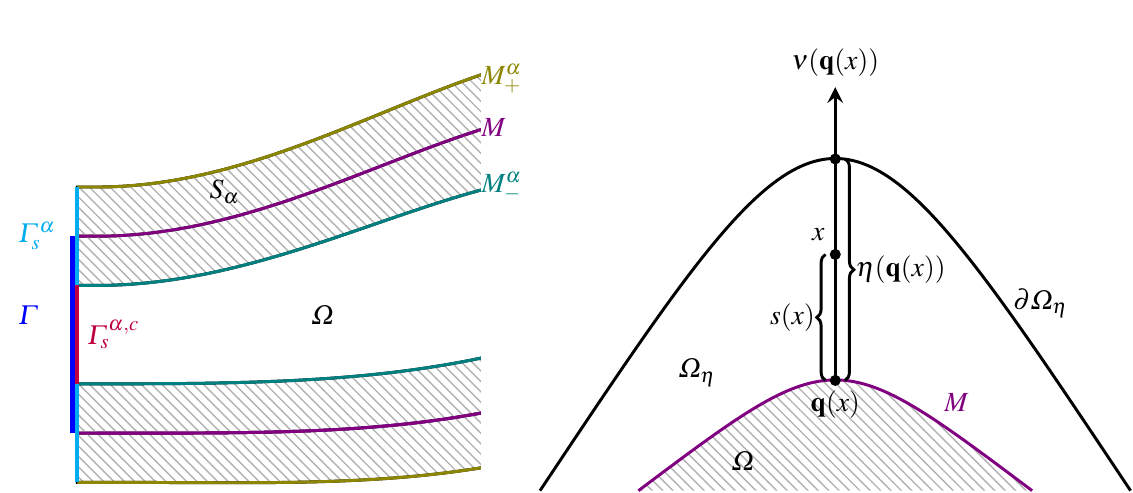}
  \caption{Notations for admissible in- and outflow domains and 
    moving domains.}
  \label{fig:randstueke}
\end{figure}
\begin{definition}\label{def:bewegtes_gebiet}
  Let $\eta: M \rightarrow (-\kappa,\kappa)$ be continuous. We set the  
  \emph{moving domain} $\Omega_{\eta}$ as
  \begin{align*}
  \Omega_{\eta}:=\big(\Omega\setminus S_\kappa\big)\cup 
    \big\{ x\in S_\kappa\bigm| s(x) < \eta\big(\mathbf{q}(x)\big)\big\},
  \end{align*}
  see Figure \ref{fig:randstueke}. Furthermore, for $0< \alpha < 
  \kappa$ we set $B_\alpha := \Omega \cup S_\alpha$.
\end{definition}
\begin{remark}\label{bem:eigenschaften_omega_eta}
  $\Omega_\eta$ and $B_\alpha$ are bounded domains, $\partial B_\alpha = 
  \Gamma_s^{\alpha,c}\cup \Gamma_s^\alpha \cup M_+^\alpha$ and $\partial 
  B_\alpha\in C^{0,1}$. With $\Gamma_\eta := 
  \partial\Omega_{\eta}\setminus\{x\in S_\kappa \,|\, 
  s(x) = \eta(\mathbf{q}(x))\}$ we have $\Gamma_\eta \subset 
  \Gamma\cup\Gamma_s^\kappa$ and in particular $\Gamma_\eta = \Gamma$ for 
  $\eta=0$ on $\partial M$.
\end{remark}
%
For a given displacement $\eta \in C^0(M)$ with $\lVert \eta 
\rVert_{L^\infty(M)}<\kappa$ we choose a cut-off function 
$\beta\in C^\infty(\R)$ with $\beta=0$ in a neighbourhood of $-1$, 
$\beta=1$ in a neighbourhood of $0$, and $\lVert \beta 
\rVert_{L^\infty([-1,0])}< \kappa\,\lVert \eta \rVert_{L^\infty(M)}^{-1}$.
We define the Hanzawa transform (see \cite{MR633045}) trough 
\begin{align*}
  \Psi_{\eta}(x) &:=
    \begin{cases}
      x + \eta(\mathbf{q}(x))\beta(\frac{s(x)}{\kappa})\boldsymbol{\nu}(\mathbf{q}(x)) 
        &x\in\overline{\Omega}\cap \underline{S_\kappa},\\
      x &x\in \overline{\Omega}\setminus \underline{S_\kappa}, 
    \end{cases}
 \end{align*}
where $\boldsymbol\nu$ on $\partial M$ is the continuous extension of 
the outer normal $\boldsymbol\nu$ on $int\,M$ to $\partial M$. By the 
choice of the cut-off function $\beta$ and the properties of the 
diffeomorphism $\Lambda$, one can show that 
$\Psi_\eta:\overline{\Omega}\rightarrow \overline{\Omega_\eta}$, as well 
as $\Phi_\eta:\partial\Omega\rightarrow \partial \Omega_\eta$ with 
$\Phi_\eta:=\Psi_\eta\lvert_{\partial\Omega}$, are homeomorphisms which  
inherit the regularity of $\eta$, i.\,e.\ $\Psi_\eta$ and $\Phi_\eta$ are 
$C^k$-diffeomorphisms if $\eta\in C^k(M)$, $k\in \{1,\,2,\,3\}$ and 
the Jacobian determinant $\operatorname{det}\, d\Psi_\eta$ is positive. 
Furthermore, the components of the Jacobians of $\Psi_\eta$, $\Phi_\eta$ 
and their inverses have the form 
\begin{align}\label{form_jacobian}
  b_0 + b_1(\eta\circ \mathbf{q}) + \mathbf{b}\cdot (\nabla\eta\circ \mathbf{q})
\end{align}
for some bounded, continuous functions $b_0$, $b_1$, $\mathbf{b}$, whose 
supports are contained in $S_\kappa$. The details of these calculations 
can be found in \cite{phdeberlein}. The cut-off function $\beta$ used in 
these definitions clearly can be chosen uniformly for a set of displacements 
fulfilling $\lVert \eta \rVert_{L^\infty(M)}\leq \alpha < \kappa$. 
Furthermore, by some easy calculations we get the following result:
\begin{lemma}\label{lem:hanzawa_konvergence}
  Let $\eta,\,\eta_n\in C^0(M)$ with $0<\lVert \eta 
  \rVert_{L^\infty(M)},\, \lVert \eta_n \rVert_{L^\infty(M)}
  <\kappa$ let $\Psi_{\eta_n}$, $\Psi_\eta$ be defined using the same 
  cut-off function $\beta$.
  \begin{enumerate}[label=\alph*)]
    \item Let $\eta_n\rightarrow \eta$ uniformly on $M$. Then 
      $\Psi_{\eta_n}$ converges uniformly to $\Psi_\eta$ on 
      $\overline{\Omega}$.
    \item Let $\eta_n\rightarrow \eta$ uniformly on $M$ and 
      $\Psi_{\eta_n}^{-1}$, $\Psi_{\eta}^{-1}$ extended by 
      $\mathbf{q}$ to $\underline{B_\kappa}:= 
      \underline{S_\kappa}\cup \overline{\Omega}$. Then 
      $\Psi_{\eta_n}^{-1}$ converges uniformly to $\Psi_\eta^{-1}$ on 
      $\underline{B_\kappa}$.
    \item Let $\eta_n\rightharpoonup \eta$ weakly in $H^2(M)$ and 
      $1\leq s <\infty$. Then $\nabla \Psi_{\eta_n}$ converges to 
      $\nabla\Psi_\eta$ in $L^s(\Omega)$ and the (canonically extended) 
      functions $(\nabla\Psi_{\eta_n}^{-1})\,\chi_{\Omega_{\eta_n}}$
      converge to \linebreak $(\nabla\Psi_{\eta}^{-1})\,\chi_{\Omega_{\eta}}$ in 
      $L^s(B_\kappa)$. Also the Jacobian determinants $\det d\Psi_{\eta_n}$ 
      converge to $\det d\Psi_\eta$ in $L^s(\Omega)$  and 
      $\det d\Psi_{\eta_n}^{-1}\,\chi_{\Omega_{\eta_n}}$ converges to 
      $\det d\Psi_\eta^{-1}\,\chi_{\Omega_\eta}$ in $L^s(B_\kappa)$.
  \end{enumerate}
\end{lemma}
By our formal a priori estimate, we cannot expect the deformation and 
therefore the Hanzawa transform to be Lipschitz. Hence, a change of 
variables by the Hanzawa transform does not hold in the classical sense.
Since the Hanzawa transformation inherits the regularity of the 
displacement $\eta$ and preserves the convergence of $\eta_n$ in a 
suitable way by the preceding Lemma, an approximation argument shows 
that a change of variables is still possible. However, we have a small 
loss of regularity since the Jacobian determinant $\det d\Psi_{\eta}$ is 
not $L^\infty$. A careful inspection of this approximation argument, 
which can be found in \cite{phdeberlein}, also gives a bound for the 
continuity constant.
\begin{lemma}\label{lem:einbettung_omega_eta}
  Let $\eta\in H^2(M)$ with $\lVert\eta\rVert_{L^\infty(M)}<\kappa$
  and $1<p\leq \infty$. Then for all $1\leq r<p$ the linear mapping
  \begin{align*}
    \mathbf{v}\mapsto \mathbf{v}\circ \Psi_\eta
  \end{align*}
  is  continuous from $L^p(\Omega_\eta)$ to $L^r(\Omega)$ and from 
  $W^{1,p}(\Omega_\eta)$ to $W^{1,r}(\Omega)$. The continuity constant
  only depends on $\Omega$, $p$, $r$, $\lVert\eta\rVert_{H^2(M)}$ and 
  the choice of $\beta$ in the definition of the Hanzawa transform. An 
  analogous statement holds for $\Psi_\eta^{-1}$ instead of $\Psi_\eta$.
\end{lemma}
Combining the last two results we also get the following Lemma.
\begin{lemma}\label{lem:konv_trafo}
  Let $1< p \leq \infty$, $\eta_n \rightharpoonup \eta$ weakly in $H^2(M)$ with 
  $\lVert \eta \rVert_{L^\infty(M)}<\kappa$ and let $\Psi_{\eta_n}$, 
  $\Psi_\eta$ be defined by the same cut-off function $\beta$. For 
  $\mathbf{v} \in W^{1,p}(B_\kappa)$ it holds $\mathbf{v}\circ 
  \Psi_{\eta_n} \rightarrow \mathbf{v}\circ\Psi_\eta$ in 
  $W^{1,r}(\Omega)$ for all $1\leq r < p$. For $\mathbf{v}\in 
  W^{1,p}_0(\Omega)$ it holds $\mathbf{v}\circ \Psi_{\eta_n}^{-1} 
  \rightarrow \mathbf{v} \circ \Psi_\eta^{-1}$ in $W^{1,r}(B_\kappa)$ 
  for all $1\leq r < p$, where the functions are extended by 
  $\boldsymbol 0$. Similar results hold for the convergence in the 
  appropriate $L^p$-spaces.
\end{lemma}
To preserve divergence constraint, we introduce the Piola transform 
$\Teta\boldsymbol\varphi$ of $\boldsymbol\varphi$ as pushforward of 
$(\det d\Psi_\eta)^{-1}\boldsymbol\varphi$ under 
$\Psi_\eta$, i.e. $\Teta\boldsymbol\varphi=\big(d\Psi_\eta\,(\det 
d\Psi_\eta)^{-1}\,\boldsymbol\varphi\big)\circ\Psi_\eta^{-1}$. Hence 
$\Teta$ with $\Teta^{-1}(\boldsymbol\varphi)= 
\big(d\Psi_\eta^{-1}\,(\det d\Psi_\eta^{-1})^{-1}\,
\boldsymbol\varphi\big)\circ \Psi_\eta$ is an isomorphism between the 
Lebesgue- and Sobolev spaces on $\Omega$, respectively $\Omega_\eta$, 
as long as the order of differentiability is not larger than one. 
Furthermore, $\Teta$ preserves the divergence-free property, as can be 
seen in \cite[Theorem 7.20]{MR1262126}.

Using Lemma \ref{lem:einbettung_omega_eta}, we can obtain the usual 
Sobolev embeddings by transforming to the reference domain, but we have 
to take into account a loss of regularity. Also, the following trace 
operator, which compares the fluid velocity and the boundary velocity 
(which is given in Lagrange coordinates) is continuous.
\begin{definition}\label{def:treta}
  Let $1<p \leq \infty$ and $\eta\in H^2(M)$ with $\lVert \eta 
  \rVert_{L^\infty(M)}<\kappa$. For $1<r<p$ we define the linear, 
  continuous trace operator $\treta$ by
  \begin{align*}
    \treta: W^{1,p}(\Omega_\eta)\rightarrow W^{1-\frac{1}{r},r}(\partial \Omega),
    \quad v\mapsto \left.(v\circ \Psi_\eta)\right\lvert_{\partial\Omega}.
  \end{align*}
\end{definition}
%
%
Since the regularity of $\partial\Omega_{\eta}$ does not guarantee the 
existence of an outer normal, we will derive an suitable analogue, which 
implies a formula of Green type. We approximate $\eta\in H^2(M)$ by 
$(\eta_{n})_{n\in \N}\subset C^2(M)$ in $H^2(M)$ and choose the same 
cut-off function $\beta$ for the definitions of $\Phi_{\eta_n}$. Since 
$int\, M$ is a two-dimensional $C^4$-manifold, there exist (locally) 
orthonormal, tangential $C^3$ vector fields $\mathbf{e}_1,\,\mathbf{e}_2$ 
with $\mathbf{e}_1\times\mathbf{e}_2 = \boldsymbol{\nu}$. We set the 
(local) vector fields $\mathbf{v}_1^n := d\Phi_{\eta_n}\mathbf{e}_1$, 
$\mathbf{v}_2^n := d\Phi_{\eta_n}\mathbf{e}_2$ and consider 
$d\Phi_{\eta_n}$ as a linear map from the parallelogram spanned by 
$\mathbf{e}_1,\,\mathbf{e}_2$ into the parallelogram spanned by 
$\mathbf{v}_1^n$, $\mathbf{v}_2^n$. Then we get
\begin{align*}
  |\mathbf{v}_1^n\times \mathbf{v}_2^n| 
    = |\operatorname{det} d\Phi_{\eta_n}|\,|\mathbf{e}_1\times \mathbf{e}_2|
    = |\operatorname{det} d\Phi_{\eta_n}|\,|\boldsymbol{\nu}|=|\operatorname{det} d\Phi_{\eta_n}|.
\end{align*}
By $\Phi_{\eta_n}\in C^2(M)$, the outer normal $\boldsymbol{\nu}_{\eta_n}$ 
of $\Omega_{\eta_n}$ admits the representation 
$\boldsymbol{\nu}_{\eta_n}\circ \Phi_{\eta_n} = 
\mathbf{v}_1^n\times \mathbf{v}_2^n\,/\, |\mathbf{v}_1^n\times 
\mathbf{v}_2^n|$. We set $\mathbf{v}^n:= \mathbf{v}_1^n\times 
\mathbf{v}_2^n$ and get  $\mathbf{v}^n = (\boldsymbol{\nu}_{\eta_n}\circ 
\Phi_{\eta_n})\,|\operatorname{det} d\Phi_{\eta_n}|$. Thus $\mathbf{v}^n$ 
is independent of the choice of $\mathbf{e}_1,\,\mathbf{e}_2$ and 
consequently defined globally on $int\, M$. For $q\in int\,M$ and a curve 
$c$ on $int\,M$ with $c(0)=q$ and $\left.\frac{d}{dt}\right\lvert_{t=0} c(t) = 
\mathbf{e}_i(q)$ we compute
\begin{align}\label{grundlagen:eqn:rand_transform_normale1}\begin{aligned}
  \mathbf{v}_i^n(q) 
  &= \left.\frac{d}{dt}\right\lvert_{t=0}\Phi_{\eta_n}(c(t))
    = \left.\frac{d}{dt}\right\lvert_{t=0} \big(c(t) 
      + \eta_n(c(t))\,\boldsymbol{\nu}(c(t))\big)\\
  &= \mathbf{e}_i(q) + d\eta_n(q)\,\mathbf{e}_i(q)\,\boldsymbol{\nu}(q) 
      + \eta_n(q)\,\left.\frac{d}{dt}\right\rvert_{t=0}\boldsymbol{\nu}(c(t))\\
  &= \mathbf{e}_i(q) + d\eta_n(q)\,\mathbf{e}_i(q)\,\boldsymbol{\nu}(q) 
      + \eta_n(q)\,\big(h_i^1(q)\,\mathbf{e}_1(q) 
      + h_i^2(q)\,\mathbf{e}_2(q)\big).
\end{aligned}\end{align}
In here, $h_i^j$ denotes the components of the Weingarten map with respect 
to the ortho\-normal basis $\mathbf{e}_1,\,\mathbf{e}_2$. Hence, we have
\begin{align}\label{grundlagen:eqn:rand_transform_normale2}\begin{aligned}
  \boldsymbol{\nu}\cdot \mathbf{v}^n 
  = 1 - (h_1^1 + h_2^2)\,\eta_n + (h_1^1h_2^2 + h_1^2h_2^1)\,\eta_n^2
  = 1 - 2\,H\, \eta_n + G\, \eta_n^2,
\end{aligned}\end{align}
where $H$ is the mean curvature and $G$ the Gauss curvature of $int\,M$.
Taking the limit $n\rightarrow \infty$ in 
\eqref{grundlagen:eqn:rand_transform_normale1} and 
\eqref{grundlagen:eqn:rand_transform_normale2} yields the convergence 
of $\mathbf{v}^n$ towards an $\mathbf{v}_{\eta}$ in $H^1(M)$, and the 
convergence of $\boldsymbol{\nu}\cdot \mathbf{v}^n$ towards 
$1 - 2\,H\, \eta + G\, \eta^2$ in $H^{2,1}(M)\cap L^\infty(M)$. This 
brings us to the following definition.
\begin{definition}\label{grundlagen:def_gamma}
  Let $\eta\in H^2(M)$ with $\lVert\eta\rVert_{L^\infty(M)}<\kappa$. We 
  call the above constructed vector field $\mathbf{v}_\eta$ with 
  $\mathbf{v}_\eta\in L^r(M)$ for $1\leq r< \infty$ \emph{scaled 
  pseudonormal} and define $\gamma(\eta)\in H^{2,1}(M)\cap L^\infty(M)$ 
  by $\gamma(\eta) := 1 - 2\,H\, \eta + G\, \eta^2$.
\end{definition}
\begin{remark}\label{grundlagen:eigenschaften_gamma} 
  We have already shown the following properties of $\mathbf{v}_\eta$ and 
  $\gamma(\eta)$.
  \begin{enumerate}[label=\alph*)]
    \item For $\eta\in C^2(M)$ with $\lVert\eta\rVert_{L^\infty(M)}<\kappa$ 
      the identities $\mathbf{v}_\eta = (\boldsymbol{\nu}_{\eta}\circ \Phi_{\eta})\,
      |\operatorname{det} d\Phi_{\eta}|$ and $\gamma(\eta) = 
      \boldsymbol{\nu}\cdot (\boldsymbol{\nu}_\eta \circ 
      \Phi_\eta)\,|\operatorname{det} d\Phi_{\eta}|$ hold on $M$.
    \item For $\eta\in H^2(M)$ with $\lVert\eta\rVert_{L^\infty(M)}<\kappa$ 
      follows $\boldsymbol\nu \cdot \mathbf{v}_\eta = \gamma(\eta)$ on $M$.
    \item Let $\eta,\,(\eta_n)_{n\in\N}\subset H^2(M)$ with 
      $\lVert\eta\rVert_{L^\infty(M)},\,\lVert\eta_n\rVert_{L^\infty(M)}<\kappa$ and $\eta_n\rightarrow \eta$ 
      uniformly on $M$. Then $\gamma(\eta_n)$ converges 
      uniformly towards $\gamma(\eta)$ on $M$.
  \end{enumerate}
\end{remark}
Using the scaled pseudonormal, we can now show a Green's type formula.
\begin{proposition}\label{prop:part_int}
  Let $\eta\in H^2(M)$, $\lVert\eta\rVert_{L^\infty(M)}<\kappa$, 
  $1<p,p'\leq \infty$ with $\frac{1}{p}+ \frac{1}{p'}=1$ and 
  $\boldsymbol{\varphi}\in W^{1,p}(\Omega_\eta)$, 
  $\psi\in W^{1,p'}(\Omega_\eta)$. Then
  \begin{align}\label{lem:normale_spur-1}
      \int_{\Omega_{\eta}}\boldsymbol{\varphi}\cdot \nabla \psi\;dx 
        + \int_{\Omega_{\eta}}\operatorname{div} \boldsymbol{\varphi}\; \psi\;dx 
      &= \int_{M} \treta\boldsymbol{\varphi}\cdot 
          \mathbf{v}_{\eta}\,\treta\psi\;dA\\
        &\qquad+\int_{\Gamma} \treta\boldsymbol{\varphi}\cdot 
          \boldsymbol{\nu}\,\treta\psi\,|\operatorname{det} d\Phi_{\eta}|\;dA.\notag
  \end{align}
\end{proposition}
%
\begin{proof}
  Since $\Omega_\eta$ is bounded, for any $1<p<\infty$ the embedding of 
  $W^{1,\infty}(\Omega_\eta)$ into $W^{1,p}(\Omega_\eta)$ is continuous.
  Hence, it suffices to treat the case $1<p,p'<\infty$.

  We approximate $\boldsymbol{\varphi}$ by $(\boldsymbol{\varphi}_k)_{k\in \N}
  \subset C_0^\infty(\R^3)$ in $W^{1,p}(\Omega_\eta)$, $\psi$ by 
  $(\psi_\ell)_{\ell\in\N} \subset C_0^\infty(\R^3)$ in 
  $W^{1,p'}(\Omega_\eta)$ and $\eta$ by $(\eta_{n})_{n\in \N}
  \subset C^2(M)$ in $H^2(M)$. Moreover, we chose the cut-off function
  $\beta$ uniformly for all definitions of the Hanzawa transforms. 
  By partial integration (notice $\partial {\Omega_{\eta_n}}\in C^{0,1}$), 
  we get
  \begin{align*}
    \int_{\Omega_{\eta_n}}\boldsymbol{\varphi}_k\cdot \nabla \psi_\ell\;dx 
      = -\int_{\Omega_{\eta_n}}\operatorname{div} \boldsymbol{\varphi}_k\; \psi_\ell\;dx 
        + \int_{\partial\Omega_{\eta_n}}\boldsymbol{\varphi}_k\cdot 
        \boldsymbol{\nu}_{\eta_n}\,\psi_\ell\;dA_{\eta_n}.
  \end{align*}
  Since $\boldsymbol{\nu}_{\eta_n}\circ \Phi_{\eta_n} = \boldsymbol\nu$
  on $\Gamma$, a change of variables yields
  \begin{align*}
    \int_{\partial\Omega_{\eta_n}}\boldsymbol{\varphi}_k\cdot 
      \boldsymbol{\nu}_{\eta_n}\,\psi_\ell\;dA_{\eta_n} 
      &= \int_{M} \tretan\boldsymbol{\varphi}_k\cdot 
          (\boldsymbol{\nu}_{\eta_n}\circ \Phi_{\eta_n})\tretan\psi_\ell\,
          |\operatorname{det} d\Phi_{\eta_n}|\;dA\\
        &\qquad +\int_{\Gamma} \tretan\boldsymbol{\varphi}_k\cdot 
          \boldsymbol{\nu}\,\tretan\psi_\ell\,
          |\operatorname{det} d\Phi_{\eta_n}|\;dA.
  \end{align*}
  Moreover, by Remark \ref{grundlagen:eigenschaften_gamma} we get
  \begin{align*}
    \int_{M} \tretan\boldsymbol{\varphi}_k\cdot 
      (\boldsymbol{\nu}_{\eta_n}\circ \Phi_{\eta_n})\tretan\psi_\ell\,
      |\operatorname{det} d\Phi_{\eta_n}|\;dA
    =\int_{M} \tretan\boldsymbol{\varphi}_k\cdot 
      \mathbf{v}_{\eta_n}\,\tretan\psi_\ell\;dA.
  \end{align*}
  At the construction of the scaled pseudonormal we already saw that 
  $\mathbf{v}_{\eta_n} \rightarrow \mathbf{v}_\eta$ in $L^r(M)$ for any 
  $1\leq r<\infty$. Using Lemma \ref{lem:hanzawa_konvergence}, Lemma 
  \ref{lem:konv_trafo} and the definition of $\treta$, taking the 
  limit $n\rightarrow \infty$ yields
  \begin{align*}
    \int_{\Omega_{\eta}}\boldsymbol{\varphi}_k\cdot \nabla \psi_\ell\;dx 
      &= -\int_{\Omega_{\eta}}\operatorname{div} \boldsymbol{\varphi}_k\; \psi_\ell\;dx 
        +\int_{M} \treta\boldsymbol{\varphi}_k\cdot 
          \mathbf{v}_{\eta}\,\treta\psi_\ell\;dA\\
        &\qquad+\int_{\Gamma} \treta\boldsymbol{\varphi}_k\cdot 
          \boldsymbol{\nu}\,\treta\psi_\ell\,
          |\operatorname{det} d\Phi_{\eta}|\;dA.
  \end{align*}
  Now passing to the limit $k\rightarrow\infty$ and $\ell\rightarrow\infty$ 
  finally proofs \eqref{lem:normale_spur-1}.
\end{proof}
Although domains with H\"older-continuous boundary generally do not 
admit a Korn's inequality (cf.\ \cite{MR2988724}), we can -- allowing 
the typical loss of regularity -- show a similar statement. 
\begin{lemma}\label{lem:korn_type_inequality}
  Let $\eta\in H^2(M)$ with $\lVert\eta\rVert_{L^\infty(M)} < \alpha 
  < \kappa$ and $1<p<\infty$. Then for all $1\leq r<p$ there 
  exists a constant $C$ such that for all $\boldsymbol\varphi\in C^{1}
  (\overline{\Omega_\eta})$ it holds
  \begin{align*}
    \lVert \nabla \boldsymbol\varphi \rVert_{L^{r}(\Omega_{\eta})} 
      \leq C\, \big(\lVert \mathbf{D}\boldsymbol\varphi\rVert_{L^p(\Omega_{\eta})} 
        + \lVert \boldsymbol\varphi \rVert_{L^p(\Omega_\eta)}\big).
  \end{align*}
  For a fixed $N\in \N$, the constant $C$ can be chosen uniformly with 
  respect to $\eta\in H^2(M)$ satisfying $\lVert \eta \rVert_{H^2(M)}\leq N$ 
  and $\lVert\eta\rVert_{L^\infty(M)} <\alpha$.
\end{lemma}
\begin{proof}
  We proceed analogously to \cite[Proposition 2.13]{MR3233099}. Since 
  $H^2(M)\hookrightarrow C^{0,\beta}$ for any $0<\beta<1$, the moving 
  domain $\Omega_\eta$ possesses a $\beta$-H\"older boundary. Therefore by  
  \cite[Theorem 3.1]{MR2988724} for $\boldsymbol\varphi\in C^{1}
  (\overline{\Omega_\eta})$ the inequality
  \begin{align}\label{korn_distance}
    \lVert \nabla \boldsymbol\varphi\,d^{1-\beta}\rVert_{L^p(\Omega_\eta)}
      \leq c\Big(\lVert \mathbf{D}\boldsymbol\varphi\rVert_{L^p(\Omega_\eta)}
        + \lVert \boldsymbol\varphi \rVert_{L^p(\Omega_\eta)}\Big),
  \end{align}
  holds, where $d(x)$ is the distance from $x\in \Omega_\eta$ to 
  $\partial\Omega_\eta$. A careful inspection of the proof shows that 
  for a fixed $N\in \N$ the constant $c$ can be chosen uniformly with 
  respect to $\eta\in H^2(M)$ satisfying $\lVert \eta \rVert_{H^2(M)}\leq N$ and 
  $\lVert\eta\rVert_{L^\infty(M)} <\alpha$. By H\"older's inequality
  one gets
  \begin{align}\label{lem:korn_1}
    \lVert \nabla \boldsymbol\varphi \rVert_{L^r(\Omega_\eta)}
    \leq \lVert \nabla \boldsymbol\varphi\,d^{1-\beta}
      \rVert_{L^p(\Omega_\eta)}  \lVert d^{(\beta-1)}
        \rVert_{L^\frac{rp}{p-r}(\Omega_\eta)},
  \end{align}
  therefore we only have to estimate the second term on the right-hand 
  side. We set $\tilde r := \frac{pr}{p-r}\in (1,\infty)$ and deduce 
  $\frac{1}{2\tilde r}\in (0, \frac{1}{2})$, i.\,e.\ we can choose 
  $\beta \in (\frac{1}{2},1)$ with $1-\frac{1}{2\tilde r} <\beta 
  <1$. Moreover, for $\varepsilon>0$ we consider the partition 
  $\Omega_\eta = U_\varepsilon \cup V_\varepsilon\cup M_\varepsilon$ with
  \begin{gather*}
    U_\varepsilon :=\big\{ x\in \Omega_\eta \bigm| 
      dist(x,\Gamma_\eta)<\varepsilon\big\},\quad
    V_\varepsilon :=\big\{ x\in \Omega_\eta \bigm| 
      dist(x,\partial\Omega_\eta\setminus\Gamma_\eta)<\varepsilon\big\},\\
    M_\varepsilon :=\big\{ x\in \Omega_\eta \bigm| 
      dist(x,\partial\Omega_\eta)\geq\varepsilon\big\}.
  \end{gather*}
  Then
  \begin{align*}
    \int_{\Omega_\eta}d^{(\beta-1)\tilde r}\;dx
    &\leq \int_{U_\varepsilon}dist(x,\Gamma_\eta)^{(\beta-1)\tilde r}\;dx
    + \int_{V_\varepsilon}dist(x,
      \partial\Omega_\eta\setminus\Gamma_\eta)^{(\beta-1)\tilde r}\;dx\\
    &\qquad + \int_{M_\varepsilon}d^{(\beta-1)\tilde r}\;dx.
  \end{align*}
  We take $\varepsilon$ small enough, such that $V_\varepsilon \subset 
  S_\alpha$. By the assumptions to our in- and outflow domain, $\Gamma_\eta$ 
  consists of the connected components $\Gamma_\eta^i$, $i=1,\ldots, \ell$, 
  each part of some hyperplane. Hence, we can take $\varepsilon>0$ even 
  smaller, such that $U_\varepsilon$ decomposes disjointly into the 
  sets $U_\varepsilon^i$, $i=1,\ldots , \ell$ of the form
  \begin{align*}
    U_\varepsilon^i = \big\{ y\in \R^3 \bigm| y = q - 
    s\,\boldsymbol\nu(q)\; \text{with } q\in \Gamma_\eta^i, 
      \;s\in (0,\varepsilon)\big\}.
  \end{align*}
  Since $(\beta-1)\tilde r \in (-1,0)$, we get the estimate
  \begin{align}\label{lem:korn_2}
    \int_{U_\varepsilon}dist(x,\Gamma_\eta)^{(\beta-1)\tilde r}\;dx
    = \sum_{i=1}^\ell \int_{\Gamma_\eta^i}\int_0^\varepsilon 
      s^{(\beta-1)\tilde r}\;ds\,dq
    \leq c(\ell, \beta, \tilde r, |\Gamma_\eta|)\,
      \varepsilon^{(\beta-1)\tilde r + 1}.
  \end{align}
  The embedding $H^2(M) \hookrightarrow C^{0,\frac{1}{2}}(M)$ and the 
  properties of the square root imply for $q_1,\,q_2\in M$ and $|s|<\kappa$ 
  \begin{align*}
    |\eta(q_1) - s| 
      &\leq |\eta(q_1) - \eta(q_2)| + |\eta(q_2) - s|\\
      &\leq c(\lVert\eta\rVert_{H^2(M)})\,|q_1 - q_2|^\frac{1}{2} 
        + |\eta(q_2) - s|\\
      &\leq c(\kappa,\,\lVert\eta\rVert_{H^2(M)})\,\big(|q_1 - q_2| 
        + |\eta(q_2) - s|\big)^\frac{1}{2}.
  \end{align*}
  By the properties of the tubular neighbourhood we deduce for 
  $q\in M$, $|s|<\kappa$
  \begin{align*}
    |\eta(q)-s|^2 
      \leq c(\kappa,\,\Lambda,\,\lVert\eta\rVert_{H^2(M)})\;
        dist(q + s\,\boldsymbol\nu(q),\partial\Omega_\eta\setminus\Gamma_\eta).
  \end{align*}
  By a change of variables and $(\beta-1)\tilde r \in (-\frac{1}{2},0)$ 
  we get
  \begin{align}\label{lem:korn_3}
    \int_{V_\varepsilon}dist(x,&\partial\Omega_\eta\setminus\Gamma_\eta)
      ^{(\beta-1)\tilde r}\;dx\\
    &\leq \int_M \int_{-\alpha}^{\eta(q)} dist(q+s\,\boldsymbol\nu,
      \partial\Omega_\eta\setminus\Gamma_\eta)^{(\beta-1)\tilde r}|\operatorname{det} d\Lambda|\;ds\,dA(q)\notag\\
    &\leq c(\kappa,\,\Lambda,\,\lVert \eta\rVert_{H^2(M)})\,
      \int_M \int_{-\alpha}^{\eta(q)} |\eta(q)-s|^{2(\beta-1)\tilde r}\;ds\,dA(q)\notag\\
    &\leq c(\kappa,\,\Lambda,\,\lVert \eta\rVert_{H^2(M)})\,
      \int_M \frac{1}{2(\beta-1)\tilde r +1}(\eta(q)+\alpha)^{2(\beta-1)\tilde r+1}\;dA(q)\notag\\
    &\leq c(\alpha,\, \tilde r,\, \beta,\, \kappa,\,\Lambda,\, M,\,\lVert \eta\rVert_{H^2(M)}),\notag
  \end{align}
  where we used the embedding $H^2(M) \hookrightarrow L^\infty(M)$.
  Finally we have  
  \begin{align}\label{lem:korn_4}
    \int_{M_\varepsilon}d^{(\beta-1)\tilde r}\;dx
    \leq \varepsilon^{(\beta-1)\tilde r}|B_\kappa|.
  \end{align}
  From \eqref{lem:korn_1}, \eqref{korn_distance}, \eqref{lem:korn_2}, 
  \eqref{lem:korn_3} and \eqref{lem:korn_4} follows the claim.%
\end{proof}
To close this Section, we define the following Banach spaces.
\begin{definition}\label{def:space_V_p}
  Let $\eta\in H^2(M)$, $\lVert\eta\rVert_{L^\infty(M)} < \alpha 
  < \kappa$ and $1\leq p \leq \infty$. We set
  \begin{align*}
    V_p(\Omega_\eta) &:=\big\{ \mathbf{u}\in L^p(\Omega_\eta) \bigm| 
      \mathbf{D}\mathbf{u} \in L^p(\Omega_\eta),\;\operatorname{div} \mathbf{u} = 0\big\},\\
    \widetilde{V}_p(\Omega_\eta) &:=\big\{ \mathbf{u}\in L^p(\Omega_\eta) \bigm| 
      \mathbf{D}\mathbf{u} \in L^p(\Omega_\eta)\big\}
  \end{align*}
  and equip these spaces with the norm
  \begin{align*}
    \lVert \mathbf{u} \rVert_{V_p(\Omega_\eta)}
      := \lVert \mathbf{u} \rVert_{L^p(\Omega_\eta)}
        + \lVert \mathbf{D}\mathbf{u} \rVert_{L^p(\Omega_\eta)}.
  \end{align*}
\end{definition}
\subsection{Generalised trace operator}
\label{subsec:general_trace}
Let $U$ be an open subset of $\R^3$ and $1\leq p \leq \infty$. Following 
\cite[section II.1.2]{MR1928881}, the Banach spaces
\begin{gather*}
    E^p(U):=\big\{\boldsymbol{\varphi}\in L^p(U)\bigm| 
      \operatorname{div} \boldsymbol{\varphi}\in L^p(U)\big\},\quad 
    L^p_\sigma(U)  :=\big\{\boldsymbol{\varphi}\in E^p(U)\,|\, \operatorname{div} 
\boldsymbol{\varphi}=0\big\}
\end{gather*}
with the norm $\lVert \boldsymbol{\varphi}\rVert_{L_\sigma^p(U)} =
\lVert \boldsymbol{\varphi}\rVert_{E^p(U)} = 
\lVert\boldsymbol{\varphi}\rVert_{L^p(U)}
+ \lVert\operatorname{div}\boldsymbol{\varphi}\rVert_{L^p(U)}$ admit a 
generalised trace operator $\trn$ for the normal component as long as 
$\partial U$ is Lipschitz. In particular, this trace operator is defined 
by an approximation argument through Green's formula from Proposition
\ref{prop:part_int} as an element 
$\trn\mathbf{u}\in (H^{1-1/p',p'}(\partial U))^\ast$, and admits the 
representation
\begin{align*}
  \langle \trn \boldsymbol{\varphi}, g\rangle 
    = \int_{\partial U} \boldsymbol{\varphi}\cdot \boldsymbol\nu \;g\;dA
      \qquad \text{for } \boldsymbol{\varphi}\in 
      C^\infty(\overline{U}),\;g\in H^{1-1/p',p'}(\partial U).
\end{align*}
It should be noted, that by mollification $C_0^\infty(\overline{U})$
is dense in $E^p(U)$ 
as long as $\partial U \in C^0$ (cf.~\cite[Prop. A.1]{MR3147436}). 
To extend the trace operator to our 
deformed domain, we will use Green's formulae from Proposition 
\ref{prop:part_int}.
\begin{proposition}\label{prop:spur_normale}
  Let $\eta\in H^2(M)$, $\lVert\eta\rVert_{L^\infty(M)}< \alpha <\kappa$
  and $1<p,p'<\infty$ with $\frac{1}{p}+ \frac{1}{p'}=1$ and  
  $r,\tilde r$ with $p'< r < \tilde r <\infty$. There exists a linear, 
  continuous operator
  \begin{align*}
    \trneta:E^p(\Omega_\eta)\rightarrow (H^{1-1/r,r}(\partial\Omega))^*
  \end{align*}
  satisfying
  \begin{align*}
    \langle \trneta \boldsymbol\varphi, \treta\psi \rangle
      = \int_{\Omega_{\eta}}\boldsymbol{\varphi}\cdot \nabla \psi\;dx 
        + \int_{\Omega_{\eta}}\operatorname{div} \boldsymbol{\varphi}\; \psi\;dx 
  \end{align*}
  for all $\boldsymbol\varphi\in E^p(\Omega_\eta)$ and all $\psi\in 
  W^{1,\tilde r}(\Omega_\eta)$. For a fixed $N\in \N$, the continuity 
  constant can be chosen uniformly with respect to $\eta\in H^2(M)$ satisfying
  $\lVert \eta \rVert_{H^2(M)}\leq N$ and $\lVert\eta\rVert_{L^\infty(M)}
  <\alpha$.
\end{proposition}
\begin{proof}
  We will define $\trneta \boldsymbol\varphi$ for $\boldsymbol\varphi\in 
  C^1(\overline{\Omega_\eta})$. Then, by density, the claim follows. 
  For $b\in H^{1-1/r,r}(\partial\Omega)$ we set
  \begin{align*}
    \langle \trneta \boldsymbol\varphi, b \rangle 
      := \int_{M} \treta\boldsymbol{\varphi}\cdot \mathbf{v}_\eta\, b\;dA
        + \int_{\Gamma} \treta\boldsymbol{\varphi}\cdot 
          \boldsymbol{\nu}\,b\,|\operatorname{det} d\Phi_{\eta}|\;dA,
  \end{align*}    
  i.\,e.\ $\trneta \boldsymbol\varphi \in (H^{1-1/r,r}(\partial\Omega))^*$. 
  Since $\Omega$ is a bounded Lipschitz domain, there exists a linear, 
  continuous extension operator $F:H^{1-1/r,r}(\partial\Omega) \rightarrow 
  W^{1,r}(\Omega)$ (see \cite[Satz 6.41]{dobrowolski}). By 
  Lemma \ref{lem:einbettung_omega_eta} and $p'< r$, the map $F\circ 
  \Psi_\eta^{-1}: H^{1-1/r,r}(\partial \Omega) \rightarrow 
  W^{1,p'}(\Omega_\eta)$ is linear and continuous with an appropriately 
  bounded continuity constant. Using the definition of $\treta$ and 
  the extension property of $F$, we have 
  \begin{align*}
    \treta ((F\circ \Psi_\eta^{-1})(b))
    = (Fb)\circ \Psi_\eta^{-1}\circ\Psi_{\eta}\lvert_{\partial\Omega}
    = (Fb)\lvert_{\partial\Omega}= b.
  \end{align*}
  Thus, by definition of $\tretan$ and Proposition \ref{prop:part_int} 
  we have  
  \begin{align*}
    \big|\langle \trneta \boldsymbol\varphi, b \rangle\big|
    &= \Bigg| \int_{M} \treta\boldsymbol{\varphi}\cdot 
        \mathbf{v}_\eta\, \treta ((F\circ \Psi_\eta^{-1})(b))\;dA\\
      &\qquad+ \int_{\Gamma} \treta\boldsymbol{\varphi}\cdot 
        \boldsymbol{\nu}\,\treta ((F\circ \Psi_\eta^{-1})(b))\,
          |\operatorname{det} d\Phi_{\eta}|\;dA\,\Bigg|\\
    &= \Bigg|\int_{\Omega_{\eta}}\boldsymbol{\varphi}\cdot 
        \nabla ((F\circ \Psi_\eta^{-1})(b))\;dx 
      + \int_{\Omega_{\eta}}\operatorname{div} \boldsymbol{\varphi}\; 
        ((F\circ \Psi_\eta^{-1})(b))\;dx\,\Bigg|\\
    &\leq c\,\lVert \boldsymbol\varphi\rVert_{E^p(\Omega_\eta)}\,
      \lVert b\rVert_{H^{1-1/r,r}(\partial \Omega)}.
  \end{align*}
  By this inequality, we can extend $\trneta$ continuously to the 
  space $E^p(\Omega_\eta)$, which shows the first part of our claim. 
  For the second part, we observe that $r<\tilde r$ and therefore 
  $\treta\psi \in H^{1-1/r,r}(\partial \Omega)$ for $\psi\in 
  W^{1,\tilde r}(\Omega_\eta)$. Thus, an approximation of 
  $\boldsymbol\varphi$, the definition of $\tretan$ for smooth functions 
  and \eqref{lem:normale_spur-1} shows the desired identity.
\end{proof}
To obtain a trace operator defined only on a part of the boundary, we 
restrict $\trneta$ to the space of test functions vanishing on the rest 
of the boundary. For a measurable subset $\gamma \subset \partial \Omega$ 
and $1<r<\infty$ we therefore set 
$W^{1,r}_{\partial \Omega \setminus \gamma}(\Omega) := \{ u \in 
W^{1,r}(\Omega) \,|\, u\lvert_{\partial\Omega\setminus\gamma} = 0 \}$
and $H^{1-1/r,r}_{00}(\gamma)$ as the image of the classical trace 
operator of $W^{1,r}_{\partial \Omega \setminus \gamma}(\Omega)$.
Hence $H^{1-1/r,r}_{00}(\gamma)$ is a closed subspace of\label{def_testraum_trn}
$H^{1-1/r,r}(\partial \Omega)$.
\begin{definition}\label{bem:def_trace_eta_n_part}
  For a measurable subset $\gamma\subset \partial \Omega$, we set 
  \begin{align*}
    \trneta\big\rvert_\gamma:E^p(\Omega_\eta)\rightarrow (H^{1-1/r,r}_{00}(\gamma))^*,
      \quad \langle \trneta\big\rvert_\gamma \boldsymbol\varphi, b \rangle 
        := \langle \trneta \boldsymbol\varphi, b \rangle
  \end{align*}
  as the restriction of $\trneta$ to $H^{1-1/r,r}_{00}(\gamma)$.
\end{definition}
\begin{remark}\label{bem:kompatibilitaet_spuroperatoren}
  Let $\eta\in H^2(M)$, $\lVert\eta\rVert_{L^\infty(M)}<\kappa$ and 
  $1<p,p'<\infty$ with $\frac{1}{p} + \frac{1}{p'}=1$ and 
  $p'<r<\infty$. Then for all $\boldsymbol\varphi\in 
  W^{1,p}(\Omega_\eta)$ and all $b\in H^{1-1/r,r}(\partial\Omega)$ one has
  \begin{align*}
    \langle \trneta\boldsymbol\varphi, b \rangle 
      = \int_{M} \treta\boldsymbol{\varphi}\cdot \mathbf{v}_\eta\, b\;dA
        + \int_{\Gamma} \treta\boldsymbol{\varphi}\cdot 
          \boldsymbol{\nu}\,b\,|\operatorname{det} d\Phi_{\eta}|\;dA.
  \end{align*}
  For $\boldsymbol\varphi \in W^{1,p}(\Omega_\eta)$ with 
  $\treta\boldsymbol\varphi = \xi\,\boldsymbol\nu$ for some 
  $\xi \in L^p(\partial\Omega)$, Remark \ref{grundlagen:eigenschaften_gamma} 
  implies
  \begin{align*}
    \langle \trneta\boldsymbol\varphi, b \rangle 
      = \int_{M} \xi\,b\,\gamma(\eta)\;dA
        + \int_{\Gamma} \xi\,b\,|\operatorname{det} d\Phi_{\eta}|\;dA,
        \qquad b\in H^{1-1/r,r}(\partial\Omega).
  \end{align*}
  Moreover, for an open subset $V\subset \R^3$ with $\Omega_\eta \subset V$ 
  and $\gamma:= \Phi_\eta^{-1}(\partial \Omega_\eta \setminus \partial 
  V)\subset \partial \Omega$ measurable, a function 
  $\mathbf{u}\in E^p(\Omega_\eta)$ with $\trneta\lvert_\gamma = 0$ 
  can be extended by $\boldsymbol 0$ to $\overline{\mathbf{u}}\in E^p(V)$.
  This can be seen by taking $\varphi\in C_0^\infty(V)$ and observing the 
  identity
  \begin{align*}
    \langle \operatorname{div} \overline{\mathbf{u}},\varphi\rangle
    &= -\int_{\Omega_\eta}\mathbf{u}\cdot \nabla\varphi\;dx\\
    &= -\langle \trneta\lvert_\gamma \mathbf{u},\treta\varphi\rangle 
      + \int_{\Omega_\eta} \varphi\operatorname{div}\mathbf{u}\;dx
    = \int_V \varphi\,(\operatorname{div}\mathbf{u})\,\chi_{U}\;dx.
  \end{align*}
\end{remark}
%
\begin{definition}\label{def:H_sigma_M}
  Let $\eta\in H^2(M)$ with $\lVert\eta\rVert_{L^\infty(M)}<\kappa$
  and $\gamma\subset \partial \Omega$ measurable. We set 
  \begin{align*}
    H_{\gamma}(\Omega_\eta):=
      \big\{ \boldsymbol\varphi \in  L^2_\sigma(\Omega_\eta)
      \bigm| \trneta\lvert_{\gamma}\,\boldsymbol\varphi= \boldsymbol 0
      \big\}
  \end{align*}
  and equip this space with the $L^2(\Omega_\eta)$-norm. Since 
  $H_{\gamma}(\Omega_\eta)$ is a closed subspace of 
  $L^2_\sigma(\Omega_\eta)$ and $L^2(\Omega_\eta)$ respectively, it is 
  a separable Hilbert space.
\end{definition}
This space admits the following density result.
\begin{lemma}\label{lem:density_part_normal_trace}
  Let $\eta\in H^2(M)$ with $\lVert\eta\rVert_{L^\infty(M)}<\alpha 
  <\kappa$. The set 
  \begin{align*}
    {\mathcal{V}}_{M}(\Omega_\eta)
      :=\{ \boldsymbol\varphi\in C_0^\infty(\overline{\Omega_\eta})
        \,|\, \operatorname{div} \boldsymbol\varphi = 0,\,
          \boldsymbol\varphi = \boldsymbol 0 \text{ in a neighbourhood of } \Phi_\eta(M) \}
  \end{align*}
  is dense in $H_{M}(\Omega_\eta)$.
\end{lemma}
\begin{proof}
  By Remark \ref{bem:kompatibilitaet_spuroperatoren} we have
  ${\mathcal{V}}_{M}(\Omega_\eta)\subset H_{M}(\Omega_\eta)$. Let 
  $\mathbf{u}\in H_{M}(\Omega_\eta)$. At first, we will construct a 
  suitable extension of $\mathbf{u}$ which vanishes on the whole 
  boundary. Again by Remark \ref{bem:kompatibilitaet_spuroperatoren}, 
  the extension $\overline{\mathbf{u}}$ by zero gives 
  $\overline{\mathbf{u}}\in L^2_\sigma(B_\alpha)$. Since $B_\alpha$ 
  is an bounded Lipschitz domain, $B_\alpha \subset\subset B$ for some 
  smooth bounded domain $B$, see Figure \ref{fig:extension_over_boundary}. 
  \begin{figure}\centering
  \includegraphics{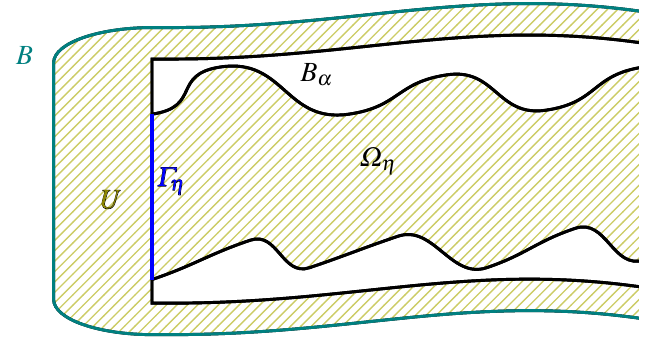}
  \caption{Extension of the fluid domain.}
  \label{fig:extension_over_boundary}
  \end{figure}
  Moreover, by \cite{MR1809290} (with $\lambda_j=0$ since $B_\alpha$ 
  is connected), there exists an extension $\widetilde{\mathbf{u}}\in 
  L^2_\sigma(B)$ of $\overline{\mathbf{u}}$, whose support is compactly 
  contained in $B$. In particular, this function vanishes on the boundary, 
  i.\,e.\ we have
  \begin{align}\label{extension_B_trace}
    \int_B \widetilde{\mathbf{u}} \cdot \nabla \psi \;dx = 0
      \qquad \forall \psi \in W^{1,2}(B).
  \end{align}
  Next, we restrict $\widetilde{\mathbf{u}}$ to the set $U:=B\setminus
  \overline{(B_\alpha\setminus \Omega_\eta)}$ and define 
  $\mathbf{v}:= \left.\widetilde{\mathbf{u}}\right\lvert_{U}$. Therefore 
  $\mathbf{v}\in L_\sigma^2(U)$. Taking $\psi \in W^{1,r}(U)$ (with 
  $2 < r <\infty$ from the definition of $\tretan$), by \cite[Satz 6.10]
  {dobrowolski} we have the extension $E(\psi)\in W^{1,r}(B)$. Hence, 
  \eqref{extension_B_trace} and the extension properties imply
  \begin{align*}
    \int_U \mathbf{v}\cdot \nabla \psi\;dx 
      &= \int_{U\setminus \overline{B_\alpha}} \mathbf{v}\cdot \nabla \psi\;dx 
        + \int_{U\cap B_\alpha} \mathbf{v}\cdot \nabla \psi\;dx\\
      &= \int_{B} \mathbf{v}\cdot \nabla E(\psi)\;dx
        - \int_{B_\alpha} \mathbf{v}\cdot \nabla E(\psi)\;dx 
        + \int_{\Omega_\eta} \mathbf{v}\cdot \nabla \psi\;dx\\
      &= \int_{\Omega_\eta} \mathbf{u}\cdot \nabla \left( \psi - E(\psi) \right)\;dx\\
      &= 0,
  \end{align*}
  since $(\psi - E(\psi))\lvert_{\Gamma_\eta} = 0$. Again this identity 
  can be interpreted as vanishing boundary values for $\mathbf{v}$.
  We set
  \begin{align*}
    F: L^2_\sigma(U) \rightarrow \big(W^{1,r}(U)\big)^\ast,\qquad
      \langle F\mathbf{u}, \varphi\rangle := 
        \int_U \mathbf{u} \cdot \nabla \varphi\;dx.
  \end{align*}
  Then $F$ is linear and continuous, i.\,e.\ $H(U):= 
  \{ \mathbf{u}\in L^2_\sigma(U) \,|\, F\mathbf{u} = 0 \}$ is a closed 
  subspace of $L^2_\sigma(U)$ and therefore -- equipped with the 
  $L^2$ scalar product -- a Hilbert space. We will now show the density of 
  \begin{align*}
    {\mathcal{V}}(U)
      :=\left\{ \boldsymbol\varphi\in C_0^\infty(U)\bigm|
          \operatorname{div} \boldsymbol\varphi = 0 \right\}
  \end{align*}
  in $H(U)$. This implies our claim by restriction of the approximating 
  functions to $\Omega_\eta$. We argue similarly as in 
  \cite[Theorem 1.6]{MR0609732}. Obviously it holds
  ${\mathcal{V}}(U)\subset H(U)$. We will show that every functional
  $G\in H(U)^\ast$ vanishing on ${\mathcal{V}}(U)$ is the zero functional. 
  Hahn-Banach's theorem then implies the density. Let $G\in H(U)^\ast$ 
  with $\langle G, \boldsymbol\varphi\rangle = 0$ for all 
  $\boldsymbol\varphi\in {\mathcal{V}}(U)$. By Riesz theorem, there 
  exists $\mathbf{g}\in H(U)$ with
  \begin{align*}
    \langle G, \mathbf{u}\rangle = \int_U \mathbf{g}\cdot \mathbf{u}\;dx
    \quad\quad\text{for all } \mathbf{u}\in H(U).
  \end{align*}
  In particular, $\int_U \mathbf{g}\cdot \boldsymbol\varphi\;dx 
  = 0$ for all $\boldsymbol\varphi\in {\mathcal{V}}(U)$. By the theorem 
  of De Rham (see \cite[Theorem IV.2.4]{MR2986590}) there exists 
  $p\in L^2_{loc}(U)$ with $\mathbf{g} = \nabla p$, i.\,e.\ $\nabla p \in 
  L^2(U)$. Approximating $p$ through a sequence $p_n\in C^1(\overline{U})$ 
  with $\nabla p_n \rightarrow \nabla p$ in $L^2(U)$, we can deduce for all 
  $\mathbf{u}\in H(U)$ (and therefore $F\mathbf{u} = 0$)
  \begin{align*}
    \langle G, \mathbf{u}\rangle 
      = \int_U \mathbf{g}\cdot \mathbf{u}\;dx
      = \lim_{n\rightarrow \infty}\int_U \nabla p_n \cdot \mathbf{u}\;dx
      = \lim_{n\rightarrow \infty}\langle F\mathbf{u}, p_n \rangle
      = 0,
  \end{align*}
  which proofs the claim.
\end{proof}
\begin{remark}
  Let $\Omega_{\eta-\rho}=(\Omega\setminus 
  S_\kappa\cup \{ x\in S_\kappa\bigm| s(x) < 
  \eta(\mathbf{q}(x)) -\rho\}$. Then $\boldsymbol\varphi\in 
  {\mathcal{V}}_{M}(\Omega_\eta)$ already implies 
  $supp\; \boldsymbol\varphi \subset \Omega_{\eta-\rho}$ for some 
  $0<\rho$ small enough.
\end{remark}
The following result ensures that we can approximate functions of 
$H_M(\Omega_\eta)$, while keeping the support uniformly away from the 
moving boundary. This Lemma adapts \cite[Lemma A.13]{MR3147436} to 
our situation and is crucial for the compactness theorem.
\begin{lemma}\label{lem:density_for_compactness}
  Let $0<\alpha<\kappa$ and $C,\,\varepsilon>0$ be given. Then there exists 
  $\rho>0$ such that for all $\eta\in H^2(M)$ with $\lVert \eta 
  \rVert_{H^2(M)}\leq C$, $\lVert \eta \rVert_{L^\infty(M)}\leq 
  \alpha$ and for all $\boldsymbol\varphi \in H_{M}(\Omega_\eta)$ 
  with $\lVert \boldsymbol\varphi \rVert_{L^2(\Omega_\eta)}\leq 1$ there 
  exists an function $\Psi = \Psi(\eta,\boldsymbol\varphi) \in 
  H_{M}(\Omega_\eta)$ with $supp\;\Psi \subset \Omega_{\eta-\rho}$, $\lVert 
  \Psi \rVert_{L^2({\Omega_\eta})} \leq 2$ and $\lVert\boldsymbol\varphi
  -\Psi\rVert_{(H^\frac{1}{4}(\R^3))^\ast} < \varepsilon$. As usual, 
  $\boldsymbol\varphi-\Psi \in H^{1/4}(\R^3)^\ast$ is realised through 
  the extension by $\boldsymbol 0$ to $\R^3$ and the $L^2$
  Riesz representation.
\end{lemma}
\begin{proof}
  We suppose the claim is false. Hence, we find sequences 
  $(\rho_n)_{n\in\N}\searrow 0$ with $0<\rho_n <\kappa-\alpha$,  
  $\eta_n\in H^2(M)$ with $\lVert \eta_n \rVert_{H^2(M)}\leq C$, 
  $\lVert \eta_n \rVert_{L^\infty(M)}\leq \alpha$ and $\boldsymbol\varphi_n
  \in H_{M}(\Omega_{\eta_n})$ with $\lVert \boldsymbol\varphi_n 
  \rVert_{L^2(\Omega_{\eta_n})}\leq 1$ such that for all
  $\Psi \in H_{M}(\Omega_{\eta_n})$ satisfying $supp\;\Psi \subset 
  \Omega_{\eta_n-\rho_n}$ and\linebreak  $\lVert \Psi \rVert_{L^2(\Omega_{\eta_n})} 
  \leq 2$ we have $\lVert\boldsymbol\varphi_n 
  -\Psi\rVert_{(H^{1/4}(\R^3))^\ast} \geq \varepsilon$. Therefore we find 
  a (not further denoted) subsequence with 
  \begin{alignat*}{4}
    \eta_n &\rightharpoonup \eta &\quad&\text{weakly in } H^2(M),
    \qquad &\boldsymbol\varphi_n &\rightharpoonup \boldsymbol\varphi
    &\quad&\text{weakly in } L^2(\R^3),\\
    \eta_n &\rightarrow \eta &\quad&\text{uniformly in } M.&&
  \end{alignat*}
  In particular, we have $\lVert \eta\rVert_{H^2(M)}\leq C$, $\lVert 
  \eta\rVert_{L^\infty(M)}\leq \alpha$, $\lVert \boldsymbol\varphi
  \rVert_{L^2(\R^3)}\leq 1$ and $\boldsymbol\varphi = \boldsymbol 0$ 
  in $\Omega_\eta^c$. Since the smooth extension by zero of 
  $\mu\in C_0^\infty(\Omega_\eta)$ implies 
  $\tretan \mu \in H^{1-1/r,r}_{00}(M)$, we have by the weak 
  convergence of $\boldsymbol\varphi_n$ and Green's formula from 
  Proposition \ref{prop:spur_normale} 
  \begin{align*}
    \langle \operatorname{div} \boldsymbol\varphi, \mu\rangle
      &= - \lim_{n\rightarrow\infty}\int_{\R^3} 
          \boldsymbol\varphi_n \cdot \nabla \mu\;dx\\
      &= \lim_{n\rightarrow\infty}\Big(\int_{\R^3} 
          \operatorname{div} \boldsymbol\varphi_n\, \mu\;dx
        - \langle \trnetan\boldsymbol\varphi_n,\tretan \mu\rangle\Big)
      = 0,
  \end{align*}
  i.\,e.\ $\boldsymbol\varphi\in L^2_\sigma(\Omega_\eta)$. For  
  $\boldsymbol\varphi\in H_M(\Omega_\eta)$, by density it suffices to 
  show $\langle \trneta\boldsymbol\varphi, b\lvert_{\partial\Omega}\rangle
  =0$ for all $b\in C^\infty(\overline{\Omega})$ vanishing in a 
  neighbourhood of $\Gamma$, see \cite{MR2852305}. For such a function 
  $b$ we get, again by Green's formula from Proposition
\ref{prop:part_int},
  \begin{align*}
    \langle \trneta \boldsymbol\varphi, b\vert_{\partial\Omega}\rangle
    = \langle \trneta \boldsymbol\varphi, \treta(b\circ\Psi_\eta^{-1})\rangle
    = \int_{\Omega_{\eta}}\boldsymbol{\varphi}\cdot
        \nabla (b\circ\Psi_\eta^{-1})\;dx.
  \end{align*}
  We would like to pass to the sequence $\boldsymbol\varphi_n$ on the 
  right hand side, but first we have to extend $b\circ\Psi_\eta^{-1}$ 
  appropriately.  Using the operator from 
  \cite[Satz 6.10]{dobrowolski}, which is locally defined as a reflection, 
  we can extend $b\in C^\infty(\overline{\Omega})\subset W^{1,r_1}(\Omega)$ 
  to $\tilde b \in W^{1,r_1}(B_\alpha)$ satisfying $\tilde b=0$ on 
  $\Gamma\cup \Gamma_s^\alpha$ for some $r<r_1<\infty$. Furthermore, by 
  a change of variables through the inverse Hanzawa transform (which can 
  be extended smoothly using the same construction on the set $B_\alpha 
  \setminus \overline{\Omega_\eta}$), we obtain a function 
  $\overline{b}\in W^{1,r}(B_\alpha)$ satisfying $(b\circ\Psi_\eta^{-1})
  = \overline{b}$ in $\overline{\Omega_\eta}$ and 
  $\tretan\lvert_{\Gamma} \overline{u}= 0$. Hence, we have 
  \begin{align*}
    \langle \trneta \boldsymbol\varphi, b\vert_{\partial\Omega}\rangle
      = \int_{\Omega_{\eta}}\boldsymbol{\varphi}\cdot 
          \nabla \overline{b}\;dx
      &= \lim_{n\rightarrow\infty}\int_{\R^3}\boldsymbol\varphi_n \cdot 
          \nabla\overline{b}\;dx\\
      &= \lim_{n\rightarrow\infty} \langle \trnetan\lvert_{M} 
          \boldsymbol\varphi_n, \tretan \overline{b}\rangle
      = 0,
  \end{align*}
  i.\,e.\ $\boldsymbol\varphi\in H_{M}(\Omega_\eta)$. By 
  Lemma \ref{lem:density_part_normal_trace} there exists $\Psi\in 
  {\mathcal{V}}_{M}(\Omega_\eta) \subset H_{M}(\Omega_\eta)$ 
  with $supp\; \Psi \subset \Omega_{\eta-\rho_0}$ for some 
  $0<\rho_0<\kappa-\alpha$ and $\lVert \boldsymbol\varphi 
  - \Psi \rVert_{L^2(\Omega_\eta)}<\min\{\frac{1}{2},
  \frac{\varepsilon}{2}\}$, i.\,e.\ $\lVert \Psi \rVert_{L^2(\R^3)} \leq 2$. 
  Using the uniform convergence, it follows that $supp\; \Psi \subset 
  \Omega_{\eta_n - \rho_n}$ for $n$ big enough. By our assumption 
  we have
  \begin{align}\label{eqn:uniform_density1}\begin{aligned}
    \varepsilon 
      \leq \lVert\boldsymbol\varphi_n 
          - \Psi\rVert_{(H^\frac{1}{4}(\R^3))^\ast}
      &\leq \lVert \boldsymbol\varphi_n 
          -\boldsymbol\varphi\rVert_{(H^\frac{1}{4}(\R^3))^\ast}
        + \lVert \boldsymbol\varphi -\Psi\rVert_{(H^\frac{1}{4}(\R^3))^\ast}\\
      &\leq \lVert \boldsymbol\varphi_n 
          - \boldsymbol\varphi\rVert_{(H^\frac{1}{4}(\R^3))^\ast}
        + \lVert \boldsymbol\varphi -\Psi\rVert_{L^2(\R^3)}\\
      &< \lVert \boldsymbol\varphi_n 
          - \boldsymbol\varphi\rVert_{(H^\frac{1}{4}(\R^3))^\ast}
        + \frac{\varepsilon}{2}.
  \end{aligned}\end{align}
  Choosing a ball $B$ big enough, the compact embedding 
  $H^\frac{1}{4}(B) \hookrightarrow\hookrightarrow L^2(B)$ and the 
  theorem of Schauder imply the compactness of $L^2(B)
  \hookrightarrow\hookrightarrow H^\frac{1}{4}(B)^\ast$. Hence, 
  $\boldsymbol\varphi_n \rightarrow \boldsymbol\varphi$ in 
  $H^\frac{1}{4}(B)^\ast$. Using the extensions by zero, this convergence 
  holds also in $(H^\frac{1}{4}(\R^3))^\ast$, i.\,e.\ 
  $\lVert \boldsymbol\varphi_n - 
  \boldsymbol\varphi\rVert_{(H^\frac{1}{4}(\R^3))^\ast}
  <\frac{\varepsilon}{2}$ for $n$ big enough, a contradiction to 
  \eqref{eqn:uniform_density1}.
\end{proof}
\subsection{Time-dependent function spaces}\label{sec:time_spaces}
We will use an obvious substitute for the classical Bochner spaces in 
our moving domain. For $I:=(0,T)$ with $0<T<\infty$ and 
a continuous $\eta:\overline{I}\times M \rightarrow (-\kappa,\kappa)$ 
we define the bounded domain 
$\Omega_\eta^I:=\cup_{t\in I}\{t\}\times \Omega_{\eta(t)}$ and set for 
$1\leq p,r\leq \infty$
\begin{align*}
  L^p(I;L^r(\Omega_{\eta(t)})) 
    &:= \left\{ v\in L^1(\Omega_\eta^I)\;\left|\;
      \begin{aligned} 
        &v(t,\cdot)\in L^r(\Omega_{\eta(t)}) 
          \text{ for almost all } t\in I,\\ 
        &\lVert v(t,\cdot) \rVert_{L^r(\Omega_{\eta(t)})}\in L^p(I)
      \end{aligned}\right.\right\},\\
  L^p(I;W^{1,r}(\Omega_{\eta(t)})) 
    &:= \Big\{ v\in L^p(I;L^r(\Omega_{\eta(t)}))\bigm| 
      \nabla v\in L^p(I;L^r(\Omega_{\eta(t)}))\Big\},\\
  H^{1}(I;L^p(\Omega_{\eta(t)})) 
    &:= \Big\{ \mathbf{v}\in L^2(I;L^p(\Omega_{\eta(t)}))\bigm| 
      \partial_t \mathbf{v}\in L^2(I;L^p(\Omega_{\eta(t)}))\Big\},\\
  L^p(I;\widetilde{V}_{r}(\Omega_{\eta(t)})) 
    &:= \Big\{ \mathbf{v}\in L^p(I;L^r(\Omega_{\eta(t)}))\bigm| 
      \mathbf{D} \mathbf{v}\in L^p(I;L^r(\Omega_{\eta(t)}))\Big\},\\
  L^p(I;V_{r}(\Omega_{\eta(t)})) 
    &:= \Big\{ \mathbf{v}\in L^p(I;\widetilde{V}_{r}(\Omega_{\eta(t)})) 
      \bigm| \operatorname{div} \mathbf{v}=0\Big\}.
\end{align*}
It should be noted, that $\nabla$, $\operatorname{div}$ and $\mathbf{D}$ are acting  
only with respect to the space variable and the derivatives are (partial) 
weak derivatives on $\Omega_\eta^I$. These spaces, equipped 
with the canonical norms, are Banach spaces. Moreover, motivated by the 
formal a priori estimate, we define\label{definition_raum_auslenkungen} 
the following spaces 
\begin{align*}
  Y^I &:= W^{1,\infty}(I;L^2(M))\cap L^\infty(I;H_0^2(M)),\\
  \widetilde{Y}^I &:= W^{1,\infty}(I;L^2(M))\cap L^\infty(I;H^2(M))
\end{align*}
and for $\eta\in \widetilde{Y}^I$ with 
$\lVert \eta \rVert_{L^\infty(I\times M)} <\kappa$ 
\begin{align*}
  X^I_{\eta} &:= L^\infty(I;L^2(\Omega_{\eta(t)}))\cap
    L^2(I;V_2(\Omega_{\eta(t)})),\\
  \widetilde{X}^I_{\eta} &:= L^\infty(I;L^2(\Omega_{\eta(t)}))\cap
    L^2(I;\widetilde{V}_2(\Omega_{\eta(t)})).
\end{align*}
It means that we renounce the vanishing boundary values for $\widetilde{Y}^I$ 
and the divergence constraint for $\widetilde{X}^I_{\eta}$.

For $\eta \in C^0(\overline{I}\times M)$ and an appropriate cut-off 
function $\beta$, applying the (stationary) Hanzawa transformation 
$\Psi_{\eta(t)}$ at every time $t\in I$ also defines a map between 
$I\times \Omega$ and $\Omega_\eta^I$. Clearly, we have the following 
result.
\begin{proposition}\label{prop:hanzawa_instationaer}
  Let $\eta \in C^0(\overline{I}\times M)$ with $\lVert \eta 
  \rVert_{L^\infty(I\times M)} < \kappa$. Then
  \begin{align*}
    &&&&&\Psi_\eta:\overline{I}\times \overline{\Omega}
        \rightarrow \overline{\Omega_\eta^I}, 
      &&(t,x)\mapsto (t,\Psi_{\eta(t)}(x)),&&&&\\
    &&&&&\Phi_\eta:\overline{I}\times \partial \Omega
        \rightarrow\cup_{t\in\overline{I}}\{t\}\times 
            \partial\Omega_{\eta(t)}, 
    &&(t,x)\mapsto (t,\Phi_{\eta(t)}(x)),&&&&
  \end{align*}
  are homeomorphisms and $C^k$-diffeomorphisms for 
  $\eta\in C^k(\overline{I}\times M)$, $k\in \{1,2,3\}$.
\end{proposition}
As shown in \cite{MR3147436}, $\widetilde{Y}^I$ embeds into a space of 
H\"older-continuous functions.
\begin{lemma}\label{lem:hoelder-einbettung}
  For $\frac{1}{2}<\lambda < 1$ the following embeddings are 
  continuous
  \begin{align*}
    \widetilde{Y}^I \hookrightarrow C^{0,1-\lambda}(\overline{I};H^{2\lambda}(M))
      \hookrightarrow C^{0,1-\lambda}(\overline{I};C^{0,2\lambda -1}(M))
      \hookrightarrow\hookrightarrow C^0(\overline{I}\times M).
  \end{align*}
\end{lemma}
As in the stationary case, the Hanzawa transform preserves convergence. 
%
\begin{lemma}\label{lem:inst.konv_hanzawa}
  Let $(\eta_n)_{n\in\N}\subset \widetilde{Y}^I$ with $\lVert \eta_n 
  \rVert_{L^\infty(I\times M)}\leq \alpha <\kappa$ be a bounded sequence and 
  $\eta_n \rightarrow \eta$ uniformly in $\overline{I}\times M$. Moreover, 
  let the cut-off function $\beta$ be chosen uniformly for the corresponding 
  Hanzawa transforms. Then
  \begin{enumerate}[label=\alph*)]
    \item $\Psi_{\eta_n}\rightarrow \Psi_\eta$ uniformly 
      in $\overline{I}\times \overline{\Omega}$.
    \item $\Psi_{\eta_n}^{-1}\rightarrow \Psi_\eta^{-1}$ uniformly in 
      $\overline{I}\times \overline{B_\alpha}$, where $\Psi_{\eta_n}^{-1}$, 
      $\Psi_{\eta}^{-1}$ are extended by $\mathbf{q}$ to 
      $\overline{I}\times \overline{B_\alpha}$.
    \item Let $1\leq s <\infty$. Then $\nabla \Psi_{\eta_n}$
      converges towards $\nabla\Psi_\eta$ in $L^s(I\times \Omega)$, and 
      $(\nabla\Psi_{\eta_n}^{-1})\,\chi_{\Omega_{\eta_n}}$ towards 
      $(\nabla\Psi_{\eta}^{-1})\,\chi_{\Omega_{\eta}}$ in 
      $L^s(I\times B_\alpha)$. Moreover, we have 
      $\det d\Psi_{\eta_n} \rightarrow \det d\Psi_\eta$ in 
      $L^s(I\times \Omega)$ and $\det d\Psi_{\eta_n}^{-1}\,\chi_{\Omega_{\eta_n}}
      \rightarrow \det d\Psi_\eta^{-1}\,\chi_{\Omega_\eta}$ in $L^s(I\times B_\alpha)$.
  \end{enumerate}
\end{lemma} 
\begin{proof}
  First we remark that using weak and weak-* convergent subsequences, 
  we can deduce $\eta \in \widetilde{Y}^I$ with 
  $\lVert \eta \rVert_{L^\infty(I\times M)}\leq \alpha$. The first two 
  claims follow from the uniform convergence of $\eta_n$ and the 
  definition of the Hanzawa transform (one should note that for the 
  inverse functions, a case differentiation has to be made, see 
  \cite[Lemma 2.13]{phdeberlein}). For $s\geq 2$, the embeddings $H^2(M) 
  \hookrightarrow\hookrightarrow W^{1,s}(M) \hookrightarrow\hookrightarrow 
  L^2(M)$ are continuous and compact. By the definition of $\widetilde{Y}^I$ 
  and Aubin-Lions Lemma, $\eta_n$ converges strongly towards $\eta$ in 
  $L^s(I;W^{1,s}(M))$. Taking into account the characterisations of the 
  (spatial) Jacobians \eqref{form_jacobian} of the transforms, the 
  claimed convergences follow.
\end{proof}
\begin{remark}
  We can transfer the results of the preceding Sections to the 
  time-dependent case by either repeating the proofs using 
  Proposition \ref{prop:hanzawa_instationaer} and Lemma 
  \ref{lem:inst.konv_hanzawa} or applying them at every time 
  $t\in I$. In particular, for $\eta \in C^2(I\times M)$ the 
  corresponding time-dependent Piola transform $\Teta$ is an isomorphism 
  between the Lebesgue and Sobolev spaces on $I\times \Omega$ respectively 
  $\Omega_\eta^I$, as long as the order of differentiability is not 
  larger than one. 
\end{remark}
The next Lemma shows that our spaces are closed under the compatibility
condition $\treta \mathbf{u} = \partial_t\eta\,\boldsymbol\nu$ on 
$I\times M$ in a suitable sense.
\begin{lemma}\label{moving_boundary_data}
  Let $(\delta_n)_{n\in \N}\subset \widetilde{Y}^I$,  
  $(\eta_n)_{n\in \N}\subset  Y^I$ and $(\mathbf{u}_n)_{n\in\N}\subset 
  X^I_{\delta_n}$ be bounded sequences with $\lVert \delta_n 
  \rVert_{L^\infty(I\times M)} \leq \alpha<\kappa$, 
  $\trdeltan\mathbf{u}_n = \partial_t \eta_n\,\boldsymbol\nu$ on $I\times M$ 
  and
  \begin{align*}\begin{alignedat}{4}
    \delta_n  &\rightarrow \delta 
      &\quad&\text{uniformly in } I \times M,
    \qquad&\mathbf{u}_n &\rightharpoonup \mathbf{u} 
      &\quad&\text{weakly in } L^2(I,L^2(\R^3)),\\
    \partial_t\eta_n &\rightharpoonup \partial_t\eta 
      &\quad&\text{weakly in } L^2(I,L^2(M))&&
  \end{alignedat}\end{align*}
  for some $\delta \in \widetilde{Y}^I$, $\eta\in Y^I$, 
  $\mathbf{u}\in X^I_\delta$, where $\mathbf{u}_n$, $\mathbf{u}$ are 
  spatially extended by zero. Then $\trdelta\mathbf{u} = 
  \partial_t\eta\,\boldsymbol\nu$ on $I\times M$.
\end{lemma}
\begin{proof}
  Again we choose the cut-off function $\beta$ for the Hanzawa transforms 
  uniformly. Let $\boldsymbol\varphi \in L^3(I\times \Omega)$. By H\"older's 
  inequality and a change of variables, we get
  \begin{align}\label{eqn:weak_convergence_compat}\begin{aligned}
    \int_I\int_\Omega &\big(\mathbf{u}_n\circ \Psi_{\delta_n}
      - \mathbf{u}\circ \Psi_{\delta}\big)\cdot\boldsymbol\varphi\;dx\,dt\\
    &= \int_I\int_\Omega \big(\mathbf{u}_n - \mathbf{u}\big)
      \circ \Psi_{\delta_n}\cdot\boldsymbol\varphi\;dx\,dt
      + \int_I\int_\Omega \big(\mathbf{u}\circ \Psi_{\delta_n} 
        - \mathbf{u}\circ \Psi_{\delta}\big)\cdot\boldsymbol\varphi\;dx\,dt\\
    &\leq \int_I\int_{\Omega_{\delta_n(t)}} \big(\mathbf{u}_n - \mathbf{u}\big)
      \cdot(\boldsymbol\varphi\circ \Psi_{\delta_n}^{-1})\,
        |\operatorname{det} d\Psi_{\delta_n}|\;dx\,dt\\
      &\qquad + \lVert \mathbf{u}\circ \Psi_{\delta_n} - 
        \mathbf{u}\circ \Psi_{\delta}\rVert_{L^{3/2}(I\times\Omega)}
        \lVert \boldsymbol\varphi\rVert_{L^3(I\times\Omega)}.
  \end{aligned}\end{align}
  By Lemma \ref{lem:inst.konv_hanzawa} and the time-dependent variant 
  of Lemma \ref{lem:konv_trafo}, we have $\mathbf{u}\circ \Psi_{\delta_n} 
  \rightarrow \mathbf{u}\circ \Psi_{\delta}$ in $L^{3/2}(I\times \Omega)$ 
  and $(\boldsymbol\varphi\circ \Psi_{\delta_n}^{-1})\,|\operatorname{det} 
  d\Psi_{\delta_n}|\rightarrow (\boldsymbol\varphi\circ \Psi_{\delta}^{-1})\,
  |\operatorname{det} d\Psi_{\delta}|$ in $L^2(I\times \R^3)$. Together 
  with the weak convergence of $\mathbf{u}_n$, \eqref{eqn:weak_convergence_compat} 
  implies the weak convergence of $\mathbf{u}_n\circ \Psi_{\delta_n}$ 
  towards $\mathbf{u}\circ \Psi_{\delta}$ in $L^{3/2}(I\times \Omega)$. 
  Furthermore, by our Korn-type inequality and the time-dependent 
  version of Lemma \ref{lem:einbettung_omega_eta}, $\mathbf{u}_n\circ 
  \Psi_{\delta_n}$ is bounded uniformly in $L^2(I; W^{1,5/3}(\Omega))$, 
  i.\,e.\ $\mathbf{u}_n\circ \Psi_{\delta_n}$ converges weakly towards 
  $\mathbf{u}\circ \Psi_{\delta}$ in $L^2(I,W^{1,5/3}(\Omega))$. By 
  linearity and continuity of the trace operator and the 
  definition of $\treta$ the sequence $\trdeltan\mathbf{u}_n$ 
  converges weakly towards $\trdelta\mathbf{u}$ in $L^2(I\times M)$. 
  Hence, the compatibility condition $\trdeltan\mathbf{u}_n= 
  \partial_t \eta_n\,\boldsymbol\nu$ on $I\times M$ and the weak 
  convergence of $\partial_t\eta_n$ imply $\trdelta\mathbf{u} = 
  \partial_t \eta\,\boldsymbol\nu$ on $I\times M$.
\end{proof}
\begin{remark}\label{grundlagen:Spuroperator_Raum_Testfkt}
  In the usual Bochner spaces, control over the (generalised) time 
  derivative implies continuity in time, i.\,e.\ for $\mathbf{u}\in 
  H^1(I;L^2(\Omega))$ the pointwise evaluation in time 
  $\mathbf{u}(t,\cdot)\in L^2(\Omega)$ is well defined. In our setting, 
  i.\,e.\ $\eta \in \widetilde{Y}^I$ with $\lVert \eta 
  \rVert_{L^\infty(I\times M)} \leq \alpha < \kappa$, a function 
  $\boldsymbol\varphi\in H^1(I;L^2(\Omega_{\eta(t)})) \cap 
  L^2(I;W^{1,2}(\Omega_{\eta(t)}))$ satisfying $\treta \boldsymbol\varphi 
  = b\,\boldsymbol\nu$ on $I\times M$ for some $b\in H^1(I;L^{2}(M))$ can 
  be extended by $(b\,\boldsymbol{\nu})\circ \mathbf{q}$ to obtain a 
  function in the Bochner space $H^1(I;L^2(B_\alpha))$ (see 
  \cite[Lemma 2.72]{phdeberlein}, \cite[Remark A.14]{MR3147436}). Hence, 
  the evaluation in time is at least for such functions well defined.
\end{remark}
\subsection{Regularization of the displacement}
To avoid the usual loss of regularity by a transformation to and from the 
reference domain, we construct a regularisation of the displacements. 
Since the initial data has to be adapted to the regularised moving domain, 
special care has to be taken. That means, we use a special mollification 
kernel and approximate from ``above''. Hence, our regularisation cannot 
be linear and does not preserve zero boundary conditions.

Let  $(\varphi_k,U_k)_{k=1}^{N}$ be a finite atlas of $M$ with 
subordinate partition of unity $(\psi_k)_{k=1}^{N}$. We extend a given 
$\delta \in C^0(\overline{I}\times M)$ constantly by $\delta(0,\cdot)$ and 
$\delta(T,\cdot)$ to $(-\infty,0)\times M$ and $(T,\infty)\times M$, 
respectively. By the generalised reflection 
$E_k(\delta\circ\varphi_k^{-1})$ from \cite[Satz 6.10]{dobrowolski}, we 
extend $\delta\circ \varphi_k^{-1}$ further to $\R^3$ (note that without 
loss of regularity $\varphi_k(U_k)$ is smooth). Let $\omega\in 
C_0^\infty(\R^3)$ with $\omega\geq 0$, 
$\int_{-\infty}^\infty\int_{\R^2}\omega(t,z)\,dz\,dt 
=1$ and $supp\,\omega \subset \left\{(t,z)\in \R\times \R^2 \,|\,
0<t<1,\, \left|z\right|< 1\right\}$. Moreover, for $\epsilon >0$ we set
$\omega_\epsilon :=\epsilon^{-3}\omega(\epsilon^{-1}\cdot)$ and 
$\randreg\delta:\overline{I}\times M\rightarrow\R$, 
\begin{align*}
  \randreg \delta(t,x):=\sum_{k=1}^N 
    \Big(w_\epsilon\ast E_k(\delta\circ\varphi_k^{-1})
      \Big)\circ \varphi_k(t,x)\,\psi_k(x) + \epsilon^\frac{1}{2}.
\end{align*}
In here, the summand with index $k$ is extended by $0$ to $M$. By basic 
properties of the mollification (see e.\,g.\ \cite[Proposition II.2.25]
{MR2986590}) and the reflection we get:
\begin{proposition}\label{prop:reg_rand:wohldef}
  Let $\epsilon>0$. The map $\randreg: C^0(\overline{I}\times M) 
  \rightarrow C^4(\overline{I}\times M)$, $\delta \mapsto \randreg \delta$
  is continuous and satisfies for all $\delta_1,\,
  \delta_2\in C^0(\overline{I}\times M)$ the estimates
  \begin{gather*}
    \lVert \randreg\delta_1 \rVert_{C^4(\overline{I}\times M)}
      \leq c_\epsilon\,\lVert \delta_1 \rVert_{C^0(\overline{I}\times M)}
        + \epsilon^\frac{1}{2},\qquad
    \rVert \randreg\delta_1 \rVert_{L^\infty(I\times M)}
      \leq \lVert \delta_1 \rVert_{L^\infty(I\times M)} 
        + \epsilon^\frac{1}{2},\\
    \rVert \randreg\delta_1 - \randreg\delta_2\rVert_{L^\infty(I\times M)}
      \leq \lVert \randreg(\delta_1 - \delta_2)\rVert_{L^\infty(I\times M)}
          + \epsilon^\frac{1}{2}.
  \end{gather*}
\end{proposition}
\begin{remark}\label{bem:randreg}
  Let $(\delta_n)_{n\in\N} \subset \widetilde{Y}^I$ be a bounded sequence.
  Then $\randreg\delta_n$ is also bounded in $\widetilde{Y}^I$ 
  independently of $0<\epsilon<\epsilon_0$ and $n\in\N$ (see 
  \cite[Th{\'e}or\`{e}me 1.8.1]{droniou}) and \cite[Proposition II.2.25]
  {MR2986590}), but does not converge in this space.
\end{remark}
%
Now we show the effect of the special mollification kernel
and the translation by $\epsilon^\frac{1}{2}$.
\begin{proposition}\label{prop:reg_rand:anfangswert}
  Let $\eta_0\in H_0^2(M)$ and $\delta \in 
  C^0(\overline{I}\times M)$ with $\delta(0,\cdot) = \eta_0(\cdot)$. 
 \begin{enumerate}[label=\alph*)]
  \item $\randreg\delta(0,\cdot)$ is independent of $\delta$ apart from 
    $\delta(0,\cdot)=\eta_0$.
  \item There exits $0<\epsilon_1=\epsilon_1(\eta_0)$ such that
    $\randreg\delta(0,\cdot) > \eta_0$ for all $0<\epsilon<\epsilon_1$.
 \end{enumerate}
\end{proposition}
\begin{proof}
  By the extension of $\delta$ through $\delta(0,\cdot) = \eta_0$ to 
  $(-\infty,0)\times M$ and the properties of the kernel $w_\epsilon$, 
  we have for $z\in \varphi_k(U_k)$
  \begin{align*}
    (w_\epsilon\ast E_k(\delta\circ\varphi_k^{-1}))(0,z)
      &= \int_\infty^\infty \int_{\R^2} w_\epsilon(s,y)\,
        E_k(\delta\circ\varphi_k^{-1})(0-s,z-y)\;dy\,ds\\
      &= \int_0^\epsilon \int_{B_\epsilon(0)} w_\epsilon(s,y)\,
        E_k(\eta_0\circ\varphi_k^{-1})(z-y)\;dy\,ds.
  \end{align*}
  Hence, $\randreg\delta(0,\cdot)$ is independent of $\delta$ apart from 
  $\delta(0,\cdot)=\eta_0$. By the continuous embedding $H_0^2(M) 
  \hookrightarrow C^{0,\frac{3}{4}}(M)$ and the properties of the 
  generalised reflection we have $E_k(\eta_0\circ\varphi_k^{-1})\in 
  C^{0,\frac{3}{4}}(\R^2)$, i.\,e.\ 
  \begin{align*}
    \lvert E_k(\eta_0\circ\varphi_k^{-1})(z_1) - 
      E_k(\eta_0\circ\varphi_k^{-1})(z_2)\rvert
        \leq c_k\,\lvert z_1 - z_2 \rvert^\frac{3}{4}
          \qquad \forall\,z_1,\,z_2\in \R^2.
  \end{align*}
  Together with the identity above, this implies for $z\in \varphi_k(U_k)$
  \begin{align*}
    |(w_\epsilon&\ast E_k(\delta\circ\varphi_k^{-1}))(0,z)
      -\eta_0\circ\varphi_k^{-1}(z)|\\
    &= \Big|\int_0^\epsilon \int_{B_\epsilon(0)} w_\epsilon(s,y)\,\Big(
        E_k(\eta_0\circ\varphi_k^{-1})(z-y) - 
          E_k(\eta_0\circ\varphi_k^{-1})(z)\Big)\;dy\,ds\Big|\\
    &\leq c_k\,\int_0^\epsilon \int_{B_\epsilon(0)} 
        w_\epsilon(s,y)\,|y|^\frac{3}{4}\;dy\,ds\\
    &\leq c_k\,\epsilon^\frac{3}{4}.
  \end{align*}
  Hence, $\randreg\delta(0,\cdot) - \eta_0 \geq \epsilon^\frac{1}{2} - 
  \sum_{k=1}^N c_k\,\epsilon^\frac{3}{4}$, i.\,e. there exists $0<\epsilon_1 
  = \epsilon_1(c_1,\ldots,c_k)$ with $\randreg\delta > \eta_0$ for 
  $0<\epsilon<\epsilon_1$.
\end{proof}
In the definition of $\randreg \delta$, only $\delta$ and 
$w_\epsilon$ depend (non trivially) on time. Thus, the classical properties of 
mollification and reflection imply the following convergences.
\begin{lemma}\label{lem:reg_rand:konv}
  Let $\delta \in C^0(\overline{I}\times M)$ and $\epsilon>0$. Then
  \begin{enumerate}[label=\alph*)]
    \item $\randreg \delta\rightarrow\delta$ uniformly on 
      $\overline{I}\times M$ for $\epsilon \rightarrow 0$.
    \item If $\partial_t \delta \in L^2(I\times M)$ then 
      $\partial_t\randreg\delta \rightarrow \partial_t\delta$ 
      in $L^2(I\times M)$ for $\epsilon\rightarrow 0$.
    \item Let $\eta_0\in H_0^2(M)$ and $\delta\in 
      C^0(\overline{I}\times M)$ with $\delta(0,\cdot) = \eta_0$.
      Then $\randreg \delta(0,\cdot)$ converges uniformly on $M$ towards 
      $\eta_0$ for $\epsilon \rightarrow 0$. This convergence is independent 
      of the particular choice of $\delta$.
  \end{enumerate}
\end{lemma}
%
%
\subsection{Divergence-free extension operator}
\label{section:div_fortsetzung}
%
%
%
To extend a test function to the shell equation defined on $M$ to a 
test function of the fluid equation, we have to construct a 
divergence-free extension. As in our moving domains, we first look at 
the situation in the stationary case.
%
%
\begin{lemma}\label{lem:fortsetzung_stationaer}
  Let $\eta\in H^2(M)$, $0<\alpha<\kappa$ with $\lVert \eta 
  \rVert_{L^\infty(M)} < \alpha$ and $\frac{3}{2}\leq p \leq 3$.
  Then there exists a linear, continuous extension operator
  \begin{align*}
    \feta: W^{1,p}_0(M) \rightarrow W^{1,p}(B_\alpha)
  \end{align*}
  which satisfies $\operatorname{div} \feta b = 0$ and $\treta(\feta b\lvert_{\Omega_\eta}) 
  =  b\, \boldsymbol\nu$ on $M$. For a fixed $N\in \N$, the continuity 
  constant can be chosen uniformly with respect to $\eta\in H^2(M)$ satisfying
  $\lVert \eta \rVert_{H^2(M)}\leq N$ and $\lVert\eta\rVert_{L^\infty(M)}
  <\alpha$.
\end{lemma}
\begin{proof}
Let $b\in W^{1,p}_0(M)$. We define for $x\in S_\alpha$
\begin{align*}
  (\feta b)(x) := \exp \Big( -\int_{\eta(\mathbf{q}(x))}^{s(x)} 
    \big(\operatorname{div} (\boldsymbol\nu\circ \mathbf{q})\big)(\mathbf{q}(x)
      +\tau \boldsymbol\nu(\mathbf{q}(x)))\;d\tau \Big)\, 
        (b\,\boldsymbol\nu)(\mathbf{q}(x))
\end{align*}
and get with the same arguments as in \cite[Proposition 2.19]{MR3147436}, 
i.\,e.\ using properties of the tubular neighbourhood, that $\feta$ 
is a linear and continuous operator from $W_0^{1,p}(M)$ to 
$W^{1,p}(S_\alpha)$ satisfying $\operatorname{div} \feta b = 0$. In particular, the 
continuity constant is bounded as claimed in the Lemma and we have the 
identity
\begin{align}\begin{aligned}\label{eqn:characterisation_spatial_derivative}
  \partial_i \feta b 
    &= \exp \Big( -\int_{\eta\circ \mathbf{q}}^{s} 
      \operatorname{div} (\boldsymbol\nu\circ \mathbf{q})\circ\big(\mathbf{q}+
        \tau (\boldsymbol\nu \circ \mathbf{q})\big)\;d\tau \Big) 
      \Bigg[ \partial_i\big((b\,\boldsymbol\nu)\circ \mathbf{q}\big)\\
    &\quad + (b\,\boldsymbol\nu)\circ \mathbf{q}\, \Bigg( 
      -\int_{\eta\circ \mathbf{q}}^s \partial_i\Big(\operatorname{div} 
        (\boldsymbol\nu\circ \mathbf{q})\circ\big(\mathbf{q}+\tau 
          (\boldsymbol\nu \circ \mathbf{q})\big)\Big)\;d\tau\\
    &\qquad + \operatorname{div}(\boldsymbol\nu \circ \mathbf{q})\circ 
      \big(\mathbf{q}+ (\eta\,\boldsymbol\nu )\circ \mathbf{q}\big)
        \partial_i(\eta\circ \mathbf{q})\\
    &\qquad\quad - \operatorname{div}(\boldsymbol\nu \circ 
          \mathbf{q})\circ \big(\mathbf{q}+ s (\boldsymbol\nu
            \circ \mathbf{q})\big)\partial_i s \Bigg)\Bigg].
\end{aligned}\end{align}
Since 
$\Omega$ is an admissible in- and outflow domain, the boundary
of $S_\alpha$ is given by the disjoint sets $M^\alpha_-$, $M^\alpha_+$ 
and $\Gamma_s^\alpha$, where the common boundary of $S_\alpha$ and 
$\Omega \setminus S_\alpha$ is given by $M^\alpha_{-}$, see Figure 
\ref{fig:randstueke}. Since $S_\alpha$ is an Lipschitz domain, an 
approximation argument shows $\feta b = 0$ on $\Gamma_s^\alpha$ and 
\begin{align}\label{eqn:trace_feta}
  \left.(\feta b\circ \Psi_\eta)\right|_{M} = b\,\boldsymbol\nu
    \text{ on } M.
\end{align}
Moreover, by the assumptions on the domain, $int\,\Gamma_s^{\alpha,c} 
= \partial (\Omega\setminus\overline{S_\alpha})\setminus M_-^\alpha$) is 
not empty. Hence, there exists an open set $B\subset\subset 
\Gamma_s^{\alpha,c}$ and a function
$\mu:\partial(\Omega\setminus S_\alpha) \rightarrow \R$ which is smooth 
on $B$, $supp\;\mu \subset\subset  B$ 
and $\int_{\partial(\Omega\setminus S_\alpha)} \mu\;dA=1$. We set
\begin{align*}
  \boldsymbol{\xi}:\partial(\Omega\setminus S_\alpha)\rightarrow \R^3,
  \quad \boldsymbol{\xi}(s) := 
  \begin{cases}
    \feta b(s)\quad\quad &s\in M_{-}^\alpha,\\
    \int_{M_-^\alpha} \feta b\cdot \boldsymbol{\nu}^\alpha \;dA \;
      \mu(s)\;\boldsymbol{\nu}^{\alpha,c}(s) &s\in \Gamma_s^{\alpha,c},
  \end{cases}
\end{align*}
where $\boldsymbol\nu^{\alpha}$ is the outer normal on $S_\alpha$ and 
$\boldsymbol\nu^{\alpha,c}$ the outer normal on the Lipschitz domain 
$\Omega\setminus S_\alpha$. Using the trace operator\footnote{Formally we 
extend $\feta b\in W^{1,p}(S_\alpha)$ in an $\epsilon$-neighbourhood 
``behind $\Gamma$'' by $\boldsymbol 0$, and apply the trace operator for 
the extension of $\feta b$ in $H^{1-1/p,p}$ by zero to 
$\Gamma_s^{\alpha,c}$.}, we have $\boldsymbol{\xi}\in H^{1-1/p,p}(\partial 
(\Omega\setminus S_\alpha))$ with 
\begin{align}\label{eqn:stokes_fortsetzung_xi_norm}
  \lVert\boldsymbol \xi\rVert_{H^{1-1/p,p}(\partial 
    (\Omega\setminus S_\alpha))}
  \leq c(\alpha,p,M, \mu)\,\lVert \feta b \rVert_{W^{1,p}(S_\alpha)}. 
\end{align}
Moreover, by $\boldsymbol{\nu}^{\alpha,c} = -\boldsymbol{\nu}^\alpha$ on
$M_-^\alpha$ the identity
\begin{align*}
  \int_{\partial(\Omega\setminus S_\alpha)} \boldsymbol{\xi}\cdot 
      \boldsymbol{\nu}^{\alpha,c}
   &= \int_{M_-^\alpha}\feta b\cdot \boldsymbol{\nu}^\alpha \;dA 
        \int_{\Gamma_s^{\alpha,c}} \mu\;
        \boldsymbol{\nu}^{\alpha,c}\cdot \boldsymbol{\nu}^{\alpha,c}\;ds 
        + \int_{M^\alpha_-} \feta b\cdot \boldsymbol{\nu}^{\alpha,c}\;dA
   = 0
\end{align*}
follows. Hence, taking the unique\footnote{Unique in the set of solutions
satisfying $(\mathbf{u})^*\in L^2(\partial(\Omega\setminus S_\alpha))$, 
where $(\mathbf{u})^*$ is the non-tangential maximal function of 
$\mathbf{u}$, see \cite{MR975121}.} weak solution of the Stokes 
equation for a vanishing forcing term and the boundary data 
$\boldsymbol{\xi}$ (see \cite[Lemma 2.4]{MR1386766} and 
\cite[Theorem 14]{MR2665048}), we can extend $\feta b\in W^{1,p}(S_\alpha)$
to $\feta b \in W^{1,p}(B_\alpha)$. By the linearity of the Stokes 
equation, the continuity of the solution operator and 
\eqref{eqn:stokes_fortsetzung_xi_norm}, the continuity constant is 
again bounded as claimed. It should be noted that the restriction on 
$p$ in the theorem comes from the existence result of Stokes equation 
in Lipschitz domains.
\end{proof}
Using the same construction as in the preceding Lemma but with the 
unique very weak solution from \cite[Theorem 0.3]{MR1223521}), we can 
also extend a given $L^p(M)$ function. Since this very weak solution 
coincides with the unique weak solution of the Stokes system if regular 
data is given, a simple approximation argument and Remark  
\ref{bem:kompatibilitaet_spuroperatoren} show that the extended 
function is divergence-free and fulfils an appropriate identity for the 
normal trace.
\begin{proposition}\label{lem:fortsetzung_stationaer_schwach}
  Let $\eta\in H^2(M)$ with $\lVert \eta \rVert_{L^\infty(M)} < \alpha 
  <\kappa$ and $2\leq p < \infty$. Then there exists an linear, 
  continuous extension operator
  \begin{align*}
    \feta: L^p(M) \rightarrow L^p_\sigma(B_\alpha)
  \end{align*}
  with $\trneta\lvert_M (\feta b)\lvert_{\Omega_\eta} = b\,
  \gamma(\eta)$. If $b\in W^{1,p}_0(M)$ for some $p\in [2,3]$, 
  $\feta b$ is equivalent to the operator constructed in
  Lemma \ref{lem:fortsetzung_stationaer}. For a fixed $N\in \N$, the 
  continuity constant can be chosen uniformly with respect to $\eta\in H^2(M)$ 
  satisfying $\lVert \eta \rVert_{H^2(M)}\leq N$ and 
  $\lVert\eta\rVert_{L^\infty(M)}<\alpha$.
\end{proposition}
By the construction of our extension operators, convergence of the 
displacements $\eta_n$ imply also the ``uniform'' convergence of the 
extended functions.
\begin{lemma}\label{lem:fortsetzung_konvergenzen_stationaer}
  Let $2\leq p < \infty$, $\eta_n,\,\eta\in H^2(M)$ with
  $\lVert \eta_n \rVert_{L^\infty(M)}, \lVert \eta \rVert_{L^\infty(M)} 
  <\kappa$ and suppose that $\eta_n$ converges uniformly towards $\eta$ on $M$. Then 
  $\fetan b \rightarrow \feta b$ in $L^p(B_\alpha)$ uniformly with 
  respect to $b\in L^p(M)$ and $b\in H_0^2(M)$ satisfying $\lVert b \rVert_{L^p(M)}\leq 1$ 
  and $\lVert b\rVert_{H_0^2(M)}\leq 1$, respectively.
\end{lemma}
\begin{proof}
  By the uniform convergence of $\eta_n$, 
  \begin{align*}
    \int_{\eta_n\circ \mathbf{q}}^{s}\big(\operatorname{div} (\boldsymbol\nu\circ 
      \mathbf{q})\big)(\mathbf{q}+\tau\, \boldsymbol\nu \circ 
        \mathbf{q})\;d\tau
      \rightarrow
        \int_{\eta \circ \mathbf{q}}^{s}\big(\operatorname{div} (\boldsymbol\nu\circ 
          \mathbf{q})\big)(\mathbf{q}+\tau\, \boldsymbol\nu \circ 
            \mathbf{q})\;d\tau
  \end{align*}
  uniformly in $S_\alpha$. Therefore, $\fetan b$ converges uniformly 
  to $\feta b$ in $S_\alpha$, where the convergence is also uniform with 
  respect to $b\in L^p(M)$ satisfying $\lVert b \rVert_{L^p(M)}\leq 1$ (or 
  $b\in H_0^2(M)$ satisfying $\lVert b\rVert_{H_0^2(M)}\leq 1$). Moreover,
  by the definition of the boundary data $\boldsymbol{\xi}$ and the 
  linearity and continuity of the solution operator to the Stokes system,
  this convergence carries over to the whole extension operator.
\end{proof}
Next, we look at the induced time-dependent extension operator.
\begin{lemma}\label{lem:fortsetzung_instationaer}
  Let $\eta\in \widetilde{Y}^I$ with $\lVert \eta \rVert_{L^\infty(I\times M)} 
  <\alpha <\kappa$ and $2\leq p \leq 3$. The application of the  
  operator from Lemma \ref{lem:fortsetzung_stationaer} at almost all 
  times defines a linear, continuous operator
  \begin{align*}
    \feta: H^1(I;L^2(M))\cap &L^p(I;H_0^2(M))\\
      &\rightarrow H^1(I;L^2(B_\alpha))\cap 
          C(\overline{I};H^1(B_\alpha))\cap 
              L^p(I;W^{1,p}(B_\alpha)),
  \end{align*}
  which satisfies $\operatorname{div} \feta b = 0$ in $I\times B_\alpha$ and 
  $\treta \feta b\lvert_{\Omega_\eta^I}  = b\,\boldsymbol\nu$ on 
  $I\times M$. Fixing $N\in\N$, the continuity constant can be 
  chosen uniformly with respect to $\eta\in \widetilde{Y}^I$ satisfying 
  $\lVert \eta  \rVert_{L^\infty(I\times M)} <\alpha$ and 
  $\lVert \eta \rVert_{\widetilde{Y}^I} <N$.
\end{lemma}
\begin{proof}
  By $p\geq 2$ and the parabolic embedding (see 
  \cite[Theorem II.5.14.]{MR2986590}) as well as the identification of 
  the resulting interpolation space by \cite[Proposition 3.1]{MR2744150}, 
  the embedding  
  \begin{align*}
    H^1(I;L^2(M))\cap L^p(I;H_0^2(M)) 
      \hookrightarrow C(\overline{I};H^1_0(M))    
  \end{align*}
  is continuous. Furthermore, using Lemma \ref{lem:hoelder-einbettung} 
  and Sobolev's embedding
  \begin{align*}
    \widetilde{Y}^I &\hookrightarrow C(\overline{I};W^{1,4}(M)),\\
    H^1(I;L^2(M))\cap L^p(I;H_0^2(M)) 
      &\hookrightarrow L^p(I;W^{1,p}_0(M))
  \end{align*}
  are continuous. By definition of $\feta b$, the characterisation
  of the spatial derivatives \eqref{eqn:characterisation_spatial_derivative} 
  and the extension to $B_\alpha$ through the solution of 
  the Stokes system, we have $\feta b \in C(\overline{I};H^1(B_\alpha))\cap 
  L^p(I;W^{1,p}(B_\alpha))$ with an appropriately bounded continuity 
  constant. Again by the definition of $\feta b$ in $I\times S_\alpha$, 
  we have for the time derivative
  \begin{align}\label{eqn:characterisation_time_derivative}\begin{aligned}
    \partial_t \feta b &= \exp \Big( -\int_{\eta\circ \mathbf{q}}^{s} 
        \operatorname{div} (\boldsymbol\nu\circ \mathbf{q})\circ\big(\mathbf{q} 
          +\tau (\boldsymbol\nu \circ \mathbf{q})\big)\;d\tau \Big) 
        \Big[ \big((\partial_t b)\,\boldsymbol\nu\big)\circ \mathbf{q}\\
      &\quad\quad + (b\,\boldsymbol\nu)\circ \mathbf{q}\, 
        \Big(\operatorname{div}(\boldsymbol\nu \circ \mathbf{q})\circ 
        \big(\mathbf{q}+ (\eta\,\boldsymbol\nu )\circ \mathbf{q}\big)
        (\partial_t\eta)\circ \mathbf{q} \Big)\Big].
  \end{aligned}\end{align}
  Taking into account $L^p(I;H^2(M)) \hookrightarrow 
  L^2(I;L^\infty(M))$, we deduce that both $\partial_t \feta b\lvert_{I\times S_\alpha}
  \in L^2(I;L^2(S_\alpha))$ and the trace $\partial_t 
  \feta b\lvert_{I\times M_-^\alpha} \in L^2(I;L^2(M_-^\alpha))$ 
  are appropriately bounded. Using the  characterisation of the generalised 
  time derivative by the difference quotient (see \cite[Proposition A.6]{MR0348562} 
  and \cite[Proposition A.7]{MR0348562}) and the properties of the 
  solution operator of the Stokes system, $\partial_t \feta b \in 
  L^2(I,L^2(B_\alpha))$ follows with a continuity constant bounded as 
  claimed.
\end{proof}
Analogously we can argue for the induced operator from 
Proposition \ref{lem:fortsetzung_stationaer_schwach}.
\begin{corollary}\label{kor:fortsetzung_instationaer_weak}
  Let $0<\alpha <\kappa$, $\eta\in \widetilde{Y}^I$ with 
  $\lVert \eta \rVert_{L^\infty(I\times M)} <\alpha$. Moreover, let 
  $1\leq p \leq \infty$ and $2\leq q <\infty$. The application of the 
  extension operator from Proposition 
  \ref{lem:fortsetzung_stationaer_schwach} at almost all 
  times defines an linear, continuous extension operator
  \begin{align*}
    \feta: L^p(I;L^q(M)) 
        &\rightarrow L^p(I;L^q_\sigma(B_\alpha)),
  \end{align*}
  which satisfies $\trneta\lvert_M (\feta b)\lvert_{\Omega_\eta^I} = b\,
  \gamma(\eta)$ on $I\times M$. For higher spatial regularity, 
  the operator coincides with the one from Lemma 
  \ref{lem:fortsetzung_instationaer}. Moreover,  
  $b\in C(\overline{I};L^q(M))$ implies $\feta b \in C(\overline{I}; 
  L^q_\sigma(B_\alpha))$. For a fixed $N\in\N$, the continuity constant 
  can be chosen uniformly with respect to $\eta\in \widetilde{Y}^I$ 
  satisfying $\lVert \eta  \rVert_{L^\infty(I\times M)} <\alpha$ and 
  $\lVert \eta \rVert_{\widetilde{Y}^I} <N$.
\end{corollary}
The next Lemma states some convergence properties of the time-dependent 
extension operators.
\begin{lemma}\label{lem:fortsetzung_konvergenzen}
  Let $0<\alpha <\kappa$, $N\in\N$ and $\eta,\, \eta_n \in \widetilde{Y}^I$ 
  with $\lVert \eta \rVert_{\widetilde{Y}^I},\, \lVert \eta_n 
  \rVert_{\widetilde{Y}^I} <N$ and \linebreak $\lVert \eta 
  \rVert_{L^\infty(I\times M)},\, \lVert \eta_n \rVert_{L^\infty(I\times M)} 
  <\alpha$ for all $n\in\N$.
  \begin{enumerate}[label=\alph*)]
    \item Let $1\leq p \leq \infty$, $2\leq q < \infty$ and $b\in 
      L^p(I;L^q(M))$. If $\eta_n$ converges uniformly towards $\eta$ in $I\times M$, 
      then $\fetan b \rightarrow \feta b$ in $L^p(I;L^q(B_\alpha))$. 
      Moreover, for $b\in C(\overline{I};L^q(M))$, one has $\fetan b \rightarrow 
      \feta b$ in $C(\overline{I};L^q(B_\alpha))$.
    \item If $b_n$ converges weakly towards $b$ in $L^2(I\times M)$,
      $\fetan b_n$ converges weakly towards $\feta b$ in 
      $L^2(I\times B_\alpha)$.
    \item Let $2\leq p \leq 3$ and $b\in H^1(I;L^2(M))\cap 
      L^p(I;H_0^2(M))$. If $\eta_n\rightarrow \eta$ uniformly 
      on $I\times M$ and $\partial_t\eta_n\rightarrow \partial_t \eta$ in 
      $L^2(I,L^2(M))$, then $\fetan b \rightarrow \feta b$ in 
      $H^1(I;L^2(B_\alpha))\cap L^p(I;W^{1,p}(B_\alpha))\cap 
      C(\overline{I};L^4(B_\alpha))$.
  \end{enumerate}
\end{lemma}
\begin{proof}
  \hspace*{1em}\textit{a)} Let $b\in C(\overline{I};L^q(M))$. By the 
    definition of $\feta b$ on $I\times S_\alpha$ and the extension to 
    $I\times B_\alpha$ through the solution of the Stokes system, 
    $\fetan b$ converges towards $\feta b$ in 
    $C(\overline{I}; L^q(B_\alpha))$. By density and the uniform bound 
    on the continuity constants from Corollary 
    \ref{kor:fortsetzung_instationaer_weak}, the claim follows.\newline
  \hspace*{1em}\textit{b)} By the definition of $\feta b$ on   
    $I\times S_\alpha$, $\fetan b_n$ converges weakly towards $\feta b$ 
    in $L^2(I\times S_\alpha)$ (note that the integral term converges 
    uniformly). Also the boundary values for the Stokes  
    equation $\boldsymbol\xi_n$ converge weakly towards $\boldsymbol\xi$
    in $L^2(I\times M_-^\alpha)$. Since the solution operator of the 
    Stokes equation is linear and continuous, the weak convergence 
    carries over to the whole extension operator.\newline
  \hspace*{1em}\textit{c)} Using the continuous embedding
    \begin{align*}
      H^1(I;L^2(M))\cap L^p(I;H_0^2(M)) 
        \hookrightarrow C(\overline{I};H_0^1(M))
        \hookrightarrow C(\overline{I};L^4(B_\alpha))
    \end{align*}
    (see the proof of Lemma \ref{lem:fortsetzung_instationaer} and 
    Sobolev's embedding), \textit{a)} implies $\fetan b \rightarrow \feta b$ in 
    $C(\overline{I};L^4(B_\alpha))$. Moreover, by the parabolic embedding
    (see \cite[Chapter I, Proposition 3.1]{MR1230384}) it follows 
    $b\in L^{2p}(I;L^{2p}(M))$. Since $H^2(M) \hookrightarrow\hookrightarrow 
    W^{1,2p}(M) \hookrightarrow\hookrightarrow L^2(M)$ are continuous and 
    compact, Aubin-Lions-Simons Lemma (see \cite[Theorem II.5.16]{MR2986590}) 
    implies that the embedding
    \begin{align*}
      \left\{v\in L^{2p}(I;H^2(M))\,\left|\, 
          \partial_t v \in L^{2}(I; L^2(M))\right.\right\}
        \hookrightarrow\hookrightarrow
      L^{2p}(I;W^{1,2p}(M))
    \end{align*}
    is also compact, i.\,e.\ $\eta_n$ converges strongly towards $\eta$ 
    in $L^{2p}(I;W^{1,2p}(M))$. Hence, by the characterisation of the 
    spatial derivative \eqref{eqn:characterisation_spatial_derivative}, 
    $\partial_i \fetan b$ converges towards $\partial_i \feta b$ in 
    $L^p(I;L^p(S_\alpha))$. To treat the time derivative, we use the 
    interpolation inequality 
    \begin{align*}
      \lVert \partial_t \eta_n - \partial_t \eta\rVert_{L^6(I;L^2(M))}
      \leq \lVert \partial_t \eta_n - \partial_t \eta
          \rVert^\frac{2}{3}_{L^\infty(I;L^2(M))}
        \;\lVert \partial_t \eta_n - \partial_t \eta
          \rVert^\frac{1}{3}_{L^2(I;L^2(M))}
    \end{align*}
    to show the strong convergence of $\partial_t \eta_n$ towards 
    $\partial_t \eta$ in $L^6(I;L^2(M))$. Arguing again by interpolation 
    and Sobolev's embedding, the embeddings $$H^1(I;L^2(M))\cap L^p(I;H_0^2(M)) 
    \hookrightarrow L^3(I;H^{4/3}(M)) \hookrightarrow L^3(I;L^\infty(M))$$ 
    are continuous, i.\,e.\ $b\in L^3(I;L^\infty(M))$. By the characterisation 
    of the time derivative \eqref{eqn:characterisation_time_derivative}, 
    $\partial_t \fetan b$ converges towards $\partial_t \feta b$ in 
    $L^2(I;L^2(S_\alpha))$. Using the properties of the Stokes operator, 
    these convergences carry over to the extension operator defined on 
    $I\times B_\alpha$, which finishes the proof.
\end{proof}
\section{Main Result}
\label{section:main_result}
Since our regularity is not sufficient to treat the force-coupling 
boundary term $\mathbf{F}$, we add the weak formulations of the 
shell and the fluid and couple the test functions in a way that the 
force-coupling term vanishes. This implies that the test functions 
depend on the solution of the shell equation. For $\eta\in \widetilde{Y}^I$ 
with $\lVert \eta \rVert_{L^\infty(I\times M)} <\kappa$, we therefore define 
the canonically normed space $T^I_{\eta}$ as the set of all couples 
$(b,\boldsymbol\varphi)$, where 
\begin{align*}
  b&\in H^1(I;L^2(M))\cap L^3(I;H_0^2(M)),\\
  \boldsymbol\varphi&\in W_\eta:= H^1(I;L^2(\Omega_{\eta(t)}))\cap 
    L^{3}(I;W^{1,3}(\Omega_{\eta(t)}))
      \cap L^\infty(I;L^4(\Omega_{\eta(t)})),
\end{align*}
with $b(T,\cdot)=0$, $\boldsymbol\varphi(T,\cdot)=\boldsymbol 0$, 
$\operatorname{div} \boldsymbol\varphi = 0$ (see Remark 
\ref{grundlagen:Spuroperator_Raum_Testfkt}) and for which 
$\boldsymbol\varphi -\feta b$ can be approximated by functions 
$\boldsymbol\varphi_n\in W_\eta$, $\boldsymbol\varphi_n(T,\cdot)= 
\boldsymbol 0$, $\operatorname{div} \boldsymbol\varphi_n = 0$ vanishing 
in a neighbourhood of the moving boundary. This implies $\treta 
\boldsymbol\varphi = b\,\boldsymbol\nu$ on $I\times M$. We call 
$(\mathbf{f},\, g,\, \mathbf{u}_0,\, \eta_0,\, \eta_1)$ 
\emph{admissible data} if $\mathbf{f}\in L^2((0,\infty)\times B_\kappa)$, 
$g\in L^2((0,\infty)\times M)$, $\eta_0\in H_0^2(M)$ with 
$\lVert \eta_0 \rVert_{L^\infty(M)}<\kappa$, $\eta_1 \in L^2(M)$ and 
$\mathbf{u}_0\in L^2_{\sigma}(\Omega_{\eta_0})$ with
\begin{align*}
  \langle \trnetanull\lvert_{M}\mathbf{u}_0,b\rangle 
    = \int_M \eta_1\,\gamma(\eta_0)\,b\;dA.
\end{align*}
Hence, our notion of a weak solution to 
\eqref{eqn:stokes}--\eqref{eqn:gln:anfangswert} is the following.
\begin{definition}[weak solution]\label{weak_solution}
  Let $(\mathbf{f},\, g,\, \mathbf{u}_0,\, \eta_0,\, \eta_1)$ be 
  admissible data. We call the couple $(\eta,\mathbf{u})$ a \emph{weak 
  solution} to \eqref{eqn:stokes}--\eqref{eqn:gln:anfangswert} 
  on the interval $I:=(0,T)$ if $\eta \in Y^I$ with 
  $\lVert \eta \rVert_{L^\infty(I\times M)} <\kappa$, $\eta(0,\cdot) = 
  \eta_0$ and $\mathbf{u}\in X^I_\eta$ with $\treta \mathbf{u} 
  = \partial_t\eta\,\boldsymbol\nu$ on $I\times M$ satisfies
  \begin{align}\label{eqn:schwache_formulierung_def}\begin{aligned}
    - &\int_I\intetat\mathbf{u}\cdot\partial_t\boldsymbol\varphi\;dx\,dt    
      - \frac{1}{2}\int_I\int_M (\partial_t\eta)^2\,b\,\gamma(\eta)\;dA\,dt
    - 2\int_I\int_M \partial_t\eta\,\partial_t b\;dA\,dt\\
      &+ 2\int_I K(\eta,b)\;dt
      + 2\int_I\intetat \mathbf{D}\mathbf{u}
        :\mathbf{D}\boldsymbol\varphi\;dx\,dt\\
    &\quad + \frac{1}{2}\int_I\intetat(\mathbf{u}\cdot \nabla)
        \mathbf{u}\cdot\boldsymbol\varphi\;dx\,dt
      - \frac{1}{2}\int_I\intetat(\mathbf{u}\cdot \nabla)\boldsymbol
        \varphi\cdot\mathbf{u}\;dx\,dt\\
    &\qquad = \int_I\int_M g\,b\;dA\,dt
      + \int_I\intetat\mathbf{f}\cdot\boldsymbol\varphi\;dx\,dt
      + \int_{\Omega_{\eta_0}}\hspace{-0.2cm}\mathbf{u}_0\cdot 
        \boldsymbol\varphi(0,\cdot)\;dx 
      + 2\int_M \eta_1\,b(0,\cdot)\;dA\raisetag{2.5cm}
  \end{aligned}\end{align}
  for all $(b,\boldsymbol\varphi)\in T^I_\eta$.
\end{definition}
We remark that by our regularity the appearing integrals are 
well-defined, as we will show exemplary for the first convective term. 
By the Korn-type inequality from Lemma \ref{lem:korn_type_inequality}, 
we have $\mathbf{u}\in L^2(I;W^{1,\frac{13}{7}}(\Omega_{\eta(t)}))$. 
Moreover, Sobolev's embedding yields \linebreak 
$\mathbf{u}\in L^2(I;L^\frac{52}{11}(\Omega_{\eta(t)}))$. Since 
$\frac{7}{13} + \frac{11}{52} + \frac{1}{4}=1$, H\"older's inequality
yields
\begin{align*}
  \Big|\int_I\intetat(\mathbf{u}&\cdot \nabla)\mathbf{u}
      \cdot\boldsymbol\varphi\;dx\,dt\Big|\\
    &\leq \int_I \lVert \mathbf{u}(t,\cdot)\rVert_{L^{52/11}(\Omega_{\eta(t)})}
      \lVert \nabla \mathbf{u}(t,\cdot)\rVert_{L^{13/7}(\Omega_{\eta(t)})}
      \lVert \boldsymbol\varphi(t,\cdot)\rVert_{L^4(\Omega_{\eta(t)})}\;dt\\
    &\leq c\,\lVert \mathbf{u}\rVert^2_{X^I_\eta}
      \lVert \boldsymbol\varphi\rVert_{T^I_\eta}.
\end{align*}
%
%
Our main theorem is the following.
\begin{theorem}\label{main_theorem}
  Let $\Omega$ be an admissible in- and outflow domain and 
  $(\mathbf{f},\, g,\, \mathbf{u}_0,\, \eta_0,\, \eta_1)$ admissible 
  data. There exists a time $0<T^\ast\leq \infty$ such that for all $0< T < 
  T^\ast$ a weak solution $(\eta,\mathbf{u})$ of 
  \eqref{eqn:stokes}--\eqref{eqn:gln:anfangswert} exists on the interval 
  $I:=(0,T)$ and satisfies
  \begin{align}\label{lem:eqn_esssup}
    \operatorname*{esssup}_{t\in (0,T)}\sqrt{E(t)}
    &\leq \sqrt{E_0} 
        + \int_0^{T}\!\! \frac{1}{\sqrt{2}}\,\lVert\mathbf{f}(s,\cdot)
        \rVert_{L^2(\Omega_{\eta(s)})}
        + \frac{1}{2}
          \lVert g(s,\cdot) \rVert_{L^2(M)}\;ds.
  \end{align}
  If the admissible data is sufficiently small, we have $T^\ast = 
  \infty$. If $T^\ast< \infty$, we find a $T<T^\ast$ and a weak solution 
  $(\eta,\mathbf{u})$ such that the maximal displacement 
  $\lVert \eta \rVert_{L^\infty(I\times M)}$ is arbitrary close to 
  $\kappa$.  
\end{theorem}
\begin{remark}
  For $T^\ast<\infty$ the displacement of our solution is arbitrary
  close to $\kappa$, at which point different parts of the shell could 
  touch each other, i.\,e.\ reach a situation which is not covered by our
  mathematical model.
\end{remark}

\subsection{Compactness}
\label{subsection:compactness}
By the non-linearity of the convective term, the weak convergences
implied by the formal a priori estimate for some approximate solutions
are insufficient to pass to the limit but a compactness argument is
required. Because of our noncylindrical domain, classical arguments
like Aubin-Lions \cite{MR0259693} are not applicable.  In particular,
the representation of the dual spaces and an appropriate notion of a
generalised time derivative are not clear. The proof of Aubin-Lions
Lemma uses basically of the fundamental theorem of calculus, an
application of Ehrling's Lemma and an Arzela-Ascoli argument.
Lengeler used in \cite{phdlengeler} (see also \cite{MR3147436}) the weak 
formulation of his problem instead of the fundamental theorem of calculus 
(and some modified Ehrling Lemma) to prove compactness. A careful analysis 
of his proof shows that replacing his extension operator
${\mathcal{F}_{\eta_n}}{\mathcal{M}_{\eta_n}}$ by our extension
operator $\feta$ and using the framework developed above, especially the
density result from Lemma \ref{lem:density_for_compactness}, the
result also holds in our situation. More precisely, we have the
following generalisation of \cite[Proposition 3.8]{MR3147436}:
\begin{lemma}\label{lem:kompaktheit}
  Let $0<\alpha<\kappa$ and $(\mathbf{f},\, g,\, \mathbf{u}_0,\, 
  \eta_0,\, \eta_1)$ be admissible data, $(\delta_n)_{n\in \N}\subset 
  \widetilde{Y}^I$ a bounded sequence with $\lVert \delta_n 
  \rVert_{L^\infty(I\times M)} <\alpha$ and assume that
  \begin{align*}
    \delta_n  \rightarrow \delta \qquad 
      \text{uniformly in } I \times M
  \end{align*}
  for some $\delta \in \widetilde{Y}^I$  with $\lVert \delta 
  \rVert_{L^\infty(I\times M)}<\alpha$. Moreover, let 
  $\big((\eta_n,\mathbf{u}_n,\mathbf{v}_n,\mathbf{u}_0^n, 
  \eta_1^n)\big)_{n\in\N}$ be a bounded sequence in $Y^I\times 
  X^I_{\delta_n}\times \widetilde{X}^I_{\delta_n}
  \times L^2(\Omega_{\delta_n(0)})\times L^2(M)$ satisfying 
  \begin{align}\label{lem:komp_equation}\begin{aligned}
    - &\int_I\!\intdeltan\mathbf{u}_n\cdot\partial_t\boldsymbol\varphi\;dx\,dt 
      - \frac{1}{2}\int_I\!\int_M (\partial_t\eta_n)\,(\partial_t\delta_n)\,b\,
        \gamma(\delta_n)\;dA\,dt\\
    &- 2 \int_I\!\int_M\! \partial_t\eta_n\,
          \partial_t b\;dA\,dt 
      + 2\int_I\! K(\eta_n,b)\;dt
      + 2\int_I\!\intdeltan\! \mathbf{D}\mathbf{u}_n:
        \mathbf{D}\boldsymbol\varphi\;dx\,dt\\
    &\quad + \frac{1}{2}\int_I\!\intdeltan(\mathbf{v}_n\cdot \nabla)
        \mathbf{u}_n\cdot\boldsymbol\varphi\;dx\,dt 
      - \frac{1}{2}\int_I\!\intdeltan(\mathbf{v}_n\cdot \nabla)
        \boldsymbol\varphi\cdot\mathbf{u}_n\;dx\,dt\\
    &\qquad= \int_I\!\int_M g\,b\;dA
      + \int_I\!\intdeltan\mathbf{f}\cdot\boldsymbol\varphi\;dx\,dt
      + \int_{\Omega_{\delta_n(0)}}\hspace{-0.5cm}\mathbf{u}_0^n
        \cdot \boldsymbol\varphi(0,\cdot)\;dx 
      +2 \int_M \eta_1^n\,b(0,\cdot)\;dA\raisetag{2.7cm}
  \end{aligned}\end{align}
  for all $(b,\boldsymbol\varphi)\in T^I_{\delta_n}$. Furthermore, let 
  $\trdeltan\mathbf{u}_n= \partial_t \eta_n\,\boldsymbol\nu$ on 
  $I\times M$ and 
  \begin{align}\label{lem:L^2_komp_vor_konv}\begin{alignedat}{2}
    \partial_t\eta_n &\rightharpoonup \partial_t\eta 
      &\qquad&\text{weakly in } L^2(I,L^2(M)),\\
    \mathbf{u}_n &\rightharpoonup \mathbf{u} 
      &&\text{weakly in } L^2(I,L^2(\R^3))
  \end{alignedat}\end{align}
  for some $\eta\in Y^I$, $\mathbf{u}\in X^I_\delta$, where 
  $\mathbf{u}_n$, $\mathbf{u}$ are spatially extended by zero. Then 
  $(\partial_t\eta_n,\mathbf{u}_n)$ converges strongly
  towards $(\partial_t\eta,\mathbf{u})$ in $L^2(I\times M)\times 
  L^2(I\times \R^3)$.
\end{lemma}
Note that the structure of the proof of Lemma
\ref{lem:kompaktheit} and of the compactness result in
\cite {phdlengeler}, \cite{MR3147436} is the same, but lengthy and technically
demanding. The proof of Lemma \ref{lem:kompaktheit} can be found in \cite[Lemma
4.4]{phdeberlein}. Since already the proof in \cite {MR3147436} is
densely written we give, for the convenience of the reader, full details in the Appendix.
%
Note that by the assumptions of the preceding Lemma, $\trdelta\mathbf{u} = 
\partial_t \eta\,\boldsymbol\nu$ is implied by Lemma 
\ref{moving_boundary_data}. This fact is missing in
\cite{phdlengeler}, \cite{MR3147436}.  
%
\subsection{Construction of basis functions}\label{kap:ansatz}
Since we do not transform our system to the reference domain, we have to 
construct an appropriate set of basis functions, at least in case of a 
given ``smooth'' deformation $\delta \in C^4(I\times M)$ with 
$\lVert \delta \rVert_{L^\infty(I\times M)} <\kappa$ . Therefore we chose a 
basis $(\widehat{Y}_k)_{k\in\N}$ of $H^2_0(M)$ and a basis 
$(\widehat{\mathbf{X}}_k)_{k\in \N}$ of the canonically normed space
\begin{align*}
  X(\Omega):= \{ \boldsymbol\varphi\in W^{1,3}(\Omega) \,|\, 
      \operatorname{div} \boldsymbol\varphi = 0,\; \boldsymbol\varphi\lvert_M = 0\}.
\end{align*}
We extend $\widehat{Y}_k$ to ${\mathcal{F}_{0}}\widehat{Y}_k \in X(\Omega)$ 
by Lemma \ref{lem:fortsetzung_stationaer} and set $(\widehat{W}_{2k}, 
\widehat{\mathbf{W}}_{2k}) := (0,\widehat{\mathbf{X}}_k)$, \linebreak 
$(\widehat{W}_{2k-1},\widehat{\mathbf{W}}_{2k-1}):=(\widehat{Y}_k,
{\mathcal{F}_{0}}\widehat{Y}_k)$. Hence $\widehat{\mathbf{W}}_k\lvert_M 
= \widehat{W}_k\,\boldsymbol\nu$ on $M$. Further we define the space 
$\widehat{T}(I,\Omega)$ as the canonically normed set of all pairs
\begin{align*}
  (b,\boldsymbol{\varphi})\in \Big[H^1(I;L^2(M))
      \cap L^3(I;H_0^2(M))\Big]\times \Big[H^1(I;L^2(\Omega))
        \cap L^{3}(I;W^{1,3}(\Omega))\Big]
\end{align*}
satisfying $b(T,\cdot) =0$ in $M$, $\boldsymbol{\varphi}(T,\cdot)= 
\boldsymbol{0}$ in $\Omega$, $\operatorname{div}\boldsymbol{\varphi}=0$ 
in $I\times\Omega$ and $\left.\boldsymbol{\varphi}\right\lvert_{I\times M} 
= b\,\boldsymbol{\nu}$. 
\begin{lemma}\label{lem:approx_dicht_ansatzf}
    The set
    \begin{align*}
      span\Big\{ (\varphi \widehat{W}_k,\varphi\widehat{\mathbf{W}}_k)
        \bigm| \varphi\in C_0^1([0,T)),\,k\in \N\Big\}
    \end{align*}
    is dense in the space $\widehat{T}(I,\Omega)$.
  \end{lemma}
  \begin{proof}
  Obviously, the set is contained in $\widehat{T}(I,\Omega)$. Let 
  $(b,\boldsymbol{\varphi})\in \widehat{T}(I,\Omega)$. We 
  approximate $b$ by functions $\widetilde{b}_\epsilon\in 
  C_0^\infty([0,\infty);H^2_0(M))$ such that $\widetilde{b}_{\epsilon}
  (T,\cdot)=0$ and $\partial_t\,\widetilde{b}_{\epsilon}(T,\cdot)=0$, 
  using classical mollifications. Since $(\widehat{Y}_k)_{k\in \N}$ is 
  a basis of $H_0^2(\Omega)$ and $\partial_t\widetilde{b}_{\epsilon} \in 
  L^2(I;H^2_0(M))$, for $\ell\in \N$ we get functions 
  $\alpha^{\epsilon,\ell}_k\in C_0^1([0,T))$, $1\leq k\leq \ell$ 
  with
  \begin{align*}
    \sum_{k=1}^\ell \alpha^{\epsilon,\ell}_k\, \widehat{Y}_k 
      \rightarrow \partial_t\widetilde{b}_{\epsilon} 
        \quad\quad \text{in } L^3(I;H^2_0(M)).
  \end{align*}
  By the inequality
  \begin{align*}
    \Big\lVert \widetilde{b}_{\epsilon}(t,\cdot) 
        + \sum_{k=1}^\ell\int_t^T\!\!\!\! \alpha^{\epsilon,\ell}_k(s)\;ds\; 
          \widehat{Y}_k(\cdot)\Big\rVert_{H_0^2(M)}
      \leq \int_0^T\!\Big\lVert \partial_t\widetilde{b}_{\epsilon}(s,\cdot) 
        - \sum_{k=1}^\ell\! \alpha^{\epsilon,\ell}_k(s)\,
          \widehat{Y}_k(\cdot)\Big\rVert_{H_0^2(M)}\;ds
  \end{align*}
  and an appropriate coupling of $\epsilon$ and $\ell$ we have 
  $\sum_{k=1}^{\ell(\epsilon)}\int_{t}^T \alpha^{\epsilon,
  \ell(\epsilon)}_k(s)\;ds\; \widehat{Y}_k \rightarrow b$ \linebreak
  in $H^1(I;L^2(M))\cap L^3(I;H^2_0(M))$. By the linearity and continuity 
  of the operator ${\mathcal{F}_{0}}$, \linebreak $\big(\sum_{k=1}^{\ell(\epsilon)}
  \int_{t}^T \alpha^{\epsilon,\ell(\epsilon)}_k(s)\;ds\; 
  {\mathcal{F}_{0}}\widehat{Y}_k(x)\big)$ converges to 
  ${\mathcal{F}_{0}}\,b$ in $H^1(I;L^2(\Omega))\cap L^{3}(I;W^{1,3}(\Omega))$.  \linebreak 
  Hence, it now suffices to approximate $(0,\boldsymbol\varphi - 
  {\mathcal{F}_{0}}\,b)$ appropriately. This can be done analogously 
  to the approximation of $b$. The missing details can be found in
  \cite[Lemma 5.3]{phdeberlein}.

\end{proof}
Using the Piola transform, we map the space $\widehat{T}(I,\Omega)$ to 
the moving domain. Since we have to preserve the compatibility constraint, 
we have to construct an compatible diffeomorphism for the structure part.
By the definition of our trace operator and the Piola transform, we have 
\begin{align}\label{eqn:piola_rand_ansatzfunktionen}\begin{aligned}
  \trdelta\Tdelta\boldsymbol{\varphi} 
    = \left.(\Tdelta\boldsymbol{\varphi} \circ \Psi_{\delta})
      \right\lvert_{I\times M}
    = \left.(d\Psi_{\delta}\, (\operatorname{det} d\Psi_{\delta})^{-1} \boldsymbol{\varphi})
      \right\lvert_{I\times M}.
\end{aligned}\end{align}
By the definition of the Hanzawa transform, $\Psi_\delta(t,x)
= x + \delta(t,\mathbf{q}(x))\,\boldsymbol\nu(\mathbf{q}(x))$ in a 
neighbourhood of $\overline{I}\times M$. Hence, the differential 
$d\Psi_{\delta}$ only scales the outer normal $\boldsymbol\nu$ on the 
boundary, i.\,e\ there exists $g:\overline{I}\times M \rightarrow \R$ with
\begin{align*}
 d\Psi_{\delta}(t,x)\,\boldsymbol{\nu}(x) 
  = g(t,x)\,\boldsymbol{\nu}(x),\quad\quad (t,x)\in \overline{I}\times M.
\end{align*}
Since $g(t,x) = d\Psi_{\delta}(t,x)\,\boldsymbol{\nu}(x) \cdot \boldsymbol
\nu(x)$, we have $g\in C^2(\overline{I}\times M)$ and $g\neq 0$. Therefore, 
the map $\mathbf{T}_\delta(b,\boldsymbol{\varphi}) := 
(g\,(\operatorname{det} d\Psi_{\delta}\lvert_{\overline{I}\times M})^{-1} b,
\Tdelta\boldsymbol{\varphi})$ is an isomorphism from $\widehat{T}(I,\Omega)$ 
into the canonically normed space $T(I,\Omega_\delta)$ of the couples 
\begin{align*}
(b,\boldsymbol{\varphi})\in \big(H^1(I;L^2(M))\cap L^3(I;H_0^2(M))\big)
\times \big(H^1(I;L^2(\Omega_{\delta(t)}))\cap L^{3}(I;W^{1,3}
(\Omega_{\delta(t)}))\big)
\end{align*}
satisfying $b(T,\cdot) =0$, $\boldsymbol{\varphi}(T,\cdot)=\boldsymbol{0}$, 
$\operatorname{div}\boldsymbol{\varphi}=0$ and 
$\trdelta\boldsymbol{\varphi} = b\,\boldsymbol{\nu}$ on $I\times M$.
%
By our construction, the basis functions
$(W_k, \mathbf{W}_k):= \mathbf{T}_\delta\big((\widehat{W}_{k},
\widehat{\mathbf{W}}_{k})\big)$,
$k \in \N$, have the following properties:
\begin{proposition}\label{prop:approx_eig_ansatz}
The following assertions hold.
\begin{enumerate}[label=\alph*)]
  \item $W_k\in C(\overline{I};H^2_0(M)) \cap C^2(\overline{I};L^2(M))$ 
    and $\mathbf{W}_k\in L^\infty(I;W^{1,3}(\Omega_{\delta(t)}))\cap 
    H^1(I;L^2(\Omega_{\delta(t)}))$. Moreover, we have 
    $\mathbf{W}_k(t,\cdot)\in W^{1,3}(\Omega_{\delta(t)})$ for all $t\in 
    \overline{I}$.
  \item $\operatorname{div} \mathbf{W}_k = 0$ and $\trdelta \mathbf{W}_k = 
    W_k\,\boldsymbol{\nu}$ on $I\times M$. In particular, 
    $\trdelta\mathbf{W}_k = \boldsymbol{0}$ on $I\times M$ for 
    $k\in \N$, $k$ even.
  \item The set $span\{ (\varphi\, W_k,\varphi\,\mathbf{W}_k) \,|\, 
    \varphi\in C_0^1([0,T)),\,k\in \N \}$ is dense in $T(I,\Omega_\delta)$.
  \item For $t\in \overline{I}$ the functions $(W_k(t,\cdot))_{k\in \N,\; k \; 
    \text{odd}}$ form a basis of $H_0^2(M)$.
  \item For $t\in \overline{I}$ the functions 
    $(\mathbf{W}_k(t,\cdot))_{k\in \N,\; k \; \text{even}}$ form a basis 
    of the functions from $W^{1,3}(\Omega_{\delta(t)})$ vanishing on the 
    moving boundary.
  \item If for $t\in \overline{I}$ the linear combination $\sum_{i=1,\, 
    i \text{ odd}}^\ell \alpha_i\, W_k(t,\cdot)$ converges in $H_0^2(M)$ 
    (or $L^2(M)$), then $\sum_{i=1,\, i \text{ odd}}^\ell \alpha_i\, 
    \mathbf{W}_k(t,\cdot)$ converges in $W^{1,3}(\Omega_\delta(t))$ 
    (or $L^2(\Omega_{\delta(t)})$).
\end{enumerate}
\end{proposition}
\subsection{The decoupled, regularised and linearised problem}
\label{chapter:entreglin}
To obtain a weak solution of \eqref{eqn:stokes}--\eqref{eqn:gln:anfangswert}, 
we first decouple the dependency of the moving domain from the solution 
of the shell equation, i.\,e. we prescribe some displacement $\delta$ 
with $\delta(0,\cdot) = \eta_0$. We will later restore this coupling by 
a fixed point argument, hence we have to choose $\delta$ in a space which 
$Y^I$ embeds compactly into. Therefore, we prescribe $\delta \in 
C(\overline{I}\times M)$ with $\delta= 0$ on $I\times \partial M$ and 
$\lVert \delta \rVert_{L^\infty(I\times M)}\leq \alpha <\kappa$. We 
further regularise the displacement, therefore we also have to adapt the 
initial fluid velocity $\mathbf{u}_0 \in L^2_{\sigma}(\Omega_{\eta_0})$ 
and by the compatibility condition also the initial velocity of the shell 
equation $\eta_1\in L^2(M)$. To avoid the usual loss of regularity by 
transformation, we use the regularisation from Proposition 
\ref{prop:reg_rand:wohldef} which approximates $\delta$ at $t=0$ from 
``above''. By Proposition \ref{prop:reg_rand:wohldef} and Proposition 
\ref{prop:reg_rand:anfangswert}, there exists $0<\epsilon_0 = 
\epsilon_0(\alpha,\eta_0)$ such that
\begin{align}\label{eqn:def_epsilon_0}
  \eta_0 = \delta(0,\cdot) \leq \randreg\delta(0,\cdot),\qquad\qquad
  \lVert \randreg \delta \rVert_{L^\infty(I\times M)}
    <\frac{\alpha+\kappa}{2}<\kappa
\end{align}
holds for all $0<\epsilon<\epsilon_0$ and all $\delta \in 
C(\overline{I}\times M)$ satisfying $\delta(0,\cdot) = \eta_0$,
$\lVert \delta \rVert_{L^\infty(I\times M)}\leq \alpha$, i.\,e.\ 
$\Omega_{\eta_0}\subset \Omega_{\randreg\delta(0,\cdot)}$. Using 
$\fetanull \eta_1\in L^2_\sigma(B_\alpha)$ from Proposition 
\ref{lem:fortsetzung_stationaer_schwach}, we set\label{def_approx_anfangsdaten}
\begin{align*}
  \mathbf{u}_0^\epsilon :=
    \begin{cases}
      \mathbf{u}_0 \quad&\text{in } \Omega_{\eta_0},\\
      \fetanull \eta_1 &\text{in } \Omega_{\randreg\delta(0,\cdot)}
        \setminus\Omega_{\eta_0}
    \end{cases}
\end{align*}
and have, by the compatibility condition, $\mathbf{u}_0^\epsilon\in 
L^2_\sigma(\Omega_{\randreg\delta(0,\cdot)})$. Furthermore, defining
\begin{align*}
  \eta_1^\epsilon 
    := \exp \Big( -\int_{\eta_0}^{\randreg\delta(0,\cdot)} 
        \operatorname{div} (\boldsymbol\nu)(\cdot +\tau \boldsymbol\nu)\;d\tau \Big)
        \,\eta_1,
\end{align*}
we have $\eta_1^\epsilon\in L^2(M)$ and, by the Definition of 
$\fetanull \eta_1$ the compatibility condition,  \linebreak $\trnregdeltanull\lvert_{M}
\mathbf{u}_0^\epsilon =\eta_1\,\gamma(\randreg\delta(0,\cdot))$. Furthermore, 
using Lemma \ref{lem:reg_rand:konv}, we deduce
\begin{alignat}{2}\label{eqn:modification_start_konvergenz}\begin{aligned}
  \mathbf{u}_0^\epsilon &\rightarrow \mathbf{u}_0
      &\qquad &\text{in } L^2(\R^3),\\
  \eta_1^\epsilon &\rightarrow \eta_1
      &\qquad &\text{in } L^2(M),
\end{aligned}\end{alignat}
where we extended $\mathbf{u}_0$ and $\mathbf{u}_0^\epsilon$ by zero. 
Especially we have for all $0<\epsilon<\epsilon_0$
\begin{align}\label{eqn:modification_epsilon}
  \lVert \mathbf{u}_0^\epsilon\rVert_{L^2(\R^3)}
    \leq 2\,\lVert \mathbf{u}_0\rVert_{L^2(\R^3)}
  \quad \text{ and } \quad 
  \lVert \eta_1^\epsilon\rVert_{L^2(M)} 
    \leq 2\,\lVert \eta_1\rVert_{L^2(M)}.
\end{align}
Since we want to use the Galerkin method, we further linearise 
the convective terms by introducing a prescribed $\mathbf{v}\in 
L^2(I\times \R^3)$, where this regularity is motivated by our 
compactness result. Using the classical regularisation, we set 
$\randreg \mathbf{v} := w_\epsilon \ast \mathbf{v}$, where 
$w\in C_0^\infty(\R^4)$ is a kernel with $w\geq 0$, $supp\,w 
\subset B_1(0)$ and $\int_{\R^4} w \;dx= 1$, $w_\epsilon(x):= 
\epsilon^{-4}\,w(x/\epsilon)$ and $\mathbf{v}$ is extended by zero to 
$\R^4$. By the properties of the smoothing operator, we have \label{def_randreg_v}
$\randreg \mathbf{v}\in C^\infty(\overline{I}\times \R^3)$ (see 
\cite[Proposition II.2.25]{MR2986590}). 
\begin{definition}\label{def:schwache_loesung_entreglin}
  Let $0<\alpha<\kappa$, $(\mathbf{f},\, g,\, \mathbf{u}_0,\, 
  \eta_0,\, \eta_1)$ be admissible data, $\delta \in 
  C^0(\overline{I}\times M)$ with $\delta = 0$ 
  on $I\times \partial M$, $\lVert \delta \rVert_{L^\infty(I\times M)} 
  \leq \alpha$, $\delta(0,\cdot) = \eta_0$ on $M$, $\mathbf{v}\in 
  L^2(I\times\R^3)$ and $0<\epsilon<\epsilon_0$. We call the couple 
  $(\eta,\mathbf{u})\in Y^I\times X^I_{\randreg\delta}$ a \emph{weak 
  solution of the decoupled, regularised and linearised problem}, if 
  $\eta(0,\cdot)=\eta_0$ on $M$, $\trregdelta\mathbf{u}= 
  \partial_t \eta\,\boldsymbol\nu$ on $I\times M$ and
  \begin{align}\label{eqn:schwache_loesung_entreglin}\begin{aligned}
  - &\int_I\intreps\mathbf{u}\cdot\partial_t\boldsymbol\varphi\;dx\,dt    
    - \frac{1}{2}\int_I\int_M (\partial_t\eta)\,
      (\partial_t\randreg\delta)\,b\,\gamma(\randreg\delta)\;dA\,dt\\
  &- 2\int_I\int_M \partial_t\eta\,\partial_t b\;dA\,dt
    + 2\,\int_I K(\eta,b)\;dt
    + 2\int_I\intreps \mathbf{D}\mathbf{u}
      :\mathbf{D}\boldsymbol\varphi\;dx\,dt\\
  &\quad+ \frac{1}{2}\int_I\intreps(\randreg\mathbf{v}\cdot \nabla)
      \mathbf{u}\cdot\boldsymbol\varphi\;dx\,dt
    - \frac{1}{2}\int_I\intreps(\randreg\mathbf{v}\cdot \nabla)\boldsymbol
      \varphi\cdot\mathbf{u}\;dx\,dt\\
  &\qquad = \int_I\int_M g\,b\;dA\,dt 
    + \int_I\intreps\mathbf{f}\cdot\boldsymbol\varphi\;dx\,dt
    +\int_{\Omega_{\randreg\delta(0)}}\hspace{-0.7cm}
      \mathbf{u}_0^\epsilon\cdot \boldsymbol\varphi(0,\cdot)\;dx 
    +2\int_M \eta_1^\epsilon\,b(0,\cdot)\;dA\raisetag{2.7cm}
  \end{aligned}\end{align}
  holds for all $(b,\boldsymbol\varphi)\in T^I_{\randreg\delta}$.
\end{definition}
In analogy to the energies defined in the formal a priori estimate, we 
set\label{def_approx_energien} 
\begin{align*}
  E(\randreg\delta,\eta,\mathbf{u},t)
    &:= \frac{1}{2}\intreps |\mathbf{u}(t,\cdot)|^2\;dx 
    + 2\,\int_0^t\intrepsecht
        |\mathbf{D}\mathbf{u}(t,\cdot)|^2\;dx\,ds\\
    &\qquad + \int_M |\partial_t\eta(t,\cdot)|^2\;dA 
    + K(\eta(t,\cdot)),\\
  E_0(\randreg\delta(0,\cdot),\eta_0,\eta_1^\epsilon, 
      \mathbf{u}_0^\epsilon)
    &:= \frac{1}{2}\intrepnull |\mathbf{u}_0^\epsilon|^2\;dx 
    + \int_M |\eta_1^\epsilon|^2\;dA 
    + K(\eta_0).
\end{align*}
Our existence result for the decoupled, regularised and linearised 
problem is the following:
\begin{lemma}\label{lem:existence_approx_entreglin}
  Let $0<\alpha<\kappa$, $0<T<\infty$, $I:=(0,T)$ and 
  $(\mathbf{f},\, g,\, \mathbf{u}_0,\, \eta_0,\, \eta_1)$ be admissible 
  data. Let $\epsilon_0 = \epsilon_0(\alpha,\eta_0)$ be as given above, 
  $\delta \in C^0(\overline{I}\times M)$ with $\delta = 0$ 
  on $I\times \partial M$, $\lVert \delta \rVert_{L^\infty(I\times M)}
  \leq \alpha$, $\delta(0,\cdot) = \eta_0$ on $M$, $\mathbf{v}\in 
  L^2(I\times\R^3)$ and $0<\epsilon <\epsilon_0$. Then there exists 
  a weak solution $(\eta,\mathbf{u})\in Y^I\times X^I_{\randreg\delta}$ 
  to the the decoupled, regularised and linearised problem which, for all 
  $0<T_1\leq T$, satisfies the energy inequality
  \begin{align}\label{lem:eqn_esssup_reglin}
    \operatorname*{esssup}_{t\in (0,T_1)}\sqrt{E(\randreg\delta,\eta,\mathbf{u},t)}
      &\leq \sqrt{E_0(\randreg\delta(0,\cdot),\eta_0, 
          \eta_1^\epsilon,\mathbf{u}_0^\epsilon)}\\
      &\quad + \int_0^{T_1}\!\!\! \frac{1}{\sqrt{2}} \lVert 
        \mathbf{f}(s,\cdot) \rVert_{L^2(\Omega_{\randreg\delta(s)})} + 
          \frac{1}{2}\lVert g(s,\cdot) \rVert_{L^2(M)}\;ds.\notag
  \end{align}
  In particular, $(\eta,\mathbf{u})$ is uniformly bounded in 
  $Y^I\times X^I_{\randreg\delta}$ independently of $\delta$, 
  $\mathbf{v}$ and $\epsilon$ given the conditions $\delta(0,\cdot) = 
  \eta_0$ and $\lVert \delta \rVert_{L^\infty(I\times M)}\leq \alpha$.
\end{lemma}
\begin{proof}
  We use the Galerkin method with the constructed basis functions
  $(W_k,\mathbf W_k)$, $k\in\N$. Therefore, for a fixed $n\in \N$, we
  seek functions $\alpha_n^k:[0,T]\rightarrow \R$,
  $1\leq k \leq n$, satisfying
\begin{align}\label{eqn:reg_galerkin}\begin{aligned}
  &\intreps \partial_t \mathbf{u}_n\cdot\mathbf{W}_j\;dx 
    + \frac{1}{2}\int_M(\partial_t \eta_n)\,(\partial_t\randreg\delta)
      \,W_j\,\gamma(\randreg\delta)\;dA\\
  &\quad + 2 \int_M\partial_t^2\eta_n\,W_j\;dA 
    + 2K(\eta_n,W_j)
    + 2\intreps \mathbf{D}\mathbf{u}_n:\mathbf{D}\mathbf{W}_j\;dx\\
  &\qquad + \frac{1}{2}\intreps (\randreg\mathbf{v}\cdot \nabla)
      \mathbf{u}_n\cdot \mathbf{W}_j\;dx
    - \frac{1}{2}\intreps(\randreg\mathbf{v}\cdot \nabla)
      \mathbf{W}_j\cdot \mathbf{u}_n\;dx\\ 
  &\hspace{2.5cm} = \int_M g_n\,W_j\;dA
    + \intreps \mathbf{f}_n\cdot\mathbf{W}_j\;dx
\end{aligned}\end{align}
for all $1\leq j \leq n$ and all $t\in [0,T]$, where
\begin{align}\label{eqn:definition_galerkin_function}\begin{aligned}
 \mathbf{u}_n(t,x):=\sum_{k=1}^n\alpha^k_n(t)\mathbf{W}_k(t,x),
 \quad \eta_n(t,x) := \int_0^t\sum_{k=1}^n\alpha_n^k(s)W_k(s,x)\;ds + \eta_0(x)
\end{aligned}\end{align}
and $g_n$, $\mathbf{f}_n$ are suitable approximations of $g$ and 
$\mathbf{f}$, respectively, with
\begin{alignat}{3}\label{konv_approx_reg_data}\begin{aligned}
  g_n &\in C^0(\overline{I}\times M), \quad\quad 
    &g_n &\rightarrow g \quad 
      &&\text{in }L^2(I\times M),\\
 \mathbf{f}_n &\in C^0(\overline{I}\times \overline{B_\kappa}),\quad\quad 
    &\mathbf{f}_n &\rightarrow \mathbf{f}\quad 
        &&\text{in }L^2(I\times B_\kappa).
\end{aligned}\end{alignat}
Using the compatibility condition $\trnregdeltanull\lvert_{M}
\mathbf{u}_0^\epsilon =\eta_1\,\gamma(\randreg\delta(0,\cdot))$ and 
the properties of our basis functions as in \cite{MR3147436}, 
\cite{phdeberlein}, we can find $\alpha^n_{k,0}$ with
\begin{align}\label{eqn:konv_anfangswerte}
  \sum_{k=1}^n\! \big(\alpha^n_{k,0} W_k(0,\cdot),
    \alpha^n_{k,0} \mathbf{W}_k(0,\cdot)\big) 
    \rightarrow 
      (\eta_1^\epsilon,\mathbf{u}_0^\epsilon)
\end{align}
in $L^2(M)\times L^2(\Omega_{\randreg\delta(0)})$. An easy computation 
using the linear independence and regularity of the basis functions 
(see \cite{phdeberlein}) shows that equivalently we can search for a 
solution to the system of ordinary integro-differential equations
\begin{align}\label{approx_gal_problem}
  \dot{\boldsymbol{\alpha}}_n(t) 
    = \widetilde{\mathbf{A}}(t,\boldsymbol{\alpha}(t)) 
      + \int_0^t\widetilde{\mathbf{B}}(t,s,\boldsymbol{\alpha}(s))\;ds,
\end{align}
where $\widetilde{\mathbf{A}} \in C^0([0,\infty)\times \R^n; \R^n)$, 
$\widetilde{\mathbf{B}} \in C^0([0,\infty)\times [0,\infty)\times \R^n;\R^n)$ 
are affine linear in the last component. By \cite[Theorem 1.1.1]{MR1336142} 
and the classical extension argument, we get a solution $\boldsymbol{\alpha}_n = 
(\alpha_n^1(\cdot),\ldots, \alpha_n^n(\cdot)) \in 
C^1([0,T],\R^n)$ of \eqref{approx_gal_problem} to the initial value  
$\boldsymbol{\alpha}_n(0)=(\alpha_{1,0}^n,\ldots,\alpha_{n,0}^n)$, i.\,e.\  
$\mathbf{u}_n \in L^\infty(I,W^{1,3}(\Omega_{\randreg\delta(t)}))\cap 
H^1(I,L^2(\Omega_{\randreg\delta(t)}))$ and 
$\eta_n \in C^0(\overline{I},H^2_0(M)) \cap C^2(\overline{I},L^2(M))$ 
which satisfies (\ref{eqn:reg_galerkin}) for all $1\leq j \leq n$ and all 
$t\in [0,T]$. Furthermore, Proposition \ref{prop:approx_eig_ansatz} 
implies
\begin{align}\label{approx:gal_rand_m}
  \trregdelta (\mathbf{u}_n) 
    = \sum_{k=1}^n \alpha_n^k\,\trregdelta(\mathbf{W}_k) 
    = \sum_{k=1}^n \alpha_n^k\, W_k\,\boldsymbol{\nu}
    =\partial_t \eta_n\, \boldsymbol{\nu},
\end{align}
$\operatorname{div} \mathbf{u}_n = \sum_{k=1}^n \alpha_n^k\,\operatorname{div} \mathbf{W}_k 
=0$, and therefore $(\eta_n,\mathbf{u}_n)\in Y^I\times 
X^I_{\randreg\delta}$. Also, by \eqref{eqn:definition_galerkin_function},
we have $\eta_n(0,\cdot) = \eta_0$.

In order to derive a uniform energy estimate, we multiply 
(\ref{eqn:reg_galerkin}) with $\alpha_n^j(t)$, take the sum over 
$j=1,\ldots, n$ and get
\begin{align}\label{eqn:galerkin_energy1}\begin{aligned}
  \intreps &\partial_t \mathbf{u}_n\cdot\mathbf{u}_n\;dx 
    + \frac{1}{2}\int_M(\partial_t \eta_n)^2\,(\partial_t\randreg\delta)
      \,\gamma(\randreg\delta)\;dA + 2\int_M (\partial_t^2\eta_n)\,
      (\partial_t \eta_n)\;dA\\
  &\qquad\quad + 2K(\eta_n,\partial_t\eta_n)
    + 2\intreps \mathbf{D}\mathbf{u}_n:\mathbf{D}\mathbf{u}_n\;dx
    = \int_M g_n\,\partial_t \eta_n\;dA
    + \intreps \mathbf{f}_n\cdot\mathbf{u}_n\;dx.
\end{aligned}\end{align}
Thus, Reynolds’ transport theorem yields
\begin{align*}
  \frac{d}{dt}\intreps |\mathbf{u}_n|^2\;dx 
    &= \intreps \partial_t \,|\mathbf{u}_n|^2\;dx 
      + \int_{\partial\Omega_{\randreg\delta(t)}}\hspace{-0.7cm} 
        |\mathbf{u}_n|^2\, \mathbf{v}_{\randreg\delta(t)}\cdot
          \boldsymbol\nu_{\randreg\delta(t)}\;dA_{\randreg\delta(t)},
\end{align*}
where $\boldsymbol\nu_{\randreg\delta(t)}$ is the outer normal of 
$\Omega_{\randreg\delta(t)}$ and $\mathbf{v}_{\randreg\delta(t)}$ the 
boundary velocity. By our assumption for the moving domain, the 
deformations happen only along the outer normal on $M$. In particular, 
the outer normal on $\partial M$ is perpendicular to the normal on 
$\Gamma_{\randreg\delta(t)}$ by the orthogonality assumption for the 
reference domain. Hence, transforming the boundary integral 
to the boundary of the reference domain and taking into account the 
compatibility condition (\ref{approx:gal_rand_m}) and Remark 
\ref{grundlagen:eigenschaften_gamma}, the identity
\begin{align*}
  \int_{\partial\Omega_{\randreg\delta(t)}}\hspace{-0.7cm} 
      |\mathbf{u}_n|^2\, &\mathbf{v}_{\randreg\delta(t)}\cdot
        \boldsymbol\nu_{\randreg\delta(t)}\;dA_{\randreg\delta(t)}
    = \int_M (\partial_t\eta_n)^2\,(\partial_t\randreg\delta)\,
      \gamma(\randreg\delta)\;dA
\end{align*}
follows. Arguing as in our formal a priori estimate, we deduce 
the energy inequality
\begin{align}\label{eqn:esssup_galerkin}
  \operatorname*{esssup}_{t\in (0,T_1)}\sqrt{E(\randreg\delta,\eta_n,\mathbf{u}_n,t)}
    &\leq \sqrt{E_0(\randreg\delta(0,\cdot),\eta_0, 
        \partial_t\eta_n(0,\cdot),\mathbf{u}_n(0,\cdot))}\\
    &\qquad + \int_0^{T_1}\!\! \frac{1}{\sqrt{2}}  
      \lVert \mathbf{f}_n(s,\cdot) \rVert_{L^2(\Omega_{\randreg\delta(s)})}
    + \frac{1}{2} \lVert g_n(s,\cdot) \rVert_{L^2(M)}\;ds\notag
\end{align}
for all $0<T_1\leq T$. Hence, by the coercivity of the Koiter energy and 
the convergences \eqref{konv_approx_reg_data}, \eqref{eqn:konv_anfangswerte} 
and \eqref{eqn:modification_epsilon} the couples $(\eta_n,\mathbf{u}_n)$ are uniformly 
bounded in $Y^I\times X^I_{\randreg\delta}$ as claimed in the Lemma.
Using the compact embedding $Y^I\hookrightarrow\hookrightarrow 
C^0(\overline{I}\times M)$, we get for a 
subsequence 
\begin{alignat}{2}\label{eqn:approx_konv_eta_n}\begin{aligned}
  \eta_n &\overset{\ast}{\rightharpoondown} \eta
    &\quad &\text{weakly-* in } 
      L^\infty(I,H_0^2(M))\text{ and uniformly in } \overline{I}\times M,\\
  \partial_t\eta_n &\overset{\ast}{\rightharpoondown}\partial_t\eta 
    &&\text{weakly-* in } L^\infty(I,L^2(M)),
\end{aligned}\end{alignat}
thus  $\eta\in Y^I$. By our Korn-type inequality, the spatial extensions of 
$\mathbf{u}_n$, $\nabla\mathbf{u}_n$ and $\mathbf{D}\mathbf{u}_n$ by 
zero are uniformly bounded in $L^\infty(I;L^2(\R^3))$, 
$L^2(I;L^{13/7}(\R^3))$ and $L^2(I;L^2(\R^3))$, respectively. For the
convenience we use the notation $\nabla\mathbf{u}$ and 
$\mathbf{D}\mathbf{u}$ also outside $\Omega_{\randreg\delta}^I$ but we
emphasize that the usual meaning of the symbols only hold on the inside. Hence,
\begin{alignat}{2}\label{eqn:approx_konv_u_n}\begin{aligned}
  \mathbf{u}_n &\overset{\ast}{\rightharpoondown} \mathbf{u} 
    \quad\quad&&\text{weakly-* in } L^\infty(I;L^2(\R^3)),\\
  \nabla\mathbf{u}_n &\rightharpoonup \nabla\mathbf{u}
    &&\text{weakly in } L^2(I;L^{13/7}(\R^3)),\\
  \mathbf{D}\mathbf{u}_n &\rightharpoonup \mathbf{D}\mathbf{u}
    &&\text{weakly in } L^2(I;L^2(\R^3)),
\end{aligned}\end{alignat}
Since  
$\operatorname{div} \mathbf{u}_n = 0$ in $\Omega_{\randreg\delta}^I$ the 
convergences imply $\operatorname{div} \mathbf{u} = 0$ in 
$\Omega_{\randreg\delta}^I$, i.\,e.\ $\mathbf{u}\in X^I_{\randreg\delta}$. 
Moreover, \eqref{eqn:esssup_galerkin} and the lower semi-continuity of  
Koiter's energy and the norms imply \eqref{lem:eqn_esssup_reglin} and the 
uniform bound on $(\eta,\mathbf{u})$ as claimed by the Lemma. To show 
(\ref{eqn:schwache_loesung_entreglin}), we first take 
$\varphi\in C_0^1([0,T))$. Using again Reynolds’ transport 
theorem, the orthogonality assumption for the reference domain, the 
compatibility condition (\ref{approx:gal_rand_m}) and Remark 
\ref{grundlagen:eigenschaften_gamma}, we have
\begin{align*}
    \int_I\intreps \partial_t (\mathbf{u}_n\cdot\varphi\,\mathbf{W}_j)\;dx\,dt
    &= -\int_I\int_{M} \partial_t \eta_n\, \varphi\,
        W_j\,(\partial_t\randreg\delta)\,\gamma(\randreg\delta)\;dA\,dt\\
       &\qquad -\int_{\Omega_{\randreg\delta(0)}}\hspace{-0.7cm} 
        \mathbf{u}_n(0,\cdot)\cdot(\varphi(0)\,\mathbf{W}_j(0,\cdot))\;dx.
\end{align*}
Therefore, by multiplying \eqref{eqn:reg_galerkin} with $\varphi$, 
integration over $I$ and integration by parts with respect to 
time, we get
\begin{align*}
  -&\int_I\intreps \mathbf{u}_n\cdot \partial_t (\varphi\,\mathbf{W}_j)\;dx\,dt
    - \frac{1}{2}\int_I\int_M(\partial_t \eta_n)\,(\partial_t\randreg\delta)
      \,\varphi\,W_j\,\gamma(\randreg\delta)\;dA\,dt\\
  &- 2\int_I\int_M\partial_t\eta_n\,\partial_t(\varphi\,W_j)\;dA\,dt 
    + 2\int_I K(\eta_n,\varphi\,W_j)\;dt
    + 2\int_I\intreps \mathbf{D}\mathbf{u}_n\!:
      \!\mathbf{D}(\varphi\,\mathbf{W}_j)\;dx\,dt\\
  &\quad + \frac{1}{2}\int_I\intreps (\randreg\mathbf{v}\cdot \nabla)\,
      \mathbf{u}_n\cdot (\varphi\,\mathbf{W}_j)\;dx\,dt
    - \frac{1}{2}\int_I\intreps(\randreg\mathbf{v}\cdot \nabla)
      \,(\varphi\,\mathbf{W}_j)\cdot \mathbf{u}_n\;dx\,dt\\ 
  &\qquad = \int_I\int_M g_n\,\varphi\,W_j\;dA\,dt
    + \int_{\Omega_{\randreg\delta(0)}}\hspace{-0.7cm} 
    \mathbf{u}_n(0,\cdot)\cdot(\varphi(0)\,\mathbf{W}_j(0,\cdot))\;dx\\
  &\qquad\qquad + \int_I\intreps \mathbf{f}_n\cdot(\varphi\,\mathbf{W}_j)\;dx\,dt
    + 2\int_M \partial_t\eta_n(0,\cdot)\,\varphi(0)\,W_j(0,\cdot)\;dA.
\end{align*}
By the linearisation of the convective term and \eqref{konv_approx_reg_data}, 
\eqref{eqn:konv_anfangswerte}, \eqref{eqn:approx_konv_eta_n} and 
\eqref{eqn:approx_konv_u_n} we can pass to the limit $n\rightarrow 
\infty$. Thus, $(\eta,\mathbf{u})$ satisfy 
\eqref{eqn:schwache_loesung_entreglin} for all test functions
\begin{align*}
  (b,\boldsymbol\varphi)\in span\big\{ (\varphi W_k,\varphi\mathbf{W}_k)\bigm| 
    \varphi\in C_0^1([0,T)),\,k\in \N\big\}.
\end{align*}
By Proposition \ref{prop:approx_eig_ansatz}, this set is dense in 
$T^I_{\randreg\delta}$, hence $(\eta,\mathbf{u})$ fulfils 
(\ref{eqn:schwache_loesung_entreglin}) for all $(b,\boldsymbol\varphi)
\in T^I_{\randreg\delta}$. Since \eqref{eqn:definition_galerkin_function} 
implies $\eta_n(0,\cdot) = \eta_0$, we deduce that $\eta(0,\cdot) = \eta_0$, 
using \eqref{eqn:approx_konv_eta_n} as well. Moreover, by 
\eqref{approx:gal_rand_m} and Lemma \ref{moving_boundary_data}, it follows 
that $\trregdelta\mathbf{u} = \partial_t \eta\,\boldsymbol{\nu}$ and 
the proof is complete.
\end{proof}
\subsection{The regularised problem}\label{kap:reg_prob}
To restore the coupling between the displacement and the domain as well 
as the convective term, we use a set-valued fixed-point theorem, 
i.\,e.\ the Bohnenblust-Karlin theorem \cite[Theorem 
17.57]{MR2378491}. By choosing a sufficiently small time interval
together with our compactness result, we satisfy the assumptions of 
the fixed-point theorem and get a weak solution of the regularised problem.
\begin{definition}\label{def:schwache_loesung_reg}
  Let $0<\alpha<\kappa$, $0<T<\infty$, $I:=(0,T)$ and let 
  $(\mathbf{f},\, g,\, \mathbf{u}_0,\, \eta_0,\, \eta_1)$ be some 
  admissible data. Let $\epsilon_0 = \epsilon_0(\alpha,\eta_0)$ be as in 
  Section \ref{chapter:entreglin}, and $0<\epsilon <\epsilon_0$. We call 
  the couple $(\eta,\mathbf{u})$ with $\eta \in Y^I$ and $\mathbf{u}\in 
  X^I_{\randreg\eta}$ a \emph{weak solution of the $\epsilon$-regularised 
  problem} if $\lVert \eta \rVert_{L^\infty(I\times M)} <\kappa$, 
  $\eta(0,\cdot)=\eta_0$ on $M$, $\trregeta\mathbf{u} = \partial_t 
  \eta\,\boldsymbol\nu$ on $I\times M$ and 
  \begin{align}\label{eqn:schwache_loesung_reg}\begin{aligned}
  -&\int_I\intregta\mathbf{u}\cdot\partial_t\boldsymbol\varphi\;dx\,dt    
    - \frac{1}{2}\int_I\int_M (\partial_t\eta)\,
      (\partial_t\randreg\eta)\,b\,\gamma(\randreg\eta)\;dA\,dt\\
  & - 2\int_I\int_M \partial_t\eta\,\partial_t b\;dA\,dt
    + 2\,\int_I K(\eta,b)\;dt
    + 2\int_I\intregta \mathbf{D}\mathbf{u}
      :\mathbf{D}\boldsymbol\varphi\;dx\,dt\\
  &\quad+ \frac{1}{2}\int_I\intregta(\randreg\mathbf{u}\cdot \nabla)
      \mathbf{u}\cdot\boldsymbol\varphi\;dx\,dt
    - \frac{1}{2}\int_I\intregta(\randreg\mathbf{u}\cdot \nabla)\boldsymbol
      \varphi\cdot\mathbf{u}\;dx\,dt\\
  &\qquad = \int_I\int_M g\,b\;dA\,dt 
    + \int_I\intregta\mathbf{f}\cdot\boldsymbol\varphi\;dx\,dt 
    +\int_{\Omega_{\randreg\eta(0)}}\hspace{-0.7cm}
      \mathbf{u}_0^\epsilon\cdot \boldsymbol\varphi(0,\cdot)\;dx 
    +2\int_M \eta_1^\epsilon\,b(0,\cdot)\;dA\raisetag{2.91cm}
  \end{aligned}\end{align}
  holds for all $(b,\boldsymbol\varphi)\in T^I_{\randreg\eta}$.
\end{definition}
Now we can formulate the existence result for the $\epsilon$-regularised 
problem.
\begin{lemma}\label{lem:reg_solution}
  Let $0<T_0<\infty$ and let $(\mathbf{f},\, g,\, \mathbf{u}_0,\, 
  \eta_0,\, \eta_1)$ be some admissible data. Let
  \begin{align*}
    0<\alpha := \frac{\lVert \eta_0 \rVert_{L^\infty(M)}+ \kappa}{2} 
      < \kappa.
  \end{align*}
  Then a time $0<T\leq T_0$ exists such that for all $0<\epsilon 
  <\epsilon_0$ a weak solution of the $\epsilon$-regularised problem 
  $(\eta,\mathbf{u})$ exists on the interval $I=(0,T)$ and satisfies 
  $\lVert \eta \rVert_{L^\infty(I\times M)}\leq \alpha$ and for all 
  $0<T_1\leq T$
  \begin{align}\label{lem:eqn_esssup_reg}
    \operatorname*{esssup}_{t\in (0,T_1)}\sqrt{E(\randreg\eta,\eta,\mathbf{u},t)}
      &\leq \sqrt{E_0(\randreg\eta(0,\cdot),\eta_0, 
          \eta_1^\epsilon,\mathbf{u}_0^\epsilon)}\\
      &\qquad + \int_0^{T_1}\!\!\frac{1}{\sqrt{2}} \lVert
        \mathbf{f}(s,\cdot) \rVert_{L^2(\Omega_{\randreg\delta(s)})}
        + \frac{1}{2}\, \lVert g(s,\cdot) \rVert_{L^2(M)}\;ds.\notag
  \end{align}
  In particular $(\eta,\mathbf{u})$ is bounded independently of 
  $0<\epsilon<\epsilon_0$ in $Y^I\times X^I_{\randreg\eta}$. The time $T$ 
  can be chosen independently of the data if an upper bound is given on 
  the norms $\lVert \mathbf{f} \rVert_{L^2(I_0\times B_\kappa)}$, 
  $\lVert g \rVert_{L^2(I_0\times M)}$, $\lVert \eta_0\rVert_{H_0^2(M)}$, 
  $\lVert \eta_1 \rVert_{L^2(M)}$ and $\lVert \mathbf{u}_0 
  \rVert_{L^2(\Omega_{\eta_0})}$ as well as a strictly positive lower 
  bound on $\kappa - \lVert \eta_0 \rVert_{L^\infty(M)}$.
\end{lemma}
\begin{proof}
  We define $I_0 :=(0,T_0)$ and chose a $\delta \in 
  C^0(\overline{I_0}\times M)$ with $\delta = 0$ on $I_0\times 
  \partial M$,  \linebreak $\lVert \delta \rVert_{L^\infty(I_0\times M)}\leq \alpha$, 
  $\delta(0,\cdot) = \eta_0$ on $M$ and a $\mathbf{v}\in 
  L^2(I_0\times\R^3)$. The associated weak solution $(\eta,\mathbf{u})$ 
  of the regularised, decoupled and linearised problem from 
  Lemma \ref{lem:existence_approx_entreglin} satisfies the energy inequality 
  \eqref{lem:eqn_esssup_reg} and, by the coercivity of the Koiter 
  energy, the estimate
  \begin{align*}
    \lVert \eta \rVert_{Y^{I_0}}
      + \lVert \mathbf{u} \rVert_{X^{I_0}_{\randreg\delta}} 
        \leq c_0.
  \end{align*}  
  In here, the constant can be chosen independent of the admissible data, 
  $\delta$, $\mathbf{v}$ and $\epsilon$, if there exits an upper bound on 
  the norms $\lVert f \rVert_{L^2(I_0\times B_\kappa)}$, $\lVert g 
  \rVert_{L^2(I_0\times M)}$, $\lVert \eta_0\rVert_{H_0^2(M)}$, 
  $\lVert \eta_1 \rVert_{L^2(M)}$ and $\lVert \mathbf{u}_0 
  \rVert_{L^2(\Omega_{\eta_0})}$ and as long as $\delta$ satisfies 
  $\delta(0,\cdot) = \eta_0$ and $\lVert \delta 
  \rVert_{L^\infty(I_0\times M)}\leq \alpha$. Extending $\mathbf{u}$ 
  by zero in the spatial direction, we also get the estimate 
  $\lVert \mathbf{u} \rVert_{L^2(I_0\times \R^3)}\leq \sqrt{T_0}\, c_0$.
  Hence, we set $c_1:= \max\{c_0,\sqrt{T_0}\,c_0\}$. By 
  Lemma \ref{lem:hoelder-einbettung}, the embedding $Y^{I_0} 
  \hookrightarrow C^{0,1-\theta}(\overline{I_0}; C^{0}(M))$ is linear 
  and continuous for some $\frac{1}{2}<\theta<1$ with an operator norm 
  denoted by $c_2$. Since $\lVert \eta_0 \rVert_{L^\infty(M)}< \kappa$, 
  we can choose $0<T \leq T_0$ which satisfies
  \begin{align}\label{eqn:lem:existence_approx_reg_T}
    c_1\,c_2\,T^{1-\theta} 
      < \frac{\kappa - \lVert \eta_0 \rVert_{L^\infty(M)}}{2}
  \end{align}
  and depends only on $c_1$, $c_2$, and a positive lower bound on 
  $\kappa - \lVert \eta_0 \rVert_{L^\infty(M)}$. We set $I:=(0,T)$ 
  and notice that this particular choice of $T$, together with the 
  initial condition $\eta(0,\cdot)= \eta_0$ and the $\theta$-H\"older-continuity 
  of $\eta$, implies $\lVert \eta \rVert_{L^{\infty}(I\times M)}< \alpha$.
  We define the space
  \begin{align*}
    Z:=C(\overline{I}\times M)\times L^2(I\times \R^3)
  \end{align*}
  and the non-empty, closed, convex subset
  \begin{align*}
    D:=\left\{ (\delta,\mathbf{v})\in Z\;\left|\; \begin{aligned} 
        &\delta = 0 \text{ on } I\times \partial M,\; 
        \delta(0,\cdot) = \eta_0,\\
        &\lVert \delta \rVert_{L^\infty(I\times M)}\leq \alpha,\,
        \lVert \mathbf{v}\rVert_{L^2(I\times \R^3)}\leq c_1\end{aligned}
      \right\}.\right.
  \end{align*}
  We define $F:D\subset Z \rightarrow 2^Z$ as the set-valued 
  map which assigns to $(\delta,\mathbf{v})\in D$ the set 
  of all weak solutions $(\eta,\mathbf{u})$ of the decoupled, regularised 
  and linearised problem on the interval $I$ to the admissible data 
  $(\mathbf{f},\, g,\, \mathbf{u}_0,\, \eta_0,\, \eta_1)$ and the 
  functions $\delta$ and $\mathbf{v}$, which satisfy for all $0<T_1\leq T$ 
  the inequality \eqref{lem:eqn_esssup_reglin} and the estimates 
  \begin{align}\label{eqn:fixpunkt:2}
    \lVert \eta \rVert_{L^\infty(I\times M)}\leq \alpha,
    \qquad \lVert \mathbf{u}\rVert_{L^2(I\times \R^3)}\leq c_1,
  \end{align}
  where $\mathbf{u}$ is again extended by $\boldsymbol 0$ in the spatial 
  direction. Taking into account the boundary and initial conditions 
  of the weak solution, $F$ maps $D$ into its power set, $F(D)\subset 
  2^D$. To show that $F$ has a fixed point, we use the theorem of 
  Bohnenblust-Karlin, see \cite[Theorem 17.57]{MR2378491}. Therefore, 
  we have to check prerequisites, i.\,e. we have to show that for 
  all $(\delta, \mathbf{v})\in D$ the set $F(\delta,\mathbf{v})$ is 
  non-empty and convex, the graph of $F$ is closed, and the image of 
  $F$ is relatively compact in $Z$.
  
  By our choice of $T$, Lemma \ref{lem:existence_approx_entreglin} 
  implies that for all $(\delta, \mathbf{v})\in D$ the set 
  $F(\delta,\mathbf{v})$ is non-empty. By the linearity of the decoupled, 
  regularised and linearised problem and the coercivity and bilinearity 
  of the Koiter energy, some straightforward computations show that 
  $F(\delta,\mathbf{v})$ is convex. To show the relative compactness of 
  $F(D)$ in $Z$, we take a sequence $(\eta_n,\mathbf{u}_n)_{n\in \N} 
  \subset F(D)$. Thus, there exists a sequence 
  $(\delta_n,\mathbf{v}_n)_{n\in\N}\subset D$ with 
  $(\eta_n,\mathbf{u}_n)\in F(\delta_n,\mathbf{v}_n)$. Since $\epsilon$ 
  is fixed, by Proposition \ref{prop:reg_rand:wohldef} 
  $(\randreg\delta_n)_{n\in \N}$ is bounded in $C^4(\overline{I}\times M)$. 
  Using the Arzela-Ascoli theorem, we get, for a subsequence, 
  \begin{alignat}{2}\label{eqn:lem_reg_konv_reg}\begin{aligned}
    \randreg\delta_n &\rightarrow \xi
      &\quad\quad\quad &\text{in } C^2(\overline{I}\times M). 
  \end{aligned}\end{alignat}
  By \eqref{eqn:def_epsilon_0}, we have 
  $\lVert \randreg\delta_n \rVert_{L^{\infty} (I\times M)}, 
  \lVert \xi \rVert_{L^\infty (I\times M)} \leq 
  \frac{\alpha + \kappa}{2}$. Therefore we can choose a uniform constant 
  in the Korn-type inequality. Using the coercivity of the Koiter 
  energy, the energy inequality \eqref{lem:eqn_esssup_reglin} implies 
  the uniform estimate
  \begin{align}\label{eqn:fixpunkt:3}
    \lVert \eta_n \rVert_{Y^I} + \lVert \mathbf{u}_n 
      \rVert_{X^I_{\randreg\delta_n}} 
        \leq c_3.
  \end{align}
  As in \eqref{eqn:approx_konv_eta_n} and \eqref{eqn:approx_konv_u_n} for 
  a subsequence we deduce
  \begin{alignat}{2}\label{eqn:lem_reg_konv_eta_n}\begin{aligned}
    \eta_n &\overset{\ast}{\rightharpoondown} \eta
      &\quad &\text{weakly-* in } L^\infty(I,H_0^2(M)) 
        \text{ and uniformly in } \overline{I}\times M,\\
    \partial_t\eta_n &\overset{\ast}{\rightharpoondown} \partial_t \eta 
      &&\text{weakly-* in } L^\infty(I,L^2(M)).
  \end{aligned}\end{alignat}
  for some $\eta \in Y^I$, and 
  \begin{alignat}{2}\label{eqn:lem_reg_konv_u_n}\begin{aligned}
    \mathbf{u}_n &\overset{\ast}{\rightharpoondown} \mathbf{u} 
      \quad\quad&&\text{weakly-* in } L^\infty(I;L^2(\R^3)),\\
    \nabla\mathbf{u}_n &\rightharpoonup \nabla \mathbf{u}
      &&\text{weakly in } L^2(I;L^{13/7}(\R^3)),\\
    \mathbf{D}\mathbf{u}_n &\rightharpoonup \mathbf{D}\mathbf{u}
      &&\text{weakly in } L^2(I;L^2(\R^3)),
  \end{aligned}\end{alignat}
  for some $\mathbf{u}\in X^I_{\xi}$, whereby the functions are extended 
  spatially by zero. Hence, we emphasize again that the symbols 
  $\nabla \mathbf{u}$ and $\mathbf{D}\mathbf{u}$ have their usual meaning 
  only on the set $\Omega_\xi^I$. Since the weak solutions satisfy the 
  identity \eqref{eqn:schwache_loesung_entreglin}, we can use the 
  compactness result (Lemma \ref{lem:kompaktheit}) to obtain the strong 
  convergences
  \begin{align}\label{eqn:lem_reg_konv_komp}
    \partial_t\eta_n \rightarrow \partial_t\eta 
      \quad\text{in } L^2(I\times M),\qquad 
    \mathbf{u}_n \rightarrow \mathbf{u} 
      \quad \text{in } L^2(I\times \R^3).
  \end{align}
  By \eqref{eqn:lem_reg_konv_eta_n} and \eqref{eqn:lem_reg_konv_komp}, a 
  subsequence of $(\eta_n,\mathbf{u}_n)_{n\in\N}$ converges towards 
  $(\eta, \mathbf{u})$ in $Z$, therefore $F(D)$ is relatively compact in $Z$.
  It remains to show that $F$ has a closed graph. To this end,  we consider    
  some sequences $(\delta_n, \mathbf{v}_n)_{n\in \N}\subset D$, 
  $(\eta_n,\mathbf{u}_n)_{n\in\N} \subset D$ with 
  $(\eta_n,\mathbf{u}_n)\in F(\delta_n,\mathbf{v}_n)$ and
  \begin{align}\label{eqn:lem_reg_konv_graph}\begin{alignedat}{4}
    \delta_n &\rightarrow \delta 
      &\quad&\text{in } C^0(\overline{I}\times M),&\qquad
    \eta_n &\rightarrow \eta 
      &\quad&\text{in } C^0(\overline{I}\times M),\\
    \mathbf{v}_n &\rightarrow \mathbf{v} 
      &&\text{in } L^2(I\times \R^3),&\qquad
    \mathbf{u}_n &\rightarrow \mathbf{u} 
      &&\text{in } L^2(I\times \R^3)
  \end{alignedat}\end{align}
  for a $(\delta,\mathbf{v})\in Z$ and a $(\eta,\mathbf{u})\in Z$. We will 
  prove that $(\delta,\mathbf{v}),\,(\eta,\mathbf{u})\in D$ and 
  $(\eta,\mathbf{u})\in F(\delta,\mathbf{v})$. Since $D$ is closed, 
  $(\delta,\mathbf{v}),\,(\eta,\mathbf{u})\in D$ follows. With the same 
  arguments as above, we find a subsequence such that 
  \eqref{eqn:lem_reg_konv_reg}, \eqref{eqn:lem_reg_konv_eta_n}, 
  \eqref{eqn:lem_reg_konv_u_n} and \eqref{eqn:lem_reg_konv_komp} hold.
  By the properties of the regularisation operators, 
  \eqref{eqn:lem_reg_konv_reg} and \eqref{eqn:lem_reg_konv_graph} imply
  $\xi = \randreg\delta$ and 
  \begin{alignat}{2}\label{eqn:lem_reg_konv_reg_v_n}\begin{aligned}
    \randreg\mathbf{v}_n &\rightarrow
      \randreg\mathbf{v} &&\text{in } C^0(\overline{I}\times \R^3).
  \end{aligned}\end{alignat}
  Furthermore, from $(\eta_n,\mathbf{u}_n)\in F(\delta_n,
  \mathbf{v}_n)$ and our convergences through the lower semi-continuity 
  of the norms and the continuity of the Koiter energy  we can deduce that 
  for all $0<T_1\leq T$ the couple $(\eta, \mathbf{u})$ satisfies the 
  energy inequality \eqref{lem:eqn_esssup_reglin} and the estimates
  \begin{align*}
    \lVert \eta \rVert_{L^\infty(I\times M)}\leq \alpha,
    \qquad \lVert \mathbf{u}\rVert_{L^2(I\times \R^3)}\leq c_1.
  \end{align*}
  Furthermore, $(\eta_n,\mathbf{u}_n)\in F(\delta_n,\mathbf{v}_n)$ 
  implies $\eta_n(0,\cdot) = \eta_0$ on $M$ and $\trregdeltan \mathbf{u}_n 
  = \partial_t \eta_n\, \boldsymbol{\nu}$ on $I\times M$. By the uniform 
  convergence of $\eta_n$ and Lemma \ref{moving_boundary_data}, we also 
  have $\eta(0,\cdot) = \eta_0$ and $\trregdelta\mathbf{u} = \partial_t 
  \eta\,\boldsymbol{\nu}$. To finally show the identity 
  \eqref{eqn:schwache_loesung_entreglin} for $(\eta, \mathbf{u})$ and all 
  test functions $(b,\boldsymbol\varphi)\in T^I_{\randreg\delta}$, we 
  will again use the fact that $(\eta_n,\mathbf{u}_n)\in 
  F(\delta_n,\mathbf{v}_n)$. Therefore, 
  \begin{align}\label{eqn:approx_graph_weak_sol}\begin{aligned}
  -&\int_I\intrepsn\mathbf{u}_n\cdot\partial_t\boldsymbol\varphi_n\;dx\,dt    
    - \frac{1}{2}\int_I\int_M (\partial_t\eta_n)\,
      (\partial_t\randreg\delta_n)\,b_n\,\gamma(\randreg\delta_n)\;dA\,dt\\
   &- 2\int_I\int_M \partial_t\eta_n\,
      \partial_t b_n\;dA\,dt
    + 2\int_I K(\eta_n,b_n)\;dt
    + 2\int_I\intrepsn \mathbf{D}\mathbf{u}_n
      :\mathbf{D}\boldsymbol\varphi_n\;dx\,dt\\
  &+ \frac{1}{2}\int_I\intrepsn(\randreg\mathbf{v}_n\cdot \nabla)
      \mathbf{u}_n\cdot\boldsymbol\varphi_n\;dx\,dt
    - \frac{1}{2}\int_I\intrepsn(\randreg\mathbf{v}_n\cdot \nabla)\boldsymbol
      \varphi_n\cdot\mathbf{u}_n\;dx\,dt\\
  &\quad= \int_I\int_M g\,b_n\;dA\,dt 
    + \int_I\intreps\mathbf{f}\cdot\boldsymbol\varphi_n\;dx\,dt
    + \int_{\Omega_{\randreg\delta_n(0)}}\hspace{-0.8cm}
      \mathbf{u}_0^\epsilon\cdot \boldsymbol\varphi_n(0,\cdot)\;dx 
    +2\int_M \eta_1^\epsilon\,b_n(0,\cdot)\;dA\raisetag{2.2cm}
  \end{aligned}\end{align}
  holds for all $(b_n,\boldsymbol\varphi_n)\in T^I_{\randreg\delta_n}$. 
  Since we cannot use $(b,\boldsymbol\varphi)\in 
  T^I_{\randreg\delta}$ directly as a test function in 
  \eqref{eqn:approx_graph_weak_sol}, we have to use the special 
  structure of this space. Hence, we take $b\in H^1(I;L^2(M))\cap 
  L^3(I;H_0^2(M))$ with $b(T,\cdot)=0$. By Lemma 
  \ref{lem:fortsetzung_instationaer}, $(b, \frandregdeltan b)\in 
  T^I_{\randreg\delta_n}$. Taking into account the convergences from 
  Lemma \ref{lem:fortsetzung_konvergenzen} and Lemma 
  \ref{lem:fortsetzung_konvergenzen_stationaer} as well as 
  \eqref{eqn:lem_reg_konv_reg}, \eqref{eqn:lem_reg_konv_eta_n}, 
  \eqref{eqn:lem_reg_konv_u_n}, \eqref{eqn:lem_reg_konv_komp} and 
  \eqref{eqn:lem_reg_konv_reg_v_n}, we can pass to the limit $n\rightarrow
  \infty$ in the identity \eqref{eqn:approx_graph_weak_sol} tested with 
  $(b, \frandregdeltan b)$ and obtain that $(\eta, \mathbf{u})$ satisfies
  \eqref{eqn:schwache_loesung_entreglin} for $(b, \frandregdeltan b)$.
  Due to our definition of $T^I_{\randreg\delta}$, it remains to show 
  that $(\eta,\mathbf{u})$ satisfies the identity 
  \eqref{eqn:schwache_loesung_entreglin} for all test functions 
  $(0,\boldsymbol\varphi)$ with $\boldsymbol\varphi\in W_{\randreg\delta}$, 
  $\boldsymbol\varphi=\boldsymbol 0$ in a neighbourhood of the moving 
  boundary, $\operatorname{div}\boldsymbol\varphi= 0$ and 
  $\boldsymbol\varphi(T,\cdot)=\boldsymbol 0$. For such a test function 
  $(0,\boldsymbol\varphi)$, the uniform convergence of $\randreg\delta_n$ 
  towards $\randreg \delta$ implies $supp\,\boldsymbol\varphi\subset 
  \Omega_{\randreg\delta_n}^I$ for $n$ big enough. Therefore, silently 
  extending this function by zero in the spatial direction, we have that 
  $(0,\boldsymbol\varphi)\in T^I_{\randreg\delta_n}$ is an admissible 
  test function for the identities \eqref{eqn:approx_graph_weak_sol} for 
  $n$ big enough. Using again the convergences from above, we can pass to the 
  limit $n\rightarrow \infty$ to show that $(\eta,\mathbf{u})$ is a weak 
  solution of the decoupled, regularised and linearised problem, 
  i.\,e.\ $(\eta,\mathbf{u})\in F(\delta,\mathbf{v})$. Hence, 
  $F$ has a closed graph and we can use theorem of Bohnenblust-Karlin, 
  \cite[Theorem 17.57]{MR2378491}), to obtain a fixed point and therefore 
  a weak solution of the regularised problem. The energy inequality 
  \eqref{lem:eqn_esssup_reg} then follows from the definition of the map 
  $F$.
\end{proof}
\subsection{Limiting Process}\label{kap:main}
Taking weak solutions of the regularised problem to the parameter 
$\epsilon$, we can now pass to the limit $\epsilon\rightarrow 0$ to 
obtain a weak solution for our problem 
\eqref{eqn:stokes}--\eqref{eqn:gln:anfangswert}.
%
\begin{proof} (of Theorem \ref{main_theorem}) We choose some 
$T_0\in (0,\infty)$ and set $I_0:=(0,T_0)$. Further we 
  set $\alpha := \frac{1}{2}(\lVert \eta_0 \rVert_{L^\infty(M)}+ \kappa)$, 
  $\epsilon_n:=\frac{1}{n}$ and take $n\in\N$ big enough so that 
  $\epsilon_n< \epsilon_0$ (with $\epsilon_0=\epsilon_0(\alpha,\eta_0)$ 
  as in Section \ref{chapter:entreglin}) holds. By Lemma 
  \ref{lem:reg_solution}, we get the existence of an 
  time interval $I=(0,T)$ with $0<T\leq T_0$ independently of $n$ and weak 
  solutions $(\eta_n,\mathbf{u}_n)$ of the regularised problem to the 
  regularization parameter $\epsilon_n$ fulfilling $\lVert \eta_n 
  \rVert_{L^\infty(I\times M)}\leq \alpha$ and \eqref{lem:eqn_esssup_reg} 
  for all $0<T_1\leq T$. Using the coercivity of $K$ and 
  \eqref{eqn:modification_epsilon} we deduce the estimate
  \begin{align}\label{eqn:apriori_entreglin_zitiert}
    \lVert \eta_n \rVert_{Y^I}
      + \lVert \mathbf{u}_n \rVert_{X^I_{\randregn\eta_n}} 
      \leq c
  \end{align}
  with a constant $c$ independent of $n$. This, together with 
  the compact embedding $Y^I \hookrightarrow\hookrightarrow 
  C^0(\overline{I}\times M)$, implies that
  \begin{alignat}{2}\label{eqn:bew_main_konv_eta_n}\begin{aligned}
    \eta_n &\overset{\ast}{\rightharpoondown} \eta
      &\quad &\text{weakly-* in } L^\infty(I,H_0^2(M))
      \text{ and uniformly in } \overline{I}\times M,\\
    \partial_t\eta_n &\overset{\ast}{\rightharpoondown} \partial_t \eta 
      &&\text{weakly-* in } L^\infty(I,L^2(M))
  \end{aligned}\end{alignat}
  for a subsequence and for an $\eta\in Y^I$. Since by Remark \ref{bem:randreg} 
  $\randregn\eta_n$ is uniformly bounded in 
  $\widetilde{Y}^I$, we similarly get a subsequence with 
  \begin{alignat}{2}\label{eqn:bew_main_konv_regeta_n}\begin{aligned}
    \randregn\eta_n &\overset{\ast}{\rightharpoondown} \zeta
      &\quad &\text{weakly-* in } L^\infty(I,H^2(M)) 
        \text{ and uniformly in } \overline{I}\times M,\\
    \partial_t\randregn\eta_n &\overset{\ast}{\rightharpoondown} \partial_t\zeta 
      &&\text{weakly-* in } L^\infty(I,L^2(M)).
  \end{aligned}\end{alignat}
  Taking Proposition \ref{prop:reg_rand:wohldef} into account, we have
  \begin{align*}
    \lVert \randregn\eta_n - \eta\rVert_{L^\infty(I\times M)}
      &\leq \lVert \eta_n -\eta\rVert_{L^\infty(I\times M)}
        + \lVert \randregn\eta - \eta\rVert_{L^\infty(I\times M)}
        + 2\,(\epsilon_n)^\frac{1}{2}.
  \end{align*}
  This implies $\zeta = \eta$, by the convergence properties of the 
  regularization operator from Lemma \ref{lem:reg_rand:konv}. As usual, 
  to treat the convective term in the fluid part, we need 
  further information from the parabolic embeddings. By our choice of 
  $\epsilon_0$ we have  $\lVert \randregn\eta_n 
  \rVert_{L^\infty(I\times M)} \leq \frac{\alpha+\kappa}{2}<\kappa$.
  Hence, using our Korn-type inequality and the Sobolev embeddings, the 
  sequence $(\mathbf{u}_n)_{n\in\N}$ is uniformly bounded in 
  $L^2(I;L^{11/2}(\Omega_{\randregn\eta_n(t)}))$ and in 
  $L^2(I;W^{1,r}(\Omega_{\randregn\eta_n(t)}))$ for all $1\leq r<2$.
  The parabolic embedding (see \cite[Korollar 2.10]{MR1409366}) and 
  the H\"older interpolation (see \cite[Proposition 3.1]{MR1230384}) 
  yield
  \begin{align*}
    L^\infty(I,L^2(\Omega_{\delta_n(t)}))\cap
      L^2(I,L^{11/2}(\Omega_{\randregn\eta_n(t)}))
    &\hookrightarrow L^q(I,L^q(\Omega_{\randregn\eta_n(t)}))
  \end{align*}
  for $q=16/5$. Extending $\mathbf{u}_n$, $\nabla \mathbf{u}_n$ 
  and $\mathbf{D}\mathbf{u}_n$ by zero to $I\times \R^3$, we get, for 
  a further subsequence, the convergences
  \begin{alignat}{2}\label{eqn:bew_main_konv_u_n}\begin{aligned}
    \mathbf{u}_n &\overset{\ast}{\rightharpoondown} \mathbf{u} 
      \quad&&\text{weakly-* in } L^\infty(I;L^2(\R^3)),\\
    \mathbf{u}_n &\rightharpoonup \mathbf{u} 
      \quad&&\text{weakly in } L^q(I;L^q(\R^3))
      \text{ and weakly in } L^2(I;L^{11/2}(\R^3)),\\
    \nabla\mathbf{u}_n &\rightharpoonup \nabla \mathbf{u}
      &&\text{weakly in } L^2(I;L^{20/11}(\R^3)),\\
    \mathbf{D}\mathbf{u}_n &\rightharpoonup \mathbf{D}\mathbf{u}
      &&\text{weakly in } L^2(I;L^2(\R^3)),
  \end{aligned}\end{alignat}
  for an $\mathbf{u}\in X^I_\eta$ which is also extended by zero in an 
  appropriate fashion. Using the compactness result (Lemma 
  \ref{lem:kompaktheit}), we deduce that
  \begin{align}\label{eqn:bew_main_konv_komp}
    \partial_t\eta_n \rightarrow \partial_t\eta 
      \quad\text{in } L^2(I\times M),\qquad 
    \mathbf{u}_n \rightarrow \mathbf{u} 
      \quad\text{in } L^2(I\times \R^3).
  \end{align}
  Again by H\"older-Interpolation, from \eqref{eqn:bew_main_konv_u_n}, 
  \eqref{eqn:bew_main_konv_komp} we deduce the 
  strong convergences in $L^2(I;L^5(\R^3))$ and $L^3(I;L^3(\R^3))$ for 
  the sequence $(\mathbf{u}_n)_{n\in\N}$. With the properties of the 
  smoothing operator (see \cite[Proposition II.2.25]
  {MR2986590} and \cite[Th{\'e}or\`{e}me 1.8.1]{droniou}), this implies 
  the convergences
  \begin{align}\label{eqn:bew_main_konv_regu}
    \randregn\mathbf{u}_n \rightarrow \mathbf{u} 
      \quad\text{in } L^3(I\times \R^3),\qquad 
    \randregn\mathbf{u}_n \rightarrow \mathbf{u} 
      \quad\text{in } L^2(I;L^5(\R^3)).
  \end{align}
  Since $(\eta_n,\mathbf{u}_n)$ are weak solutions to the regularised 
  problem, we have $\eta_n(0,\cdot) = \eta_0$ and 
  $\trregnetan \mathbf{u}_n = \partial_t \eta_n\, \boldsymbol{\nu}$.
  By the uniform convergence from \eqref{eqn:bew_main_konv_eta_n}
  and Lemma \ref{moving_boundary_data}, $\eta(0,\cdot) = \eta_0$ and 
  $\treta \mathbf{u} = \partial_t \eta\, \boldsymbol{\nu}$ follows. Also, 
  due to the lower semi-continuity of the norms and the Koiter energy, 
  the energy estimates \eqref{eqn:modification_epsilon} of the 
  regularised solutions imply the energy estimate \eqref{lem:eqn_esssup}
  for all $0<T_1\leq T$. To show the integral identity 
  \eqref{eqn:schwache_formulierung_def} for all $(b,\boldsymbol\varphi) 
  \in T^I_\eta$, we have to use the special structure of the space of test 
  functions. Let $b\in H^1(I;L^2(M))\cap L^3(I;H_0^2(M))$ with 
  $b(T,\cdot)=0$. Then, using the properties of our extension operator 
  $\frandregnetan$, i.\,e.\ Lemma \ref{lem:fortsetzung_instationaer}, we 
  have $(b, \frandregnetan b) \in T^I_{\randregn\eta_n}$. Hence for all 
  $n\in \N$ the tuple $(b, \frandregnetan b)$ is an admissible test 
  function for the corresponding integral identity of the regularised 
  weak solution \eqref{eqn:schwache_loesung_reg}. By Lemma 
  \ref{lem:fortsetzung_konvergenzen}, $\frandregnetan b$ converges to 
  $\feta b$ in $L^\infty(I,L^4(B_\alpha))$. Thanks to the convergence 
  of $\randregn\mathbf{u}_n \rightarrow \mathbf{u}$ in $L^2(I;L^5(\R^3))$ 
  by \eqref{eqn:bew_main_konv_regu} and the weak convergence of $\nabla
  \mathbf{u}_n \rightharpoonup \nabla\mathbf{u}$ in $L^2(I;L^{20/11}(\R^3))$
  by \eqref{eqn:bew_main_konv_u_n}, we get for the first part of the 
  convective term
  \begin{align*}
    \int_I\intregntan(\randregn\mathbf{u}_n\cdot \nabla)
      \mathbf{u}_n\cdot\frandregnetan b\;dx\,dt
    \rightarrow
    \int_I\intetat (\mathbf{u}\cdot \nabla)
      \mathbf{u}\cdot\feta b\;dx\,dt.
  \end{align*}
  In here, we also used the spatial extension by zero. Arguing similarly 
  for the rest of the terms in \eqref{eqn:schwache_loesung_reg} tested 
  with $(b, \frandregnetan b)$, i.\,e.\ taking into account the 
  convergences 
  \eqref{eqn:bew_main_konv_eta_n}--\eqref{eqn:bew_main_konv_regu} as well 
  as the convergences from Lemma \ref{lem:fortsetzung_konvergenzen} and 
  Lemma \ref{lem:fortsetzung_konvergenzen_stationaer}, the integral 
  identity \eqref{eqn:schwache_formulierung_def} holds for all 
  $(b, \frandregdeltan b)$ with $b\in H^1(I;L^2(M))\cap L^3(I;H_0^2(M))$ 
  and $b(T,\cdot)=0$. By the definition of our space of test functions, 
  it only remains to show the identity \eqref{eqn:schwache_formulierung_def} 
  for all $(0,\boldsymbol\varphi)$ with $\boldsymbol\varphi\in W_{\eta}$, 
  $\boldsymbol\varphi=\boldsymbol 0$ in a neighbourhood of the moving 
  boundary, $\operatorname{div}\boldsymbol\varphi= 0$ and $\boldsymbol\varphi(T,\cdot)
  =\boldsymbol 0$. For such a test function $(0,\boldsymbol\varphi)$, the 
  uniform convergence of $\randregn\eta_n$ towards $\eta$ implies 
  $supp\,\boldsymbol\varphi\subset \Omega_{\randregn\eta_n}^I$ for $n$ 
  big enough. Hence, silently extending this function by zero in the 
  spatial direction, we have that $(0,\boldsymbol\varphi)\in  
  T^I_{\randregn\eta_n}$ is an admissible test function for the 
  regularised identities \eqref{eqn:schwache_loesung_reg} for 
  $n$ big enough. Using again the convergences 
  \eqref{eqn:bew_main_konv_eta_n}--\eqref{eqn:bew_main_konv_regu}, we 
  can pass to the limit $n\rightarrow \infty$ to show that 
  $(\eta,\mathbf{u})$ is a weak solution to our problem on the interval 
  $I=(0,T)$.
  
  To extend this local existence to an global result, we observe 
  that by the coercivity of the Koiter energy, the energy 
  inequality \eqref{lem:eqn_esssup} implies the estimate
  \begin{align}\label{eqn:apriori_fortsetzen}
    \lVert \eta \rVert_{Y^I}
      + \lVert \mathbf{u} \rVert_{X^I_{\eta}} 
        \leq c,
  \end{align}
  where the constant $c$ only depends on $\lVert 
  \mathbf{u}_0\rVert_{L^2(\Omega_{\eta_0})}$, $\lVert \eta_1 
  \rVert_{L^2(M)}$, $\lVert \mathbf{f} 
  \rVert_{L^2((0,\infty)\times B_\kappa)}$ and $\lVert g 
  \rVert_{L^2((0,\infty)\times M)}$. By the 
  compatibility condition $\treta \mathbf{u} = 
  \partial_t\eta\,\boldsymbol\nu$ and $\lVert \eta 
  \rVert_{L^\infty(I\times M)} \leq \alpha <\kappa$, the estimate 
  \eqref{eqn:apriori_fortsetzen} implies that the  quintuple 
  $(\mathbf{f}(\cdot - s),\, g(\cdot - s),\, \mathbf{u}(s),\, \eta(s),\, 
  \partial_t\eta(s))$ is an admissible data for almost all $s\in I$. 
  Repeating the fist part of the proof, we get for almost all 
  $s\in I$ a local weak solution to this data, whereby the length of 
  the time interval $(0,\widetilde T)$ can be chosen independent of $s$
  by Lemma \ref{lem:reg_solution}. Choosing an 
  $s_0\in I$ outside of an appropriate null set close enough to $T$ and 
  shifting the local weak solution by $s_0$, we get the existence of a 
  weak solution $(\eta^0,\mathbf{u}^0)$ on the time interval $(s_0,T+ 
  \widetilde T /2)$. Considering the initial conditions and an 
  the alternative integral identity \eqref{lem:komp_equation_neu}, we 
  see that the combined function (again denoted by $(\eta,\mathbf{u})$) 
  is a weak solution to \eqref{eqn:stokes}--\eqref{eqn:gln:anfangswert} 
  on the time interval $(0,T + \widetilde T / 2)$. Since \eqref{lem:eqn_esssup} 
  holds for $(\eta^0,\mathbf{u}^0)$ as well as for $(\eta,\mathbf{u})$ 
  (for clarity we denote the initial energy for $(\eta^0,\mathbf{u}^0)$ by 
  $E_0(s_0)$ and notice that $E_0(s_0)\leq \operatorname{esssup}_{t\in 
  (0,T)} E(t)$), we have
  \begin{align*}
    \operatorname*{esssup}_{t\in (s_0,T + \widetilde T / 2)}
        \sqrt{E(t)}
      &\leq \sqrt{E_0(s_0)}  + \frac{1}{\sqrt{2}} \int_0^{T + \widetilde T / 2-s_0} \lVert 
        \mathbf{f}(s-s_0,\cdot) \rVert_{L^2(\Omega_{\eta^0(s)})}\;ds\\
        &\qquad + \frac{1}{2}\, 
          \int_0^{T + \widetilde T / 2-s_0} \lVert g(s-s_0,\cdot) \rVert_{L^2(M)}\;ds\\
      &\leq \sqrt{E_0} + \int_{0}^{T_1}\!\! \frac{1}{\sqrt{2}} \lVert 
        \mathbf{f}(s,\cdot) \rVert_{L^2(\Omega_{\eta(s)})}
        + \frac{1}{2} \lVert g(s,\cdot) 
          \rVert_{L^2(M)}\;ds.
  \end{align*}
  Hence, the combined solution also fulfils the energy inequality 
  \eqref{lem:eqn_esssup} and therefore the estimate 
  \eqref{eqn:apriori_fortsetzen} with the same constant. 
  Iterating this procedure, we obtain a maximal time $0<T^\ast\leq 
  \infty$ such that for all $0< T < T^\ast$ a weak solution 
  $(\eta,\mathbf{u})$ of \eqref{eqn:stokes}--\eqref{eqn:gln:anfangswert} 
  exists on the interval $I:=(0,T)$, fulfilling the energy inequality 
  \eqref{lem:eqn_esssup}. Since by \eqref{eqn:apriori_fortsetzen} the 
  norms of the initial data for the extension stay uniformly bounded, 
  Lemma \ref{lem:reg_solution} implies that $T^\ast < \infty$ is only 
  possible if the norm of the displacement converges to $\kappa$. Since 
  small enough data implies a uniform bound for the displacement (see 
  \eqref{eqn:lem:existence_approx_reg_T}), we have $T^\ast = \infty$ 
  for this case.
\end{proof}

\section{Appendix}\label{appendix}
For the convenience of the reader, we include the proof of the compactness 
result. As already mentioned, we will need a modified Ehrling's Lemma.
\begin{lemma}\label{lem:ehrling}
  Let $0<\alpha<\kappa$, $N\in \N$ and, for the second inequality,
  $\delta \in C^4(M)$ with
  $\lVert \delta \rVert_{L^\infty(M)}<\kappa$.  Moreover, for
  $\boldsymbol\varphi\in X(\Omega)$, we extend
  $\Tdelta\boldsymbol\varphi$ by $\boldsymbol 0$ to $\R^3$. Then for
  all $\varepsilon >0$ there exists a constant $c>0$ such that for
  all $\mathbf{v}\in V_2(\Omega_\eta)$,
  $\tilde{\mathbf{v}} \in V_2(\Omega_{\tilde\eta})$ and all
  $\eta, \tilde\eta \in H^2(M)$ with
  $\lVert \eta\rVert_{H^2(M)} + \lVert \tilde\eta\rVert_{H^2(M)} \leq
  N$
  and $\lVert \eta \rVert_{L^\infty(M)}\leq \alpha$,
  $\lVert \tilde \eta \rVert_{L^\infty(M)}\leq \alpha$ it holds that
  \begin{align*}
    &\sup_{\lVert b\rVert_{L^2(M)}\leq 1} 
      \Big(\int_{\Omega_\eta}\hspace{-0.2cm} \mathbf{v}\cdot \feta b\;dx
        - \int_{\Omega_{\tilde \eta}}\hspace{-0.2cm} \tilde{\mathbf{v}}\cdot \ftildeeta b\;dx
        + \int_{M}\big(\treta\mathbf{v}\cdot \boldsymbol\nu 
          - \trtildeeta\tilde{\mathbf{v}}\cdot \boldsymbol\nu\big)\,b\;dA\Big)\\
      &\quad\leq c\,\sup_{\lVert b \rVert_{H_0^2(M)}\leq 1}
          \Big(\int_{\Omega_\eta}\hspace{-0.2cm} \mathbf{v}\cdot\feta b\;dx
          - \int_{\Omega_{\tilde \eta}}\hspace{-0.2cm} \tilde{\mathbf{v}}\cdot \ftildeeta b\;dx 
          + \int_{M}\big(\treta\mathbf{v}\cdot \boldsymbol\nu 
          - \trtildeeta\tilde{\mathbf{v}}\cdot \boldsymbol\nu\big)\,b\;dA\Big)\\
        &\qquad\quad+ \varepsilon\,\big(\lVert \mathbf{v}\rVert_{V_2(\Omega_\eta)}
        + \lVert \tilde{\mathbf{v}}\rVert_{V_2(\Omega_{\tilde \eta})}\big)
  \end{align*}
  as well as
  \begin{align*}
    &\sup_{\lVert \boldsymbol\varphi\rVert_{H_M(\Omega)}\leq 1} 
      \Big(\int_{\Omega_\eta}\hspace{-0.2cm}\mathbf{v}\cdot\Tdelta\boldsymbol\varphi\;dx
        - \int_{\Omega_{\tilde \eta}}\hspace{-0.2cm} \tilde{\mathbf{v}}
          \cdot\Tdelta\boldsymbol\varphi\;dx\Big)\\
      &\quad\leq c\,\sup_{\lVert \boldsymbol\varphi\rVert_{X(\Omega)}\leq 1}
          \Big(\int_{\Omega_\eta}\hspace{-0.2cm}\mathbf{v}\cdot\Tdelta\boldsymbol\varphi\;dx
        - \int_{\Omega_{\tilde \eta}}\hspace{-0.2cm} \tilde{\mathbf{v}}
          \cdot\Tdelta\boldsymbol\varphi\;dx\Big)
        + \varepsilon\,\big(\lVert \mathbf{v}\rVert_{V_2(\Omega_\eta)}
        + \lVert \tilde{\mathbf{v}}\rVert_{V_2(\Omega_{\tilde \eta})}\big).
  \end{align*}
\end{lemma}
\begin{proof}
  Assume that the first claim is 
  wrong. Then there exist $\varepsilon>0$ and bounded sequences 
  $(\eta_n)_{n\in \N}, (\tilde\eta_n)_{n\in \N}\subset H^2(M)$, 
  $(\mathbf{v}_n)_{n\in\N}\subset V_2(\Omega_{\eta_n})$ and 
  $(\tilde{\mathbf{v}}_n)_{n\in \N} \subset V_2(\Omega_{\tilde\eta_n})$, 
  which satisfy $\lVert \eta_n \rVert_{L^\infty(M)} \leq \alpha$, 
  $\lVert \tilde \eta_n \rVert_{L^\infty(M)}\leq \alpha$ and, after some 
  scaling, 
  \begin{align*}
    \lVert \mathbf{v}_n\rVert_{V_2(\Omega_{\eta_n})} + \lVert 
    \tilde{\mathbf{v}}_n\rVert_{V_2(\Omega_{\tilde \eta_n})} = 1
  \end{align*}
  as well as
  \begin{align}\label{lem:erhling_1}\begin{aligned}
    &\sup_{\lVert b\rVert_{L^2(M)}\leq 1} 
      \Big(\int_{\Omega_{\eta_n}}\hspace{-0.2cm} \mathbf{v}_n\cdot \fetan b\;dx
        - \int_{\Omega_{{\tilde \eta}_n}}\hspace{-0.2cm} \tilde{\mathbf{v}}_n\cdot \ftildeetan b\;dx
        + \int_{M}\big(\tretan\mathbf{v}_n\cdot \boldsymbol\nu 
          - \trtildeetan\tilde{\mathbf{v}}_n\cdot \boldsymbol\nu\big)\,b\;dA\Big)\\
      &\quad >\varepsilon + n\,\sup_{\lVert b \rVert_{H_0^2(M)}\leq 1}
          \Big(\int_{\Omega_{\eta_n}}\hspace{-0.2cm} \mathbf{v}_n\cdot\fetan b\;dx
          - \int_{\Omega_{\tilde \eta_n}}\hspace{-0.2cm} \tilde{\mathbf{v}}_n\cdot \ftildeetan b\;dx\\
      &\hspace*{4cm} + \int_{M}\big(\tretan\mathbf{v}_n\cdot \boldsymbol\nu 
          - \trtildeetan\tilde{\mathbf{v}}_n\cdot \boldsymbol\nu\big)\,b\;dA\Big).
      \raisetag{1.2cm}
  \end{aligned}\end{align}
  By the Korn-type inequality and Definition \ref{def:treta}, the 
  sequences $(\tretan\mathbf{v}_n)_{n\in\N}$, $(\trtildeetan 
  \tilde{\mathbf{v}}_n)_{n\in\N}$ are bounded in $H^{\frac{2}{5};
  \frac{5}{3}}(M)$. By Sobolev's embedding theorem we get a subsequence
  satisfying
  \begin{alignat*}{4}
    \tretan\mathbf{v}_n\cdot\boldsymbol{\nu} 
        &\rightarrow d &\quad&\text{in } L^2(M),
    \quad\quad&\eta_n &\rightharpoonup \eta 
        &\quad&\text{weakly in } H^2(M),\\
    \trtildeetan\tilde{\mathbf{v}}_n\cdot\boldsymbol{\nu} 
        &\rightarrow \tilde d &&\text{in } L^2(M),
    &\tilde\eta_n &\rightharpoonup \tilde\eta 
        &&\text{weakly in } H^2(M).
  \end{alignat*}
  In particular, $(\eta_n)_{n\in\N}$ and $(\tilde \eta_n)_{n\in \N}$ 
  converge uniformly in $M$, and therefore 
  $\lVert \eta\rVert_{L^\infty(M)}\leq \alpha$, 
  $\lVert \tilde \eta \rVert_{L^\infty(M)}\leq \alpha$.
  Taking, as usual, the cut-off function $\beta$ uniformly for all 
  transformations, by Lemma \ref{lem:einbettung_omega_eta} the sequences 
  $\mathbf{w}_n:=\mathbf{v}_n \circ \Psi_{\eta_n}$ and 
  $\tilde{\mathbf{w}}_n:=\tilde{\mathbf{v}}_n \circ \Psi_{\tilde\eta_n}$
  are bounded in $W^{1,\frac{3}{2}}(\Omega)$. By Sobolev's embedding, 
  for a subsequence we have $\mathbf{w}_n\rightarrow \mathbf{w}$ and 
  $\tilde{\mathbf{w}}_n \rightarrow \tilde{\mathbf{w}}$ in $L^{5/2}(\Omega)$. 
  We extend the functions $\mathbf{w}\circ \Psi_\eta^{-1}$, 
  $\mathbf{w}_n\circ \Psi_{\eta_n}^{-1}$  and $\mathbf{v}_n$ by 
  $\boldsymbol 0$ to $\R^3$ and get, using Lemma 
  \ref{lem:einbettung_omega_eta},
  \begin{align*}
    \lVert \mathbf{v}_n - \mathbf{w}\circ \Psi_\eta^{-1}\rVert_{L^2(\R^3)}
    &= \lVert \mathbf{w}_n\circ\Psi_{\eta_n}^{-1} - \mathbf{w}\circ \Psi_\eta^{-1}\rVert_{L^2(\R^3)}\\
    &\leq \lVert \big(\mathbf{w}_n - \mathbf{w}\big)\circ\Psi_{\eta_n}^{-1}\rVert_{L^2(\Omega_{\eta_n})}
      + \lVert \mathbf{w}\circ\Psi_{\eta_n}^{-1} - \mathbf{w}\circ \Psi_\eta^{-1}\rVert_{L^2(\R^3)}\\
    &\leq c\,\lVert \mathbf{w}_n - \mathbf{w}\rVert_{L^{5/2}(\Omega)}
      + \lVert \mathbf{w}\circ\Psi_{\eta_n}^{-1} - \mathbf{w}\circ \Psi_\eta^{-1}\rVert_{L^2(\R^3)}.
  \end{align*}
  Taking Lemma \ref{lem:konv_trafo} into account, we deduce that 
  $\mathbf{v}_n\rightarrow \mathbf{w}\circ \Psi_\eta^{-1}=: \mathbf{v}$ 
  in $L^2(\R^3)$. Analogously we get $\tilde{\mathbf{v}}_n\rightarrow 
  \tilde{\mathbf{w}}_n\circ \Psi_{\tilde\eta}^{-1}=: 
  \tilde{\mathbf{v}}$ in $L^2(\R^3)$. By Lemma 
  \ref{lem:fortsetzung_konvergenzen_stationaer}, the equation
  \begin{align*}
    \int_{\Omega_{\eta_n}}\hspace{-0.2cm} \mathbf{v}_n\cdot \fetan b\;dx
        - \int_{\Omega_{{\tilde \eta}_n}}\hspace{-0.2cm} \tilde{\mathbf{v}}_n\cdot \ftildeetan b\;dx
        + \int_{M}\big(\tretan\mathbf{v}_n\cdot \boldsymbol\nu 
          - \trtildeetan\tilde{\mathbf{v}}_n\cdot \boldsymbol\nu\big)\,b\;dA
  \end{align*}
  converges to
  \begin{align*}
    \int_{\Omega_{\eta}} \mathbf{v}\cdot \feta b\;dx
        - \int_{\Omega_{{\tilde \eta}}} \tilde{\mathbf{v}}\cdot \ftildeeta b\;dx
        + \int_{M}\big(d - \tilde d\big)\,b\;dA
  \end{align*}  
  uniformly with respect to $\lVert b\rVert_{L^2(M)}\leq 1$. Therefore the left-hand 
  side of \eqref{lem:erhling_1} is bounded and 
  \begin{align*}
  \sup_{\lVert b \rVert_{H_0^2(M)}\leq 1}
          \Big(\int_{\Omega_{\eta_n}}\hspace{-0.2cm} \mathbf{v}_n\cdot\fetan b\;dx
          - \int_{\Omega_{\tilde \eta_n}}\hspace{-0.2cm} \tilde{\mathbf{v}}_n\cdot \ftildeetan b\;dx
          + \int_{M}\big(\tretan\mathbf{v}_n\cdot \boldsymbol\nu 
          - \trtildeetan\tilde{\mathbf{v}}_n\cdot \boldsymbol\nu\big)\,b\;dA\Big)
  \end{align*}
  converges to $0$. Since $H_0^2(M)$ is dense in $L^2(M)$, the left-hand 
  side of \eqref{lem:erhling_1} converges to $0$ as well, a contradiction 
  to $\varepsilon>0$.

  Analogously, we assume that the second claim is wrong. Then there 
  exist $\varepsilon>0$ and bounded sequences $(\eta_n)_{n\in \N}, 
  (\tilde\eta_n)_{n\in \N}\subset H^2(M)$, $(\mathbf{v}_n)_{n\in\N} 
  \subset V_2(\Omega_{\eta_n})$ and $(\tilde{\mathbf{v}}_n)_{n\in \N} 
  \subset V_2(\Omega_{\tilde\eta_n})$, which satisfy  
  $\lVert \eta_n \rVert_{L^\infty(M)} \leq \alpha$, $\lVert \tilde \eta_n
  \rVert_{L^\infty(M)}\leq \alpha$ and, after some scaling, 
  \begin{align*}
    \lVert \mathbf{v}_n\rVert_{V_2(\Omega_{\eta_n})} + \lVert 
  \tilde{\mathbf{v}}_n\rVert_{V_2(\Omega_{\tilde \eta_n})} = 1
  \end{align*}
  as well as
  \begin{align}\label{eqn:ehrling_2}\begin{aligned}
    \sup_{\lVert \boldsymbol\varphi\rVert_{H_M(\Omega)}\leq 1} 
      &\Big(\int_{\Omega_{\eta_n}}\hspace{-0.2cm}\mathbf{v}_n\cdot\Tdelta\boldsymbol\varphi\;dx
        - \int_{\Omega_{\tilde \eta_n}}\hspace{-0.2cm} \tilde{\mathbf{v}}_n
          \cdot\Tdelta\boldsymbol\varphi\;dx\Big)\\
      &> \varepsilon + n\,\sup_{\lVert \boldsymbol\varphi\rVert_{X(\Omega)}\leq 1}
          \Big(\int_{\Omega_{\eta_n}}\hspace{-0.2cm}\mathbf{v}_n\cdot\Tdelta\boldsymbol\varphi\;dx
        - \int_{\Omega_{\tilde \eta_n}}\hspace{-0.2cm} \tilde{\mathbf{v}}_n
          \cdot\Tdelta\boldsymbol\varphi\;dx\Big).
  \end{aligned}\end{align}
  As in the first part of the proof, we find a subsequence with
  $\mathbf{v}_n \rightarrow \mathbf{v}$, $\tilde{\mathbf{v}}_n 
  \rightarrow \tilde{\mathbf{v}}$ in $L^2(\R^3)$. By the continuity
  of the Piola-transform (notice $\delta\in C^4(M)$), for all 
  $\boldsymbol\varphi\in H_M(\Omega)$ with $\lVert 
  \boldsymbol\varphi\rVert_{H_M(\Omega)}\leq 1$ it follows
  \begin{align*}
    \Big|\int_{\Omega_{\eta_n}}\hspace{-0.2cm}\mathbf{v}_n\cdot\Tdelta\boldsymbol\varphi\;dx
        &- \int_{\Omega_{\tilde \eta_n}}\hspace{-0.2cm} \tilde{\mathbf{v}}_n
          \cdot\Tdelta\boldsymbol\varphi\;dx
        - \int_{\R^3}\mathbf{v}\cdot\Tdelta\boldsymbol\varphi\;dx
        + \int_{\R^3} \tilde{\mathbf{v}}\cdot\Tdelta\boldsymbol\varphi\;dx
        \Big|\\
     &\leq \lVert \mathbf{v}_n - \mathbf{v}\rVert_{L^2(\R^3)}
      \lVert \Tdelta\boldsymbol\varphi\rVert_{L^2(\R^3)}
        + \lVert \tilde{\mathbf{v}}_n - \tilde{\mathbf{v}}\rVert_{L^2(\R^3)}
        \lVert \Tdelta\boldsymbol\varphi\rVert_{L^2(\R^3)}\\
     &\leq c\,\big(\lVert \mathbf{v}_n - \mathbf{v}\rVert_{L^2(\R^3)}
        + \lVert \tilde{\mathbf{v}}_n - \tilde{\mathbf{v}}\rVert_{L^2(\R^3)}\big)
        \lVert \boldsymbol\varphi\rVert_{L^2(\Omega)}\\
     &\leq c\,\big(\lVert \mathbf{v}_n - \mathbf{v}\rVert_{L^2(\R^3)}
        + \lVert \tilde{\mathbf{v}}_n - \tilde{\mathbf{v}}\rVert_{L^2(\R^3)}\big).
  \end{align*}
  Hence, the left-hand side of \eqref{eqn:ehrling_2} converges and is 
  therefore bounded. This implies 
  \begin{align*}
  \sup_{\lVert \boldsymbol\varphi\rVert_{X(\Omega)}\leq 1}
          \Big(\int_{\Omega_{\eta_n}}\mathbf{v}_n\cdot\Tdelta\boldsymbol\varphi\;dx
        - \int_{\Omega_{\tilde \eta_n}} \tilde{\mathbf{v}}_n
          \cdot\Tdelta\boldsymbol\varphi\;dx\Big)
          \rightarrow 0.
  \end{align*}
  Since by Lemma \ref{lem:density_part_normal_trace} the space $X(\Omega)$ 
  is dense in $H_M(\Omega)$, the left-hand side of \eqref{eqn:ehrling_2} 
  converges to zero as well, a contradiction to $\varepsilon > 0$.
\end{proof}
%
\begin{proof}[of Lemma \ref{lem:kompaktheit}] 
Using the extension operator from Section \ref{section:div_fortsetzung},
we observe
\begin{align*}
  \int_I\int_{\R^3}&|\mathbf{u}_n - \mathbf{u}|^2\;dx\,dt
    + 2\int_I\int_M |\partial_t\eta_n - \partial_t\eta|^2\;dA\,dt\\
  &= \int_I\intdeltan \mathbf{u}_n\cdot(\mathbf{u}_n - \fdeltan\partial_t\eta_n)\;dx\,dt 
    + \int_I\intdeltan \mathbf{u}_n\cdot\fdeltan\partial_t\eta_n\;dx\,dt\\
    &\quad\quad+ 2\int_I\int_M |\partial_t\eta_n|^2\;dA\,dt
    + \int_I\int_{\R^3} |\mathbf{u}|^2\;dx\,dt
    + 2\int_I\int_M |\partial_t\eta|^2\;dA\,dt\\
    &\quad\quad- 2\int_I\int_{\R^3}\mathbf{u}_n\cdot \mathbf{u}\;dx\,dt
    - 4\int_I\int_M \partial_t\eta_n\,\partial_t\eta\;dA\,dt.   
\end{align*}
By the assumption \eqref{lem:L^2_komp_vor_konv}, it therefore suffices 
to show 
\begin{align}\label{eqn:komp_1}
  \int_I\intdeltan \mathbf{u}_n\cdot(\mathbf{u}_n 
    - \fdeltan\partial_t\eta_n)\;dx\,dt
  \rightarrow 
    \int_I\intdelta \mathbf{u}\cdot(\mathbf{u} 
      - \fdelta\partial_t\eta)\;dx\,dt 
\end{align}
and
\begin{align}\label{eqn:komp_2}\begin{aligned}
  \int_I\intdeltan \mathbf{u}_n\cdot\fdeltan\partial_t\eta_n\;dx\,dt
    &+ 2\int_I\int_M |\partial_t\eta_n|^2\;dA\,dt\\
  \rightarrow &
    \int_I\intdelta \mathbf{u}\cdot\fdelta\partial_t\eta\;dx\,dt
      + 2\int_I\int_M |\partial_t\eta|^2\;dA\,dt.
\end{aligned}\end{align}
To this end, we will use equation \eqref{lem:komp_equation}. Since 
the functions which are constant in time are not admissible in 
\eqref{lem:komp_equation}, we will construct an alternative integral 
identity. We choose $\tau \in C^\infty(\R)$ with $\tau(t)=1$ for $t\leq 0$, 
$\tau(t) = 0$ for $t\geq 1$ and  $\tau'\leq 0$. For $\varepsilon>0$ and 
$s\in \R$ we set $\tau_\varepsilon^s:\R\rightarrow\R$, 
$\tau_\varepsilon^s(t):=\tau(\varepsilon^{-1}(t-s))$. Then 
$\tau_\varepsilon^s\in C^\infty(\R)$, $\tau_\varepsilon^s\leq 1$ and 
$\tau_\varepsilon^s(t)$ converges to $\chi_{t\leq s}$ for $\varepsilon 
\rightarrow 0$ and all $t\in \R$. Moreover, $-\tau'$ is a smoothing kernel 
and $-(\tau_\epsilon^s)'(t)=-\epsilon^{-1}\tau'(\varepsilon^{-1}(t-s))$ 
the corresponding smoothing operator. Taking 
$(b,\boldsymbol\varphi)\in T^I_{\delta_n}$, we deduce that 
$(\tau_\varepsilon^s\,b,\tau_\varepsilon^s\,\boldsymbol\varphi)
\in T^I_{\delta_n}$. Hence, we can use 
$(\tau_\varepsilon^s\,b,\tau_\varepsilon^s\,\boldsymbol\varphi)$ in 
\eqref{lem:komp_equation} and pass to the limit $\varepsilon \rightarrow 0$.
For example, by the dominated convergence theorem and the properties of 
the smoothing operator, for almost all $s\in I$ we have
\begin{align*}
  -\int_I\intdeltan\mathbf{u}_n\cdot\partial_t(\tau_\varepsilon^s\,
    \boldsymbol\varphi)\;dx\,dt
  &=-\int_I\intdeltan\mathbf{u}_n\cdot\big((\tau_\varepsilon^s)'\,
    \boldsymbol\varphi\big)\;dx\,dt 
    -\int_I\intdeltan\mathbf{u}_n\cdot\big(\tau_\varepsilon^s\,
    \partial_t\boldsymbol\varphi\big)\;dx\,dt\\
  &=-\int_I\frac{1}{\varepsilon}\,\tau'\Big(\frac{t-s}{\varepsilon}\Big)
      \intdeltan\mathbf{u}_n\cdot\boldsymbol\varphi\;dx\,dt
    -\int_I\tau_\varepsilon^s\intdeltan\mathbf{u}_n\cdot
      \partial_t\boldsymbol\varphi\;dx\,dt\\
  &\rightarrow
    \intdeltans\mathbf{u}_n(s,\cdot)\cdot\boldsymbol\varphi(s,\cdot)\;dx
    -\int_0^s\intdeltan\mathbf{u}_n\cdot\partial_t\boldsymbol\varphi\;dx\,dt.
\end{align*}
Arguing similarly for the remaining terms in \eqref{lem:komp_equation}, 
we get, for almost all $s\in I$ and all $(b,\boldsymbol\varphi)\in 
T^I_{\delta_n}$, the identity 
\begin{align}
  -&\int_0^s\intdeltan\mathbf{u}_n\cdot\partial_t\boldsymbol\varphi\;dx\,dt 
    - \frac{1}{2}\int_0^s\int_M (\partial_t\eta_n)\,(\partial_t\delta_n)\,b\,
        \gamma(\delta_n)\;dA\,dt\notag\\
  &- 2\int_0^s\int_M \partial_t\eta_n\,\partial_t b\;dA\,dt 
    + 2\int_0^s K(\eta_n,b)\;dt
    +2\int_0^s\intdeltan \mathbf{D}\mathbf{u}_n:
      \mathbf{D}\boldsymbol\varphi\;dx\,dt\notag\\
  &\quad+ \frac{1}{2}\int_0^s\intdeltan(\mathbf{v}_n\cdot \nabla)
      \mathbf{u}_n\cdot\boldsymbol\varphi\;dx\,dt 
    - \frac{1}{2}\int_0^s\intdeltan(\mathbf{v}_n\cdot \nabla)
      \boldsymbol\varphi\cdot\mathbf{u}_n\;dx\,dt\label{lem:komp_equation_neu}\\
  &\qquad= \int_0^s\int_M g\,b\;dA
    + \int_0^s\intdeltan\mathbf{f}\cdot\boldsymbol\varphi\;dx\,dt
    +\int_{\Omega_{\delta_n(0)}}\hspace{-0.5cm}
      \mathbf{u}_0^n\cdot \boldsymbol\varphi(0,\cdot)\;dx 
    +2\int_M \eta_1^n\,b(0,\cdot)\;dA\notag\\
  &\qquad\quad- \intdeltans\mathbf{u}_n(s,\cdot)
      \cdot\boldsymbol\varphi(s,\cdot)\;dx
    - 2\int_M \partial_t\eta_n(s,\cdot)
      \,b(s,\cdot)\;dA.\notag
\end{align}
Note that the requirements $b(T,\cdot)=0$ and 
$\boldsymbol\varphi(T,\cdot)=0$ for the functions 
$(b,\boldsymbol\varphi)\in T^I_{\delta_n}$ can be omitted in 
\eqref{lem:komp_equation_neu}.

To show \eqref{eqn:komp_2}, we take $b\in H^2_0(M)$ and extend it 
constantly in time. By Lemma \ref{lem:fortsetzung_instationaer}, the 
couple $(b,\fdeltan b)$ is admissible in \eqref{lem:komp_equation_neu}
and we have the estimate
\begin{align}\label{eqn:lem-komp-1}\begin{aligned}
  \lVert \fdeltan b\rVert_{H^1(I,L^2(B_\alpha))\cap 
    C(\overline{I},H^1(B_\alpha))\cap L^3(I,W^{1,3}(B_\alpha))} 
  &\leq c\,\lVert b \rVert_{H^1(I,L^2(M))\cap L^3(I,H_0^2(M))}\\
  &\leq c\,\lVert b \rVert_{H_0^2(M)},
\end{aligned}\end{align}
where the constant $c$ is independent of $n$. Considering the integrands 
with respect to time in \eqref{lem:komp_equation_neu} with 
$\boldsymbol\varphi = \fdeltan b$, by H\"older's inequality we get
%
%
\begin{align*}
  \big\lVert \intdeltan\mathbf{u}_n\cdot\partial_t\fdeltan b\;dx\, 
      \big\rVert_{L^{12/11}(I)}^{12/11}
  &\leq\int_I \lVert \mathbf{u}_n(t,\cdot)\rVert_{L^2(\Omega_{\delta_n(t)})}^{12/11}
    \lVert (\partial_t\fdeltan b)(t,\cdot)\rVert_{L^2(\Omega_{\delta_n(t)})}^{12/11}\;dt\\
  &\leq \lVert \mathbf{u}_n \rVert_{L^{12/5}(I,L^2(\Omega_{\delta_n(t)}))}^{12/11}
    \lVert \partial_t\fdeltan b\rVert_{L^2(I,L^2(\Omega_{\delta_n(t)}))}^{12/11}\\
  &\leq c\,\lVert \mathbf{u}_n \rVert_{L^{\infty}(I,L^2(\Omega_{\delta_n(t)}))}^{12/11}
    \lVert \fdeltan b\rVert_{H^1(I,L^2(B_\alpha))}^{12/11}.
\end{align*}
As usual, the convective term needs a special treatment. By the 
Korn-type inequality, the sequences $\mathbf{u}_n$, 
$\mathbf{v}_n$ are bounded in $L^\infty(I;L^2(\Omega_{\delta_n(t)}))\cap L^2(I;
W^{1,r}(\Omega_{\delta_n(t)}))$ for any $1\leq r< 2$.  Using Sobolev's 
embedding, $W^{1,r}(\Omega_{\delta_n(t)}) \hookrightarrow 
L^{\tilde r}(\Omega_{\delta_n(t)})$ is continuous for all 
$1\leq \tilde r < \frac{3r}{3-r}$. Hence, for $2>r>\frac{510}{263}$, 
the embeddings $W^{1,r}(\Omega_{\delta_n(t)}) 
\hookrightarrow L^{\frac{170}{31}}(\Omega_{\delta_n(t)})$ and 
$W^{1,r}(\Omega_{\delta_n(t)}) \hookrightarrow L^{\frac{14}{3}}
(\Omega_{\delta_n(t)})$ are continuous. Using the H\"olderinterpolation (see
\cite[Korollar 2.10]{MR1409366} and \cite[Kapitel 1, Proposition 3.1]
{MR1230384}), we have
\begin{align*}
  L^\infty(I,L^2(\Omega_{\delta_n(t)}))\cap L^2(I,L^\frac{170}{31}
    (\Omega_{\delta_n(t)}))
  &\hookrightarrow L^\frac{12}{5}(I,L^\frac{17}{4}(\Omega_{\delta_n(t)})),\\
  L^\infty(I,L^2(\Omega_{\delta_n(t)}))\cap L^2(I,L^\frac{14}{3}(\Omega_{\delta_n(t)}))
  &\hookrightarrow L^\frac{24}{7}(I,L^3(\Omega_{\delta_n(t)})).
\end{align*}
In particular, since $\delta_n$ is bounded in $\widetilde{Y}^I$ and 
$\lVert \delta_n \rVert_{L^\infty(I\times M)} < \alpha$, the appearing 
constants can be chosen independently of $n\in \N$. By H\"older's 
inequality (notice $\frac{4}{17} + \frac{35}{68} + \frac{1}{4} = 1$) 
we therefore get, for the first part of the convective term,
\begin{align*}
  \big\lVert \intdeltan(\mathbf{v}_n&\cdot \nabla)\mathbf{u}_n
    \cdot\fdeltan b\;dx\big\rVert_{L^{12/11}(I)}^{12/11}\\
  &\leq \int_I \lVert \mathbf{v}_n(t,\cdot)\rVert_{L^{17/4}
      (\Omega_{\delta_n(t)})}^{12/11}\,
  \lVert \nabla\mathbf{u}_n(t,\cdot)\rVert_{L^{68/35}
      (\Omega_{\delta_n(t)})}^{12/11}\,
  \lVert \fdeltan b(t,\cdot)\rVert_{L^4(\Omega_{\delta_n(t)})}^{12/11}\;dt\\
  &\leq \lVert \mathbf{v}_n\rVert_{L^{12/5}(I;L^{17/4}
      (\Omega_{\delta_n(t)}))}^{12/11}\,
    \lVert \nabla\mathbf{u}_n\rVert_{L^2(I;L^{68/35}
      (\Omega_{\delta_n(t)}))}^{12/11}\,
  \lVert \fdeltan b\rVert_{L^\infty(I;
      L^4(\Omega_{\delta_n(t)}))}^{12/11}\\
  &\leq c( \lVert \mathbf{v}_n\rVert_{\widetilde{X}^I_{\delta_n}},
    \lVert \mathbf{u}_n\rVert_{\widetilde{X}^I_{\delta_n}})\,
    \lVert \fdeltan b\rVert_{L^\infty(I;
      L^4(\Omega_{\delta_n(t)}))}^{12/11}.
\end{align*}
For the second part of the convective term, we get similarly
\begin{align*}
  \big\lVert \intdeltan(\mathbf{v}_n&\cdot \nabla)\fdeltan b 
      \cdot\mathbf{u}_n\;dx\big\rVert_{L^{12/11}(I)}^{12/11}\\
  &\leq \int_I \lVert \mathbf{v}_n(t,\cdot)\rVert_{L^{3}(\Omega_{\delta_n(t)})}^{12/11}\,
    \lVert \nabla \fdeltan b(t,\cdot)\rVert_{L^3(\Omega_{\delta_n(t)})}^{12/11}\,
    \lVert\mathbf{u}_n(t,\cdot)\rVert_{L^{3}(\Omega_{\delta_n(t)})}^{12/11}\;dt\\
  &\leq \lVert \mathbf{v}_n\rVert_{L^{\frac{24}{7}}(I,L^{3}(\Omega_{\delta_n(t)}))}^{\frac{12}{11}}\,
    \lVert \nabla \fdeltan b\rVert_{L^3(I,L^3(\Omega_{\delta_n(t)}))}^{\frac{12}{11}}\,
    \lVert\mathbf{u}_n\rVert_{L^\frac{24}{7}(I,L^{3}(\Omega_{\delta_n(t)}))}^{\frac{12}{11}}\\
  &\leq c( \lVert \mathbf{v}_n\rVert_{\widetilde{X}^I_{\delta_n}}, 
      \lVert \mathbf{u}_n\rVert_{\widetilde{X}^I_{\delta_n}})\,
    \lVert \nabla \fdeltan b\rVert_{L^3(I,L^3(\Omega_{\delta_n(t)}))}^{\frac{12}{11}}.
\end{align*}
With similar arguments for the remaining terms and taking into account 
\eqref{eqn:lem-komp-1}, as well as $\partial_t b = 0$, the
integrands with respect to time in \eqref{lem:komp_equation_neu}, with 
$\varphi = \fdeltan b$, are bounded in $L^{12/11}(I)$ uniformly with 
respect to $n\in\N$ and $b\in \overline{B_{1}(0;{H_0^2(M)})}$, where 
$\overline{B_{1}(0;{H_0^2(M)})}$ denotes the closed unit ball in 
$H_0^2(M)$. By our assumptions, 
\begin{align*}
  \int_{\Omega_{\delta_n(0)}}\hspace{-0.5cm}\mathbf{u}_0^n\cdot 
      \fdeltan b(0,\cdot)\;dx
  +2\int_M \eta_1^n\,b(0,\cdot)\;dA
\end{align*}
is also bounded uniformly with respect to $n\in\N$ and $b\in 
\overline{B_{1}(0;{H_0^2(M)})}$. The identity
\eqref{lem:komp_equation_neu} for $\boldsymbol \varphi = \fdeltan b$ can 
therefore be written as 
\begin{align}\label{lem:kompakt_glng_hoelder}
  \int_0^s f_{b,n}(t)\;dt + g_{b,n}=  \intdeltans\mathbf{u}_n(s,\cdot)
      \cdot\fdeltan b(s,\cdot)\;dx
    + 2\int_M \partial_t\eta_n(s,\cdot)\,b(s,\cdot)\;dA,
\end{align}
where $f_{b,n}\in L^{12/11}(I)$, $g_{b,n}\in \R$ are uniformly bounded with 
respect to $n\in \N$ and $b\in \overline{B_{1}(0;{H_0^2(M)})}$. By H\"older's 
inequality, for $s_1\leq s_2\in \overline{I}$ we get
\begin{align*}
  \Big|\int_0^{s_1} f_{b,n}\;dt - \int_0^{s_2} f_{b,n}\;dt\Big|
  \leq \int_{s_1}^{s_2} |f_{b,n}|\;dt
  \leq \lVert f_{b,n}\rVert_{L^{12/11}(I)}\,|s_2-s_1|^\frac{1}{12}, 
\end{align*}
i.\,e.\ $\int_0^s f_{b,n}\;dt \in C^{0,\frac{1}{12}}(\overline{I})$ 
is uniformly bounded with respect to $n\in\N$ and 
$b\in \overline{B_{1}(0;{H_0^2(M)})}$. We set
\begin{align*}
  c_{b,n}(s)&:= \intdeltans\mathbf{u}_n(s,\cdot)\cdot\fdeltan b(s,\cdot)\;dx
    + 2\int_M \partial_t\eta_n(s,\cdot)\,b(s,\cdot)\;dA,\\
  c_b(s)&:= \intdeltas\mathbf{u}(s,\cdot)\cdot\fdelta b(s,\cdot)\;dx
    + 2\int_M \partial_t\eta(s,\cdot)\,b(s,\cdot)\;dA.
\end{align*}
Considering \eqref{lem:kompakt_glng_hoelder}, the sequence $c_{b,n}\in 
C^{0,\frac{1}{12}}(\overline{I})$ is uniformly bounded with respect to 
$n\in\N$ and $b\in \overline{B_{1}(0;{H_0^2(M)})}$. In particular, by 
Arzela-Ascoli's theorem, for any fixed $b\in H_0^2(M)$ a subsequence of 
$(c_{b,n})_{n\in \N}$ converges uniformly on $\overline{I}$. Since Lemma 
\ref{lem:fortsetzung_konvergenzen} and our assumptions imply that 
$(c_{b,n})_{n\in\N}$ converges to $c_b$ weakly in $L^{2}(I)$, the sequence 
$c_{b,n}$ converges to $c_{b}$ uniformly on $\overline{I}$. We will show 
that 
\begin{align*}
  h_n(s):= \sup_{\lVert b\rVert_{H_0^2(M)}\leq 1} \big|c_{b,n}(s) - c_b(s)\big|
\end{align*}
converges uniformly to $0$ on $\overline{I}$. To this end, we recall 
that the sequence $(\partial_t\eta_n,\mathbf{u}_n)_{n\in \N}$ is  
bounded in $L^\infty(I,L^2(M))\times L^\infty(I;L^2(\Omega_{\delta_n(t)}))$,
and chose a countable dense subset $I_0$ of $I$ such that the 
functions $\mathbf{u}_n(s,\cdot)\in L^2(B_\alpha)$ and 
$\partial_t\eta_n(s,\cdot)\in L^2(M)$ are bounded in their respective 
norms for all $n\in \N$ and $s\in I_0$. Using a diagonal sequence argument, 
we get a (not further denoted) subsequence such that for all $s\in I_0$ 
we have $\partial_t\eta_n(s,\cdot)\rightharpoonup \eta_s^\ast$ weakly in 
$L^2(M)$ and $\mathbf{u}_n(s,\cdot) \rightharpoonup \mathbf{u}_s^\ast$ weakly in 
$L^2(B_\alpha)$. Since the embedding $H^1(M)\hookrightarrow L^2(M)$ is 
compact, by Schauder's theorem the embedding through the dual operator 
$(L^2(M))^\ast\hookrightarrow (H^1(M))^\ast$ is also compact. Therefore, 
$\partial_t\eta_n(s,\cdot) \rightarrow \eta_s^\ast$ 
in $(H^1(M))^\ast$ and, analogously, $(\mathbf{u}_n(s,\cdot))_{n\in \N}
\rightarrow \mathbf{u}_s^\ast$ in $(H^1(B_\alpha))^\ast$. By the 
estimate
\begin{align*}
  \Big|\intdeltans &\mathbf{u}_n(s,\cdot)\cdot\fdeltan b(s,\cdot)\;dx 
    -\intdeltas\mathbf{u}_s^\ast\cdot\fdelta b(s,\cdot)\;dx 
    + 2\int_M \big(\partial_t\eta_n(s,\cdot) 
        - \eta_s^\ast\big)\,b(s,\cdot)\;dA\Big|\\
    & \leq  \Big|\intdeltans\big(\mathbf{u}_n(s,\cdot) 
        -\mathbf{u}_s^\ast\big)\cdot\fdeltan b(s,\cdot)\;dx\Big|
      + \Big|\intdeltas\mathbf{u}_s^\ast\cdot\big(\fdeltan b(s,\cdot) 
        - \fdelta b(s,\cdot)\big)\;dx\Big|\\
      &\quad\quad+ 2\, \lVert \partial_t\eta_n(s,\cdot) 
        - \eta_s^\ast \rVert_{(H^1(M))^\ast}\,\lVert b(s,\cdot)\rVert_{H^1(M)}\\
    & \leq  \lVert \mathbf{u}_n(s,\cdot)
        -\mathbf{u}_s^\ast\rVert_{(H^1(B_\alpha))^\ast}
          \lVert \fdeltan b(s,\cdot)\rVert_{H^1(B_\alpha)}\\
      &\quad\quad+ \lVert\mathbf{u}_s^\ast\rVert_{L^2(B_\alpha)}\lVert\fdeltan b(s,\cdot) 
        - \fdelta b(s,\cdot)\rVert_{L^2(B_\alpha)}\\
      &\quad\quad+ 2\,\lVert \partial_t\eta_n(s,\cdot) 
        - \eta_s^\ast \rVert_{(H^1(M))^\ast}\,\lVert b(s,\cdot)\rVert_{H^1(M)},
\end{align*}
the estimate \eqref{eqn:lem-komp-1} and Lemma 
\ref{lem:fortsetzung_konvergenzen_stationaer} imply the convergence 
of $(c_{b,n})_{n\in \N}$ on $s\in I_0$, uniformly with respect to 
$b\in \overline{B_{1}(0;{H_0^2(M)})}$. Since we already identified the 
weak limit of $(c_{b,n})_{n\in \N}$, we also have, for the original 
sequence, $c_{b,n}\rightarrow c_{b}$ on $I_0$ uniformly with 
respect to $b\in \overline{B_{1}(0;{H_0^2(M)})}$. By the uniform bound on 
$c_{b,n}\in C^{0,\frac{1}{12}}(\overline{I})$, for $s, s'\in \overline{I}$ 
we get
\begin{align*}
  \big| c_{b,n}(s) - c_{b,m}(s)\big| 
    &\leq \big| c_{b,n}(s) - c_{b,n}(s')\big| 
      + \big| c_{b,n}(s') - c_{b,m}(s')\big| 
      + \big| c_{b,m}(s') - c_{b,m}(s)\big|\\
    &\leq c\,|s-s'|^\frac{1}{12} + \big| c_{b,n}(s') - c_{b,m}(s')\big|,
\end{align*}
where the constant $c$ is independent of $n,m\in \N$, $s,s'
\in \overline{I}$ and $b\in \overline{B_{1}(0;{H_0^2(M)})}$. Let 
$\varepsilon>0$. Since $\overline{I}$ is compact and $I_0$ dense in 
$\overline{I}$, we get a finite subset $I_0^\varepsilon$ of $I_0$, such 
that for all $s\in \overline{I}$ there exists $s'\in I_0^\varepsilon$ 
with $c\,|s-s'|^\frac{1}{12}<\frac{\varepsilon}{2}$. Since 
$I_0^\varepsilon$ is finite, the convergence from above implies 
$\big| c_{b,n}(s') - c_{b,m}(s')\big|<\frac{\varepsilon}{2}$ for all 
$s'\in I_0^\varepsilon$ and $n,m\geq N$ uniformly with respect to 
$b\in \overline{B_{1}(0;{H_0^2(M)})}$. Therefore we have the uniform 
convergence of $c_{b,n}(s)$ to $c_{b,m}(s)$ independently of 
$b\in \overline{B_{1}(0;{H_0^2(M)})}$ and therefore the uniform 
convergence of $h_n$ to $0$ on $\overline{I}$.

Again, let $\varepsilon>0$. By the definitions of $c_{b,n}$, $c_b$ 
and their linearity with respect to $b$, the compatibility condition 
$\trdeltan\mathbf{u}_n= \partial_t \eta_n\,\boldsymbol\nu$,  Lemma 
\ref{moving_boundary_data} and the Ehrling-type Lemma we get
\begin{align*}
  \Big|\int_I c_{\partial_t\eta_n,n}(t) &- c_{\partial_t\eta_n}(t)\;dt\Big|\\
  &\leq \int_I \sup_{\lVert b\rVert_{L^2(M)}\leq 1}\big( c_{b,n}(t) 
      - c_{b}(t)\big)\;dt\\
  &\leq c\,\int_I \sup_{\lVert b\rVert_{L^2(M)}\leq 1}\Big(
    \intdeltans\,\mathbf{u}_n(s,\cdot)\cdot\fdeltan b(s,\cdot)\;dx
      - \intdeltas\,\mathbf{u}(s,\cdot)\cdot\fdelta b(s,\cdot)\;dx\\
    &\qquad + \int_M \left(\trdeltan\mathbf{u}_n(s,\cdot)\cdot\boldsymbol\nu
        -\trdelta\mathbf{u}(s,\cdot)\cdot\boldsymbol\nu\right)\,b(s,\cdot)\;dA\Big)\;ds\\
  &\leq c\,\int_I \varepsilon\,\rho\,\big(\lVert \mathbf{u}_n(t,\cdot)\rVert_{V_2(\Omega_{\delta_n(t)})}
            + \lVert \mathbf{u}(t,\cdot)\rVert_{V_2(\Omega_{\delta(t)})}\big)\;dt\\
            &\quad\quad+ c\,\int_I \sup_{\lVert b \rVert_{H_0^2(M)}\leq 1}
              \big( c_{b,n}(t) - c_{b}(t)\big)\;dt\\
  &\leq \varepsilon\,c\,\big(\lVert \mathbf{u}_n\rVert_{L^2(I,V_2(\Omega_{\delta_n}(t)))}
            + \lVert \mathbf{u}\rVert_{L^2(I,V_2(\Omega_{\delta}(t)))}\big)
            + c\,\int_I h_n(t)\;dt.
\end{align*}
%
By the uniform convergence of $h_n$ to $0$, this implies 
$|\int_I c_{\partial_t\eta_n,n} - c_{\partial_t\eta_n}\;dt| \rightarrow 0$. 
We now consider the inequality
\begin{align*}
  \Big| \int_I\intdeltan &\mathbf{u}_n\cdot\fdeltan\partial_t\eta_n\;dx\,dt
    - \int_I\intdelta \mathbf{u}\cdot\fdelta\partial_t\eta\;dx\,dt\\
    &\qquad+ 2\int_I\int_M |\partial_t\eta_n|^2\;dA\,dt
    - 2\int_I\int_M |\partial_t\eta|^2\;dA\,dt\Big|\\
    &= \Big|\Big(\int_I\intdeltan \mathbf{u}_n\cdot\fdeltan\partial_t\eta_n\;dx 
        + 2\int_M |\partial_t\eta_n|^2\;dA\;dt\Big)\\
        &\quad\quad - \Big( \int_I \intdelta\mathbf{u}\cdot\fdelta \partial_t\eta_n\;dx
        + 2\int_M \partial_t\eta\,\partial_t\eta_n\;dA\,dt\Big)\\
        &\quad\quad + \int_I \intdelta\mathbf{u}\cdot\big(\fdelta \partial_t\eta_n 
            - \fdelta \partial_t\eta\big)\;dx\;dt
            + 2 \int_I\int_M \partial_t\eta\,
          \big(\partial_t\eta_n - \partial_t\eta)\;dA\,dt\Big|\\
    & \leq \Big|\int_I c_{\partial_t\eta_n,n} - c_{\partial_t\eta_n}\;dt\Big|
      + \Big|\int_I \intdelta\mathbf{u}\cdot\big(\fdelta \partial_t\eta_n 
          - \fdelta \partial_t\eta\big)\;dx\;dt\Big|\\
      &\quad\quad+ 2\,\Big| \int_I\int_M 
        \partial_t\eta\,\big(\partial_t\eta_n - \partial_t\eta)\;dA\,dt\Big|.
\end{align*}
The weak convergence of $\partial_t \eta_n$, Corollary 
\ref{kor:fortsetzung_instationaer_weak} and the convergence from above 
therefore imply \eqref{eqn:komp_2}.

To show \eqref{eqn:komp_1}, we argue similarly but have to use the 
vanishing boundary values of $\mathbf{u}$ to get a uniformly admissible 
test function. Since the sequences $(\mathbf{u}_n)_{n\in \N}\in L^\infty(I; 
L^2(\Omega_{\delta_n(t)}))$ and $(\partial_t\eta_n)_{n\in \N}\in L^2(I;L^2(M))$ 
are uniformly bounded and satisfy $\operatorname{div}\mathbf{u}_n= 0$ and 
$\trdeltan\mathbf{u}_n = \partial_t \eta_n\,\boldsymbol\nu$, Lemma 
\ref{lem:fortsetzung_stationaer_schwach} implies that 
$\mathbf{u}_n(t,\cdot) - (\fetan\partial_t\eta_n)(t,\cdot)$ 
is uniformly bounded in  \linebreak $H_M(\Omega_{\delta_n(t)})$ for all $n\in\N$ 
and almost all $t\in I$. Let $\varepsilon>0$. By Lemma 
\ref{lem:density_for_compactness}, there exist $\lambda_\epsilon >0$, 
$c_0>0$ and $\Psi_{t,n}\in H_M(\Omega_{\delta_n(t)})$ with
$supp\;\Psi_{t,n} \subset \Omega_{\delta_n(t)-\lambda_\epsilon}$, such 
that for all $n\in\N$ and almost all $t\in I$ it holds that $\lVert \Psi_{t,n} 
\rVert_{H_M(\Omega_{\delta_n(t)})}\leq c_0$ and
\begin{align}\label{eqn:komp_3}
  \lVert \mathbf{u}_n(t,\cdot) - (\fetan\partial_t\eta_n)(t,\cdot)
    - \Psi_{t,n}\rVert_{(H^\frac{1}{4}(\R^3))^\ast}<\varepsilon.
\end{align}
Similarly to Proposition \ref{prop:reg_rand:anfangswert}, we approximate 
the displacement $\delta$ ``uniformly'' from the inside by 
$\aussenrandlambda\delta\in C^4(\overline{I}\times M)$, i.\,e.\ 
$\aussenrandlambda\delta < \delta_n$ for $\lambda$ small enough and $n$ 
big enough, see \cite[Lemma 2.65]{phdeberlein} for the details. Then 
$\aussenrandlambda\delta$ and $\delta_n$ converge uniformly to $\delta$ 
on $\overline{I}\times M$. Hence, for a (in the following fixed) 
$\lambda>0$ small enough and all $n\in \N$ big enough, we have 
$\delta_n - \lambda_\varepsilon < \aussenrandlambda\delta \leq \delta_n$ 
and $\lVert \aussenrandlambda \delta \rVert_{L^\infty(I\times M)}<\kappa$.
Therefore, we get $supp\;\Psi_{t,n} \subset \Omega_{\delta_n(t)-\lambda_\varepsilon}
\subset \Omega_{\aussenrandlambda\delta(t)} \subset \Omega_{\delta_n(t)}$ 
for almost all $t\in I$.

We extend a function $\boldsymbol\varphi\in X(\Omega)$ constantly in time 
and transform it, using the Piola-transform, to 
$\Omega^I_{\aussenrandlambda\delta}$. Since the Piola-transform preserves 
vanishing boundary values, we can extend $\Tlambda\boldsymbol\varphi$ 
further continuously by $\boldsymbol 0$ to $I\times B_\kappa$. 
By our uniform approximation from the inside, we have $supp 
\Tlambda\boldsymbol\varphi \subset \Omega^I_{\aussenrandlambda\delta} 
\subset \Omega^I_{\delta_n}$ and the estimate
\begin{align}\label{eqn:komp_reference_embedding}
  \lVert \Tlambda\boldsymbol\varphi \rVert_{H^1(I;L^2(\Omega_{\delta_n(t)}))
    \cap L^{3}(I;W^{1,3}(\Omega_{\delta_n(t)}))
      \cap L^\infty(I;L^4(\Omega_{\delta_n(t)}))}
  \leq c\, \lVert\boldsymbol\varphi\rVert_{X(\Omega)},
\end{align}
where the constant $c$ depends, among others, on $\lambda$, but not on $n$. 
In particular, the couple $(0,\Tlambda\boldsymbol\varphi)$ is admissible 
for the identity \eqref{lem:komp_equation_neu}. 
We can now argue analogously to the proof of \eqref{eqn:komp_2}. Using 
\eqref{eqn:komp_reference_embedding} instead of \eqref{eqn:lem-komp-1}, 
the integrands with respect to time in \eqref{lem:komp_equation_neu}, 
tested by $(0,\Tlambda\boldsymbol\varphi)$, are bounded in $L^{12/11}(I)$ 
uniformly with respect to $n\in \N$ and $\boldsymbol\varphi \in 
\overline{B_{1}(0;{X(\Omega)})}$, where $\overline{B_{1}(0;{X(\Omega)})}$ 
denotes the closed unit ball in $X(\Omega)$. Also the integral over the 
initial data $\mathbf{u}_0^n$ is bounded independently of $n\in \N$ and 
$\boldsymbol\varphi \in \overline{B_{1}(0;{X(\Omega)})}$. The identity  
\eqref{lem:komp_equation_neu}, tested by $(0,\Tlambda\boldsymbol\varphi)$,
can therefore be written as
\begin{align}\label{lem:kompakt_glng_hoelder2}
  \int_0^s f_{\boldsymbol\varphi,\lambda,n}(t)\;dt 
      + g_{\boldsymbol\varphi,\lambda,n}
    =  \intdeltans\mathbf{u}_n(s,\cdot)\cdot
      \Tlambda\boldsymbol\varphi(s,\cdot)\;dx,
\end{align}
where $f_{\boldsymbol\varphi,\lambda,n} \in L^{12/11}(I)$
and  $g_{\boldsymbol\varphi,\lambda,n}\in \R$ are uniformly bounded with 
respect to $n\in \N$ and $\boldsymbol\varphi \in \overline{B_{1}(0;{X(\Omega)})}$. 
Hence, we get $\int_0^s f_{\boldsymbol\varphi,\lambda,n}(t)\;dt \in 
C^{0,\frac{1}{12}}(\overline{I})$. We set
\begin{align*}
  c^\lambda_{\boldsymbol\varphi,n}(s)&:=
    \intdeltans\mathbf{u}_n(s,\cdot)\cdot
        \Tlambda\boldsymbol\varphi(s,\cdot)\;dx,\\
  c^\lambda_{\boldsymbol\varphi}(s)&:=
    \intdeltas\mathbf{u}(s,\cdot)\cdot
        \Tlambda\boldsymbol\varphi(s,\cdot)\;dx
\end{align*}
and deduce that $c^\lambda_{\boldsymbol\varphi,n}$ in 
$C^{0,\frac{1}{12}}(\overline{I})$ is bounded independently of $n\in\N$ 
and $\boldsymbol\varphi \in \overline{B_{1}(0;{X(\Omega)})}$. For 
$\boldsymbol\varphi$ fixed, by Arzela-Ascoli's theorem and the weak 
convergence of $\mathbf{u}_n$ we get the uniform convergence of 
$c^\lambda_{\boldsymbol\varphi,n}$ to $c^\lambda_{\boldsymbol\varphi}$ on 
$\overline{I}$. With the same arguments as in the proof of 
\eqref{eqn:komp_2} we get that
\begin{align*}
  h_n^\lambda(s):= \sup_{\lVert \boldsymbol\varphi\rVert_X\leq 1}
  \Big| c^\lambda_{\boldsymbol\varphi,n}(s) 
    - c^\lambda_{\boldsymbol\varphi}(s)\Big|
\end{align*}
converges to $0$ uniformly in $\overline{I}$. Since the Piola-transform 
$\Tlambdat$ is an isomorphism between $H_M(\Omega)$ and 
$H_M(\Omega_{\aussenrandlambda\delta(t)})$, the property 
$supp\;\Psi_{t,n} \subset \Omega_{\aussenrandlambda\delta(t)}$ implies 
$\Psi_{t,n} \in H_M(\Omega_{\aussenrandlambda\delta(t)})$ and therefore
$\invTlambda\Psi_{t,n}\in H_M(\Omega)$. Let  $\widetilde \varepsilon>0$. 
By the definition of $c^\lambda_{\boldsymbol\varphi,n}$, 
$c^\lambda_{\boldsymbol\varphi}$ (particularly their linearity with 
respect to $\boldsymbol\varphi$) and Lemma \ref{lem:ehrling}, we get
%
%
\begin{align*}
  \Big|\int_I \intdeltan\mathbf{u}_n\cdot\Psi_{t,n}\;dx
      &- \intdelta\mathbf{u}\cdot\Psi_{t,n}\;dx\;dt\,\Big|\\
  &= \Big|\int_I c^\lambda_{\invTlambda\Psi_{t,n},n}(t) 
    - c^\lambda_{\invTlambda\Psi_{t,n}}(t)\;dt\,\Big|\\
  &\leq c\, \int_I \sup_{\lVert \boldsymbol\varphi\rVert_{H_M(\Omega)}\leq 1}
      \big(c^\lambda_{\boldsymbol\varphi,n}(t) - c^\lambda_{\boldsymbol\varphi}(t)\big)\;dt\\
  &\leq c \int_I \widetilde\varepsilon\,\big(\lVert \mathbf{u}_n(t,\cdot)
      \rVert_{V_2(\Omega_{\delta_n(t)})} + \lVert \mathbf{u}(t,\cdot)
          \rVert_{V_2(\Omega_{\delta(t)})}\big)\;dt\\
      &\quad\quad+ c\,\int_I \sup_{\lVert \boldsymbol\varphi\rVert_{X(\Omega)}\leq 1} 
        c^\lambda_{\boldsymbol\varphi,n}(t) - c^\lambda_{\boldsymbol\varphi}(t)\;dt\\
  &\leq \widetilde\varepsilon\,c\,\big(\lVert \mathbf{u}_n\rVert_{L^2(I,V_2(\Omega_{\delta_n}(t)))}
            + \lVert \mathbf{u}\rVert_{L^2(I,V_2(\Omega_{\delta}(t)))}\big)
            + c\,\int_I h^\lambda_n(t)\;dt.
\end{align*}
Using the uniform convergence of $h^\lambda_n$ and the bound on 
$\mathbf{u}_n$ in $X^I_{\delta_n}$, we therefore deduce 
\begin{align}\label{eqn:komp_4}
  \Big|\int_I \intdeltan\mathbf{u}_n\cdot\Psi_{t,n}\;dx
      - \intdelta\mathbf{u}\cdot\Psi_{t,n}\;dx\;dt\,\Big|
      \rightarrow 0.
\end{align}
Taking \eqref{eqn:komp_3} into account, we get
\begin{align*}
  \Big|\int_I\intdeltan \mathbf{u}_n&\cdot(\mathbf{u}_n 
      - \fdeltan\partial_t\eta_n)\;dx\,dt
    - \int_I\intdelta \mathbf{u}\cdot(\mathbf{u} 
      - \fdelta\partial_t\eta)\;dx\,dt\,\Big|\\
  &\leq \Big|\int_I\intdeltan \mathbf{u}\cdot(\mathbf{u}_n 
      - \fdeltan\partial_t\eta_n)\;dx\,dt
    - \int_I\intdelta \mathbf{u}\cdot(\mathbf{u} 
      - \fdelta\partial_t\eta)\;dx\,dt\,\Big|\\
    &\quad\quad +\Big|\int_I\intdeltan \mathbf{u}_n\cdot\Psi_{t,n}\;dx\,dt
    - \int_I\intdeltan \mathbf{u}\cdot\Psi_{t,n}\;dx\,dt\,\Big|\\
    &\quad\quad + \Big|\int_I\intdeltan \mathbf{u}_n\cdot(\mathbf{u}_n 
      - \fdeltan\partial_t\eta_n - \Psi_{t,n})\;dx\,dt\,\Big|\\
    &\quad\quad + \Big|\int_I\intdeltan \mathbf{u}\cdot(\mathbf{u}_n 
      - \fdeltan\partial_t\eta_n - \Psi_{t,n})\;dx\,dt\,\Big|\\
  &\leq \Big|\int_I\intdeltan \mathbf{u}\cdot(\mathbf{u}_n 
      - \fdeltan\partial_t\eta_n)\;dx\,dt
    - \int_I\intdelta \mathbf{u}\cdot(\mathbf{u} 
      - \fdelta\partial_t\eta)\;dx\,dt\,\Big|\\
    &\quad\quad +\Big|\int_I\intdeltan \mathbf{u}_n\cdot\Psi_{t,n}\;dx\,dt
    - \int_I\intdeltan \mathbf{u}\cdot\Psi_{t,n}\;dx\,dt\,\Big|\\
    &\quad\quad + c\,\lVert \mathbf{u}_n \rVert_{L^2(I;H^{\frac{1}{4}}(\R^3))}
      \lVert\mathbf{u}_n - \fdeltan\partial_t\eta_n - \Psi_{t,n}
        \rVert_{L^\infty(I;(H^\frac{1}{4}(\R^3))^\ast)}\\
    &\quad\quad + c\,\lVert \mathbf{u}\rVert_{L^2(I;H^{\frac{1}{4}}(\R^3))}
      \lVert \mathbf{u}_n - \fdeltan\partial_t\eta_n - \Psi_{t,n}
        \rVert_{L^\infty(I;(H^\frac{1}{4}(\R^3))^\ast)}\\
  &< \Big|\int_I\intdeltan \mathbf{u}\cdot(\mathbf{u}_n 
      - \fdeltan\partial_t\eta_n)\;dx\,dt
    - \int_I\intdelta \mathbf{u}\cdot(\mathbf{u} 
      - \fdelta\partial_t\eta)\;dx\,dt\,\Big|\\
    &\quad\quad +\Big|\int_I\intdeltan \mathbf{u}_n\cdot\Psi_{t,n}\;dx\,dt
    - \int_I\intdeltan \mathbf{u}\cdot\Psi_{t,n}\;dx\,dt\,\Big|\\
    &\quad\quad + \varepsilon\,c\,\lVert \mathbf{u}_n 
      \rVert_{L^2(I;H^{\frac{1}{4}}(\R^3))} + \varepsilon\,c\,
        \lVert \mathbf{u} \rVert_{L^2(I;H^{\frac{1}{4}}(\R^3))}.
\end{align*}
Since the extension by zero is continuous in $H^\frac{1}{4}$ (see 
\cite[Lemma A.3]{phdeberlein}), the convergences from \eqref{eqn:komp_4} 
and \eqref{lem:L^2_komp_vor_konv}, together with Lemma 
\ref{lem:fortsetzung_konvergenzen}, therefore imply \eqref{eqn:komp_1}, 
which finishes the proof.
\end{proof}
\section*{Acknowledgment}
This work has been partially supported by the DFG, namely within the
project C2 of the SFB/TR "Geometric Partial Differential Equations".


\end{document}